\documentclass[12pt]{amsart}
\usepackage{geometry}
\usepackage[english]{babel}
\usepackage{amsmath, amssymb,mathrsfs}
\usepackage{enumitem}
\usepackage{amsmath}
\usepackage{graphicx}
\usepackage{amsthm}
\usepackage{verbatim}
\usepackage[colorlinks=true]{hyperref}
\usepackage{bm}
\usepackage{color}
\usepackage{autobreak}
\usepackage[all]{xy}
\usepackage{tikz-cd}
\usepackage{graphicx}
\numberwithin{equation}{section}
\newtheorem{theorem}{Theorem}[section]

\newtheorem{lemma}[theorem]{Lemma}
\newtheorem{proposition}[theorem]{Proposition}
\newtheorem{definition}[theorem]{Definition}
\newtheorem{remark}[theorem]{Remark}

\newtheorem{condition}[theorem]{Condition}
\newcommand{\St}{\mathrm{St}}
\newcommand{\Rep}{\mathrm{Rep}}
\newcommand{\IM}{\mathrm{IM}}
\newcommand{\GL}{\mathrm{GL}}
\newcommand{\Qp}{\mathbb{Q}_p}
\newcommand{\triv}{\mathrm{triv}}
\newcommand{\Mod}{\mathrm{Mod}}
\newcommand{\Ind}{\mathrm{Ind}}
\newcommand{\Exp}{\mathrm{Exp}}
\newcommand{\op}{\mathrm{op}}
\newcommand{\Res}{\mathrm{Res}}
\newcommand{\Hom}{\mathrm{Hom}}
\newcommand{\IC}{\mathrm{IC}}
\newcommand{\Ad}{\mathrm{Ad}}
\newcommand{\MJ}{\mathcal{J}}
\newcommand{\C}{\mathbb{C}}

\begin{document}

\title{Twisted Kazhdan-Lusztig conjecture for $p$-adic general linear group}
\author{CHAI Yuan}
\maketitle
\begin{abstract}
We use enhanced Langlands parameters to obtain a classification for irreducible representations of twisted $p$-adic general linear groups in unramified principal series.  We give the definition of standard representations and prove the twisted Kazhdan-Lusztig conjecture for the multiplicities in the Grothendieck group. We mainly follow Lusztig's work in the connected case using graded Hecke algebra. We show that the parametrization is compatible with the Whittaker-normalized one.
\end{abstract}
\tableofcontents
\section{Introduction}
Let $F=\Qp$, $\mathcal{G}=\GL_n(F)$, $\mathscr{O}$ be the ring of integers of $F$ and $\mathscr{P}$ be the maximal ideal in $\mathscr{O}$. Let $\mathcal{B}$ be the Borel subgroup of upper triangular matrices of $\mathcal{G}$ and $\mathcal{T}$ be the Levi of $\mathcal{B}$ which is a maximal torus of $\mathcal{G}$.

Let $\mathrm{Rep}(\mathcal{G})$ denote the category of smooth $\mathcal{G}$-representations. 
Let $X_{nr}(\mathcal{T})$ be the group of unramified characters $\mathcal{T}\to \mathbb{C}^\times$, i.e., characters of the form $\rho=\prod_{i=1}^n|\cdot|^{s_i}_F$ with $s_i\in \mathbb{C}$. Let $\mathrm{Rep}(\mathcal{G})^{[\mathcal{T},1_{\triv}]}$ be the full subcategory of $\mathrm{Rep}(\mathcal{G})$ consisting of representations  whose irreducible subquotients are subquotients of the normalized parabolic induction $i_{\mathcal{B}}^{\mathcal{G}}(\rho)$ for some  $\rho \in X_{nr}(\mathcal{T})$ \cite{zbMATH03687666}.

Denote by $G=\GL_n(\mathbb{C})$ the complex dual group of $\mathcal{G}$.  The $L$-group of  $\mathcal{G}$ is $^L\mathcal{G}=G\times W_F$, where $W_F$ is the Weil group of $F$. 

A  Langlands parameter for $\mathcal{G}$ is a continuous group homomorphism
\begin{equation*}
	\phi: W_F\times \mathrm{SL}_2(\mathbb{C}) \to {^L\mathcal{G}} 
\end{equation*}
such that it respects the projections to $W_F$ for both  $W_F\times \mathrm{SL}_2(\mathbb{C})$ and ${^L\mathcal{G}}$, its restriction to $\mathrm{SL}_2(\mathbb{C})$ is algebraic and  the image of $W_F$ consists of semisimple elements in ${^L\mathcal{G}}$. Let $P({^L\mathcal{G}})$ denote the set of Langlands parameters of $\mathcal{G}$.  Two Langlands parameters are equivalent if they are conjugate under $G$, and we denote the set of equivalence classes by $\Phi(\mathcal{G})$.

For $\phi\in P({^L\mathcal{G}})$, we associate an infinitesimal parameter $\lambda_\phi$ given by 
\begin{equation*}
	\lambda_\phi(\omega)=\phi(\omega,\begin{pmatrix}
		|\omega|^{1/2} & \\ &|\omega|^{-1/2}
	\end{pmatrix}), \quad \omega\in W_F
\end{equation*}
Fixing an infinitesimal parameter $\lambda$, we define
\begin{equation*}
	P(\lambda,{^L\mathcal{G}}):=\{\phi\in P({^L\mathcal{G}})|\lambda_\phi=\lambda\}
\end{equation*}

and \begin{equation*}
	\Xi(\lambda,{^L\mathcal{G}}):=\{(\phi,\rho)|\phi\in P(\lambda,{^L\mathcal{G}})/Z_{{G}}(\lambda),\rho\in \mathrm{Irr}(\pi_0(Z_{{G}}(\phi)))\}
\end{equation*}

For $G=\GL_n(\mathbb{C})$, we have $\pi_0(Z_{{G}}(\phi))=1$, so $\rho=1$. 

Varying over all conjugacy classes of infinitesimal parameters of $\mathcal{G}$, we obtain 
\begin{equation*}
	\Xi({^L\mathcal{G}}):=\{(\phi,\rho)|\phi\in \Phi(\mathcal{G}), \rho\in \mathrm{Irr}(\pi_0(Z_{{G}}(\phi)))\}
\end{equation*}

By the local Langlands correspondence, we can parametrize the irreducible representations $M(\xi)$ in $\mathrm{Rep}(\mathcal{G})^{[\mathcal{T},1_{\triv}]}$ of $\mathcal{G}$ by $\xi\in \Xi({^L\mathcal{G}})$ whose Langlands parameter $\phi(\xi)\in \Phi(\mathcal{G})$  is trivial on the inertial group $I_F$.

For $\xi \in \Xi(\lambda,{^L\mathcal{G}})$, we have an irreducible representation $M(\xi)$ and a standard representation $E(\xi)$, which contains $M(\xi)$ as the Langlands quotient.

Let $\hat{\mathfrak{g}}$ denote the Lie algebra of $G$. Define  the Vogan moduli space \begin{equation*}
	V(\lambda,{^L\mathcal{G}}):=\{x\in \hat{\mathfrak{g}}|\mathrm{Ad}(\lambda(\mathrm{Fr}))x=px\}
\end{equation*}
where $\mathrm{Fr}$ is the Frobenius element.

Denote by $\mathrm{Per}(\lambda,{^L\mathcal{G}})$ the category of $Z_{{G}}(\lambda)$-equivariant perverse sheaves on $V(\lambda,{^L\mathcal{G}})$. The simple objects  $\mathrm{IC}(\gamma)$ in $\mathrm{Per}(\lambda,{^L\mathcal{G}})$ are parametrized by $\gamma\in \Xi(\lambda,{^L\mathcal{G}})$. The Kazhdan-Lusztig conjecture, proved in \cite{KazhLus}\cite{Ginz}, states that:
\begin{align}
	\mathrm{mult}(M(\xi),E(\zeta))=\sum_i\mathrm{dim}H^i(i^!_{y(\zeta)}\mathrm{IC}(\xi))
\end{align}
where $y=d_{\phi(\zeta)}\begin{pmatrix}
	0 &1\\
	0& 0
\end{pmatrix}$.

We denote by $\gamma$ the automorphism of $\mathcal{G}$ defined for each $g$ by $\gamma(g)=J(^tg^{-1})J^{-1}$ where $J$ is the antidiagonal matrix with entries $(-1, 1,\cdots,(-1)^n)$.  We define the disconnected group $\mathcal{G}^+:=\mathcal{G}\rtimes\Gamma$, where $\Gamma$ is the group of two elements generated by $\gamma$.
Define the category $\mathrm{Rep}(\mathcal{G}^+)^{[\mathcal{T},1_{\triv}]}$
of $\mathcal{G}^+$ representations whose restriction to $\mathcal{G}$ belongs to $\mathrm{Rep}(\mathcal{G})^{[\mathcal{T},1_{\triv}]}$. For $\pi\in\mathrm{Rep}(\mathcal{G})$, define $\gamma^*\pi$ by $(\gamma^*{\pi}) (g):= \pi(\gamma(g))$. For an irreducible  representation $\pi\in \mathrm{Rep}(\mathcal{G})$, if $\pi\simeq\gamma^*\pi$, there are two extensions; otherwise, there is only one extension. 

Define the complex dual group $G^+:=G\rtimes \hat{\Gamma}$, where $\hat{\Gamma} = \langle\hat{\gamma}\rangle$  for $\hat{\gamma}(g)=J(^tg^{-1})J^{-1}$.  Define \begin{equation*}
	\Xi(\lambda,{^L\mathcal{G}^+}):=\{(\phi,\rho)|\phi\in P(\lambda,{^L\mathcal{G}})/Z_{{G^+}}(\lambda),\rho\in \mathrm{Irr}(\pi_0(Z_{{G^+}}(\phi)))\}
\end{equation*}

By considering  all conjugacy classes of infinitesimal parameters of $\mathcal{G}$, we obtain $\Xi({^L\mathcal{G}^+})$.

From \cite[Theorem 7.3]{RepBernZel}, if $\pi\in \mathrm{Rep}(\mathcal{G})$ is irreducible, then ${\gamma^*(\pi)}$ is  the contragredient representation of $\pi$. Taking $\pi= M(\phi,1)$, we have $\gamma^* M(\phi,1)\simeq M(\hat{\gamma}(\phi),1)$ \cite{zbMATH03687666}. There is a bijection between irreducible representations in $\mathrm{Rep}(\mathcal{G}^+)^{[\mathcal{T},1_{\triv}]}$ and a subset of $\Xi({^L\mathcal{G}^+})$ whose Langlands parameters are trivial on $I_F$. Using the  geometry of graded Hecke algebras, we obtain the Langlands classification $M^+(\phi,\rho)\in \mathrm{Irr}(\mathrm{Rep}(\mathcal{G}^+)^{[\mathcal{T},1_{\triv}]})$ where  $(\phi,\rho)\in  \Xi({^L\mathcal{G}^+})$, and $\phi(I_F)=1$.

We define the standard representation in $\mathrm{Rep}(\mathcal{G^+})^{[\mathcal{T},1_{\triv}]}$ as follows:

If $M(\phi,1)$ is not $\gamma$-invariant, we define  \begin{equation*}
	E^+(\phi,1):=\mathrm{ind}_{\mathcal{G}}^{\mathcal{G}^+} E(\phi,1)
\end{equation*}

If $M(\phi,1)$ is  $\gamma$-invariant, there exists a $\gamma$-stable parabolic subgroup $\mathcal{P}=\mathcal{Q}\mathcal{U}$ and a $\gamma$-invariant irreducible representation \eqref{InvRep} \begin{align}
	M(\phi_{\mathcal{Q}},1)\simeq \pi_1\otimes\cdots\otimes\pi^0\otimes\cdots\otimes\pi_k
\end{align} 
where
\begin{equation*}
	\begin{aligned}
		\gamma^*(\pi_i)&\simeq \pi_{k+1-i}\\
		\gamma^*(\pi^0)&\simeq {\pi}^0
	\end{aligned}
\end{equation*}

and $\pi^0$ is tempered. 

Since $\gamma$ normalizes $\mathcal{Q}$,  we define $\mathcal{Q}^+=\mathcal{Q}\rtimes \Gamma$, $\mathcal{P}^+=\mathcal{P}\rtimes\Gamma$. Hence we use the same method of Langlands classification of $\mathrm{Rep}(\mathcal{G}^+)^{[\mathcal{T},1_{\triv}]}$ to get the irreducible representation  $M^+(\phi_{\mathcal{Q}},\rho)$ for $\mathcal{Q}^+$.
We define the  standard representation:  
\begin{equation*}
	E^+(\phi,\rho):= i_{\mathcal{P}^+} M^+(\phi_{\mathcal{Q}},\rho)
\end{equation*}
where $i_{\mathcal{P}^+}$ is the normalized parabolic induction \cite[1.8]{zbMATH03639915}.

Let $\mathrm{Per}(\lambda,{^L\mathcal{G}}^+)$ be the category of $Z_{{G}^+}(\lambda)$-equivariant perverse sheaves on $V(\lambda,{^L\mathcal{G}})$. The simple objects $\mathrm{IC}(\zeta^+)$  are  parametrized by $\zeta^+=(\phi(\zeta^+),\rho(\zeta^+))\in \Xi(\lambda,{^L\mathcal{G}^+})$.

Our first main result is the twisted Kazhdan-Lusztig conjecture:
\begin{theorem}[Theorem \ref{MainTh2}]
	Taking $\xi^+,\zeta^+\in \Xi(\lambda,{^L\mathcal{G}^+})$, we have irreducible representation $M^+(\phi(\xi^+),\rho(\xi^+))$ and standard representation $E^+(\phi(\zeta^+),\rho(\zeta^+))$. Then the multiplicity in the Grothendieck group is given by  
	\begin{equation*}
		m(M^+(\phi(\xi^+),\rho(\xi^+)),E^+(\phi(\zeta^+),\rho(\zeta^+)))=\sum_k\mathrm{dim}H^k(i_{y(\zeta^+)}^!\mathrm{IC}(\xi^+)))^{\rho(\zeta^+)}
	\end{equation*}
	where $y(\zeta^+)=d_{\phi(\zeta^+)}\begin{pmatrix}
		0 &1\\
		0& 0
	\end{pmatrix}$.
\end{theorem}

In fact,  if $M(\phi,1)$ is $\gamma$-invariant, we can also use the Whittaker normalization introduced by Arthur to obtain another canonical extension  $(M(\phi,1))^+_W$ of $\mathcal{G}^+=\mathcal{G}\rtimes\Gamma$.

Recall that we have a $\gamma$-invariant irreducible representation $M(\phi_{\mathcal{Q}},1)$. Then we have the standard representation \begin{align}
E(\phi,1)=i_{\mathcal{P}}M(\phi_{\mathcal{Q}},1)
\end{align}
which contains $M(\phi,1)$ as its unique quotient.

Let $\Pi$ be the set of simple roots determined by the upper triangular Borel subgroup. The abelian group $\prod_{\alpha\in\Pi}\mathcal{U}_\alpha$ is a quotient of $\mathcal{U}$, where $\mathcal{U}_\alpha$ is the root subgroup for $\alpha$. Given an additive character $\varphi$ of $F$, we define a character $\Psi$ on $\prod_{\alpha\in\Pi}\mathcal{U}_\alpha$  by $\Psi(\prod_{\alpha\in\Pi}u_\alpha)=\prod_{\alpha\in\Pi}\varphi(u_\alpha)$ $(u_\alpha\in\mathcal{U}_\alpha)$. Thus $\Psi$ can be viewed as a character on $\mathcal{U}$, and it is invariant under the action of $\gamma$. 

Then, $E(\phi,1)$ has a Whittaker module associated with $\Psi$, and we denote the Whittaker functional by $\mathcal{W}$.
Since $M(\phi_{\mathcal{Q}},1)$ is $\gamma$-invariant, the standard representation $E(\phi,1)$ is $\gamma$-invariant (cf. section 6.1).  We take  a nontrivial intertwining operator $I$ from $E(\phi,1)$ to $\gamma^*E(\phi,1)$, then $\mathcal{W}\circ I$ is also a Whittaker functional. By \cite{rodier1973whittaker}, the space of Whittaker functionals is one‑dimensional, so there exists $c \in \mathbb{C}^*$ such that $\mathcal{W} \circ I = c\mathcal{W}$.
We set \begin{equation*}
	I_W:=c^{-1}I
\end{equation*}
Then $I_W$ is the unique intertwining operator from $E(\phi,1)$ to $\gamma^*E(\phi,1)$ satisfying \begin{equation*}
	\mathcal{W}=\mathcal{W}\circ I_W
\end{equation*}

Because $M(\phi,1)$ is the unique quotient of $E(\phi,1)$, the operator $I_W$ descends to an operator on $M(\phi,1)$, which we also denote by $I_W$.

Hence we can define a representation $(M(\phi,1))^+_W$ of $\mathcal{G}^+$  by 
\begin{equation*}
	\begin{aligned}
		(g,1)\cdot v&=g\cdot v\\
		(g,\gamma)\cdot v&=g\cdot I_W(v) 
	\end{aligned}
\end{equation*}

for $g\in\mathcal{G}$ and $v\in M(\phi,1)$.

Our second main result is that these two classifications coincide:
\begin{theorem}[Theorem \ref{MainTh}]
	There is an isomorphism of $\mathcal{G}^+$ representations \begin{equation*}
		(M(\phi,1))^+_W\simeq M^+(\phi,1)
	\end{equation*}    
\end{theorem}

In section 2, we use type theory to get an equivalence of categories between $\mathrm{Rep}(\mathcal{G}^+)^{[\mathcal{T},1_{\triv}]}$ and 
the category of modules over the Iwahori–Hecke algebra $\mathcal{H}(\mathcal{J}\backslash\mathcal{G}^+/\mathcal{J})$, and we discuss compatibility with induction.

In section 3, we define the twisted affine Hecke algebra $\mathcal{H}(G^+,v)$ and the twisted graded Hecke algebra, and we describe the relation between their module categories. 

In section 4, we use the equivariant derived category to geometrize the graded Hecke algebra, obtaining irreducible modules and standard modules. We also prove a version of the Kazhdan-Lusztig conjecture for twisted graded Hecke algebras and compare the irreducible modules and standard modules arising from \cite{Aubert2016GradedHA}.

In section 5, we define the standard representations in $\mathrm{Rep}(\mathcal{G^+})^{[\mathcal{T},1_{\triv}]}$ and formulate our twisted Kazhdan–Lusztig conjecture.

In section 6, we prove that the Langlands classification coming from geometry is compatible with the Whittaker-normalized one. 

The paper by \cite{SolCon} has some intersection with our section 4. Solleveld uses a corrected version of Lusztig \cite{lusztig2} and proves a version of the Kazhdan-Lusztig conjecture for (twisted) graded Hecke algebras. When we only consider the constant sheaf, we can get the odd vanishing condition for cohomology with compact support (Lemma \ref{OddV}) in the original version of Lusztig   \cite{lusztig2}. We mainly use the descent datum (Lemma \ref{Inv}) to relate the connected case and the disconnected case and some techniques in equivariant derived category in section 4.

The author would like to thank XU Bin for his support and advice. This work would not be possible without his guidance. The author would also like to thank DENG Taiwang for discussions.

\section{Representations of $p$-adic group}
\subsection{Twisted Iwahori Hecke algebra}
Let $\mathbb{B}$ be the Borel subgroup of upper triangular matrices of $\mathrm{GL}_n$ and let $\mathbb{T}$ be the Levi subgroup of $\mathbb{B}$ which is a maximal torus of $\GL_n$.   Let $X=\operatorname{Hom}(\mathbb{T},\mathbb{G}_m)$, $Y=\operatorname{Hom}(\mathbb{G}_m,\mathbb{T})$, and $\langle,\rangle:X\times Y\to \mathbb{Z}$ be the usual perfect pairing defined by $x \circ y(t) = t^{\langle x, y \rangle}$ for $t \in \mathbb{G}_m$. Let $R$ denote the set of roots of $\mathrm{GL}_n$ with respect to $\mathbb{T}$, $R^+ \subset R$ the positive roots determined by $\mathbb{B}$, and $\Pi \subset R^+$ the simple roots. We fix a pinning of $\GL_n$. In particular, for every $\alpha\in R$, there is a homomorphism $i_\alpha: \mathrm{SL}_2\to \GL_n$ such that $i_\alpha\begin{pmatrix}
	1 & u\\
	0 &1
\end{pmatrix}$ is a one-parameter root group.

Define $F=\Qp$, $\mathcal{G}=\GL_n(F)$, and $\mathcal{T}=\mathbb{T}(F)$. The associated root datum is $\mathcal{R}(\mathcal{G},\mathcal{T}) = (X, R, Y, R^\vee, \Pi)$, where $R^\vee$ is the set of coroots.
Let $\mathscr{O}$ be the ring of integers of $F$, let $\varpi$ be a uniformizer of $F$, and let $\mathscr{P}$ be the maximal ideal of $\mathscr{O}$.

Let $\mathcal{G}^+=\mathcal{G}\rtimes \Gamma$, where $\Gamma = \langle\gamma\rangle$ with   $\gamma(g) = J(^{t}g^{-1})J^{-1}$ and
\begin{equation*}
	J=\begin{pmatrix}
		0 & 0 &\cdots &0 &-1 \\
		0 & 0 &\cdots &(-1)^{2} &0 \\
		\vdots &\vdots &\ddots & \vdots &\vdots\\
		(-1)^n & 0 & \cdots &0 &0 
	\end{pmatrix}.
\end{equation*}

Then $\gamma^2=1$ and the conjugation action of $(1,\gamma)$ on $\mathcal{G}$ satisfies $(1,\gamma)(g,1)(1,\gamma)^{-1} = (\gamma(g),1)$. Let $W^+=N_{\mathcal{G}^+}(\mathcal{T})/\mathcal{T}$. Hence $W^+=W\rtimes \Gamma$, where $W \simeq S_n$ is the Weyl group of $\mathcal{G}$. We choose the standard simple reflections $s_{\alpha_i} = s_{e_i - e_{i+1}}$ for $i = 1,\dots,n-1$; then $\gamma(s_{\alpha_i}) = s_{\alpha_{n-i}}$.

The  $\Gamma$ action also induces an automorphism of $\mathbb{T}$, hence we have a $\Gamma$ action on $X$ and $Y$.

Let $\mathcal{J}$ be the Iwahori subgroup of $\mathcal{G}$, defined as the inverse image of $\mathbb{B}(\mathscr{O}/\mathscr{P})$ under the map $\GL_n(\mathscr{O})\to \GL_n(\mathscr{O}/\mathscr{P})$. Explicitly,
\begin{equation}
	\mathcal{J}=\begin{pmatrix}
		\mathscr{O}^\times & \mathscr{O} & \cdots & \mathscr{O} \\
		\mathscr{P} & \ddots & \ddots & \vdots \\
		\vdots & \ddots & & \mathscr{O} \\
		\mathscr{P} & \cdots & \mathscr{P} & \mathscr{O}^\times
	\end{pmatrix}
\end{equation}

Denote by $\mathcal{H}(\mathcal{G}^+)$ the space of locally constant complex‑valued functions on $\mathcal{G}^+$ with compact support.  Then we have their convolution with respect to the left Haar measure:
\begin{equation}\label{HeMul}
	f_1\cdot f_2(g)=\int_{\mathcal{G}^+}f_1(h)f_2(h^{-1}g)dh
\end{equation}

We normalize the left Haar measure such that $\mathcal{J}$ has volume 1.

For subgroup $(\mathcal{J},1)$ of $\mathcal{G}^+$, we will simply write $\mathcal{J}$. We define $\mathcal{H}(\mathcal{J}\backslash\mathcal{G}^+/\mathcal{J})$ as the subalgebra of  $\mathcal{H}(\mathcal{G}^+)$ consisting of  $\mathcal{J}$ bi-invariant functions: 
$$\mathcal{H}(\mathcal{J}\backslash\mathcal{G}^+/\mathcal{J})=\{f\in \mathcal{H}(\mathcal{G}^+) \mid f(k_1gk_2)=f(g) \quad \forall g\in \mathcal{G}^+,k_1,k_2\in \mathcal{J}\}$$

There is an isomorphism $Y \simeq \mathcal{T}/\mathbb{T}(\mathscr{O})$ given by $y \mapsto \bar{y} := y(\varpi^{-1})$, and 
$W\simeq N_{\mathcal{G}}(\mathcal{T})/\mathcal{T}$ given by $s_{\alpha_i}\mapsto \bar{s}_{\alpha_i}:=i_{\alpha_i}\begin{pmatrix}
	0 & 1\\
	-1 &0
\end{pmatrix}=\begin{pmatrix}
    I_{i-1} & & & \\
     &0 & 1 & \\
      & -1 &0 \\
      & & &I_{n-1-i}
\end{pmatrix}$. Then $N_{\mathcal{G}}(\mathcal{T})/\mathbb{T}(\mathcal{O})\simeq Y\rtimes W$.  
Hence we have \begin{equation}\label{GAct1}
\gamma(\bar{s}_{\alpha_i})=\bar{s}_{\alpha_{n-i}} \end{equation}
which shows that the $\gamma$ action is compatible with the action on $W$, 
and \begin{equation}\label{GAct2}
	\gamma(y((\varpi^{-1})))=\gamma(y)(\varpi^{-1})
\end{equation} 
which shows that the $\gamma$ action is compatible with the action on $Y$.

The Iwahori decomposition of $\mathcal{G}$ is
\begin{equation}
	\mathcal{G}=\sqcup_{{y\in Y, w\in W}}\mathcal{J}\bar{y}\bar{w}\mathcal{J}
\end{equation}

Since $\gamma$ normalizes $\mathcal{J}$, we obtain a decomposition for $\mathcal{G}^+$:
\begin{equation}\label{HeVe}
\begin{aligned}
	\mathcal{G}^+ &=\mathcal{G}\rtimes \Gamma \\
	&= (\sqcup_{{y\in Y, w\in W}}\mathcal{J}\bar{y}\bar{w}\mathcal{J},1)\sqcup (\sqcup_{{y\in Y, w\in W}}\mathcal{J}\bar{y}\bar{w}\mathcal{J},\gamma)\\
	&=\sqcup_{y\in Y,w\in W,\gamma\in \Gamma}\mathcal{J}(\bar{y}\bar{w},\gamma)\mathcal{J}
\end{aligned} 
\end{equation}

Let $T_{(w,1)}$ be the characteristic function of the double coset $\mathcal{J}(\bar{w},1)\mathcal{J}$ where $w\in Y\rtimes W$. 
For $w\in Y\rtimes W$, define the length function $l$ such that
\begin{equation}
	p^{l(w)}=\operatorname{vol}(\mathcal{J}\bar{w}\mathcal{J})=[\mathcal{J}\bar{w}\mathcal{J}:\mathcal{J}]=[\mathcal{J}:\mathcal{J}\cap \bar{w}\mathcal{J}\bar{w}^{-1}]
\end{equation}

Denote $T_{s_\alpha}=T_{(s_\alpha,1)}$, then  $(T_{s_\alpha}+1)(T_{s_\alpha}-p)=0$ for a simple reflection $s_\alpha\in W$, $\alpha\in \Pi$.

Let
\begin{equation}
Y^+=\{y\in Y \mid \langle \alpha,y\rangle \geq 0 ,\alpha\in R^+\}
\end{equation}

Denote $T_y=T_{(y(\varpi^{-1}),1)}$. Any element $y\in Y$ can be written as a linear combination $y=y_1-y_2$ for $y_1,y_2\in Y^+$. Let $\theta_y=p^{-(l(y_1)-l(y_2))/2}T_{y_1}T_{y_2}^{-1}$. If $u,v\in Y^+$, then $l(u+v)=l(u)+l(v)$. Hence $T_{u+v}=T_{u}T_{v}$, and $\theta_{u+v}=\theta_{u}\theta_{v}$. Then we have Bernstein-Zelevinsky presentation 
\begin{equation} \label{mul-hecke}
	\theta_yT_{s_\alpha}-T_{s_\alpha}\theta_{s_\alpha(y)}=(p-1)\frac{\theta_y-\theta_{s_\alpha(y)}}{1-\theta_{-\alpha}}
\end{equation} 
where $y\in Y$ and $\alpha\in  \Pi$.

Then $T_{s_\alpha}$, $\theta_y$ generate the Iwahori Hecke algebra $\mathcal{H}(\mathcal{J}\backslash\mathcal{G}/\mathcal{J})$ which we can view as the affine Hecke algebra $\mathcal{H}(G,p^{1/2})$ of the complex dual group $G=\GL_n(\mathbb{C})$ with the root datum $\mathcal{R}(G,T)=\mathcal{R}^\vee (\mathcal{G},\mathcal{T})=(Y,R^\vee,X,R,\Pi^\vee)$, where the algebra $\mathcal{H}(\mathcal{J}\backslash\mathcal{G}/\mathcal{J})$ is $\mathcal{J}$ bi-invariant compactly supported functions on $\mathcal{G}$, $T=\mathbb{T}(\mathbb{C})$.

\begin{proposition}\label{IwaHec}
	The twisted affine Hecke algebra $\mathcal{H}(\mathcal{J}\backslash\mathcal{G}^+/\mathcal{J})$ is the vector space $\mathcal{H}(\mathcal{J}\backslash\mathcal{G}/\mathcal{J})\otimes\mathbb{C}[\Gamma]$   with the following multiplication rules: 
	
	$\bullet$ $(T_{s_\alpha}+1)(T_{s_\alpha}-p)=0$, $\alpha\in 
	\Pi$ 
	
	$\bullet$ $T_{w_1}T_{w_2}=T_{w_1w_2}$, if $l(w_1w_2)=l(w_1)+l(w_2)$, $w_1,w_2\in W$
	
	$\bullet$ The group algebra $\mathbb{C}[Y]$ is embedded as subalgebra
	
	$\bullet$ For $y\in Y$ and $\alpha\in  \Pi$:
	\begin{equation}
		\theta_yT_{s_\alpha}-T_{s_\alpha}\theta_{s_\alpha(y)}=(p-1)\frac{\theta_y-\theta_{s_\alpha(y)}}{1-\theta_{-\alpha}}
	\end{equation}

	$\bullet$ For $T_\gamma\in \Gamma$, $w\in W$, and $y\in Y$:
	\begin{equation}
		\begin{aligned}
		T_\gamma^2&=1\\
		T_\gamma \theta_yT_w T_\gamma&=\theta_{\gamma(y)}T_{\gamma(w)}
		\end{aligned}
	\end{equation}
\end{proposition}
\begin{proof}
Let $T_{\gamma}$ be the characteristic function of the double coset $\mathcal{J}(1,\gamma)\mathcal{J}$.
Using \eqref{HeVe}, we have the isomorphism of vector spaces.
    We have an algebra embedding \begin{align*}
        \mathcal{H}(\mathcal{J}\backslash\mathcal{G}/\mathcal{J})\hookrightarrow \mathcal{H}(\mathcal{J}\backslash\mathcal{G}^+/\mathcal{J})
    \end{align*}
    So we get the first four multiplications.

    The action $\gamma$ normalizes $\mathcal{J}$. Hence from \eqref{HeMul} we have \begin{align*}
        T_\gamma\cdot T_\gamma(g)= \int_{\mathcal{G}^+}T_\gamma(h)T_\gamma(h^{-1}g)dh
    \end{align*}
    The integral vanishes unless $h^{-1}g\in \mathcal{J}(1,\gamma)\mathcal{J}$ and $h\in \mathcal{J}(1,\gamma)\mathcal{J}$, which implies that $g$ is in $\mathcal{J}$. Thus the convolution is supported on a single double coset and so $T_\gamma\cdot T_\gamma=c$ for some constant $c$.   Since \begin{align*}
        T_\gamma\cdot T_\gamma(1)= \int_{\mathcal{G}^+}T_\gamma(h)T_\gamma(h^{-1})dh=\int_{\mathcal{J}(1,\gamma)\mathcal{J}}1dh=1
    \end{align*}
    We have $T_\gamma\cdot T_\gamma=1$.

For $w\in Y\rtimes W$, we have \begin{align*}
    T_{\gamma}\cdot T_{(w,1)}(g)= \int_{\mathcal{G}^+}T_\gamma(h)T_{(w,1)}(h^{-1}g)dh
\end{align*}

The integral vanishes unless $h^{-1}g\in \mathcal{J}(\bar{w},1)\mathcal{J}$ and $h\in \mathcal{J}(1,\gamma)\mathcal{J}$, which implies that $g$ is in $\mathcal{J}(1,\gamma)(\bar{w},1)\mathcal{J}=\mathcal{J}(\gamma(\bar{w}),\gamma)\mathcal{J}$. Since \begin{align*}
    T_{\gamma}\cdot T_{(w,1)}((\gamma(\bar{w}),\gamma))= \int_{\mathcal{G}^+}T_\gamma(h)T_{(w,1)}(h^{-1}(\gamma(\bar{w}),\gamma))dh=\int_{\mathcal{J}(1,\gamma)\mathcal{J}}1dh=1
\end{align*}
Using \eqref{GAct1} and \eqref{GAct2}, we have $(\gamma(\bar{w}),\gamma)=(\overline{\gamma(w)},1)$, thus $T_{\gamma}\cdot T_{(w,1)}=T_{(\gamma(w),\gamma)}$. Using the same method, we also have \begin{align*}
T_{(w,1)}T_{\gamma}=T_{(w,\gamma)}
\end{align*} and \begin{align*}
T_{\gamma}T_{(w,1)}T_{\gamma}=T_{(\gamma(w),1)}
\end{align*}

Recall $T_{s_\alpha}=T_{(s_\alpha,1)}$ and $T_y=T_{(y(\varpi^{-1}),1)}$, then $T_\gamma T_{s_\alpha}T_\gamma=T_{\gamma(s_\alpha)}$ and $T_\gamma T_yT_\gamma^{-1}=T_{\gamma(y)}$.

For $y\in Y^+$, we have $l(y)=\langle \rho,y\rangle$, where $\rho=\sum_{\alpha\in R^+}\alpha$. Recall we have $\Gamma$ action on $X$ and $Y$. In fact, it preserves the pairing $\langle,\rangle$, and permutes all positive roots. For every $\alpha\in R^+$, \begin{align*}
    \langle \alpha,\gamma(y)\rangle =\langle \gamma(\alpha),y\rangle \geq 0
\end{align*}
which means that $\gamma(y)\in Y^+$. Hence we have \begin{align*}
    l(\gamma(y))=\langle \rho,\gamma(y)\rangle=\langle \gamma(\rho),y\rangle=\langle \rho,y\rangle=l(y)
\end{align*}

Recall $\theta_y=p^{-l(y)/2}T_{y}$. We have \begin{align*}
    T_\gamma \theta_yT_\gamma^{-1}=\theta_{\gamma(y)}
\end{align*}
Hence  $\theta_y^{-1}T_\gamma\theta_{\gamma(y)}T_\gamma=1$, which means $T_\gamma\theta_y^{-1}T_\gamma\theta_{\gamma(y)}=1$  i.e.  $T_\gamma\theta_{-y}T_\gamma=\theta_{-\gamma(y)}$. 

So for any element $y\in Y$, $y=y_1-y_2$, where $y_1,y_2\in Y^+$, \begin{align*}
T_\gamma\theta_yT_\gamma=T_\gamma\theta_{y_1}\theta_{-y_2}T_\gamma=\theta_{\gamma(y)}
\end{align*}

\end{proof}

\begin{remark}
    We define the $\gamma$ action on $\mathcal{H}(\mathcal{J}\backslash\mathcal{G}/\mathcal{J})$ by $\gamma(\theta_yT_w):=\theta_{\gamma(y)}T_{\gamma(w)}$.
Recall the crossed product algebra $\mathcal{H}(\mathcal{J}\backslash\mathcal{G}/\mathcal{J})\rtimes \Gamma$ is the vector space $\mathcal{H}(\mathcal{J}\backslash\mathcal{G}/\mathcal{J})\otimes\mathbb{C}[\Gamma]$ with the multiplication defined by
\begin{equation}
\gamma\cdot \theta_yT_w \cdot \gamma^{-1} = \gamma(\theta_yT_w)
\end{equation} 
There is an isomorphism \begin{align*}
    \mathcal{H}(\mathcal{J}\backslash\mathcal{G}/\mathcal{J})\rtimes \Gamma\simeq \mathcal{H}(\mathcal{J}\backslash\mathcal{G}^+/\mathcal{J})
\end{align*}
\end{remark}

\subsection{Type theory}
Denote by $\mathrm{Rep}(\mathcal{G})$ the category of smooth $\mathcal{G}$-representation. By \cite{zbMATH03687666}, all irreducible representations in $\mathrm{Rep}(\mathcal{G})$ are classified by multisegments. In particular, they are admissible.

Let $X_{nr}(\mathcal{T})$ be the group of unramified characters $\mathcal{T}\to \mathbb{C}^\times$, i.e.  characters trivial on  $\mathbb{T}(\mathscr{O})$. Any $\rho\in X_{nr}(\mathcal{T})$  can be written as $\rho=\prod\limits_{i=1}^n|\cdot|_F^{s_i}$ with $s_i\in \mathbb{C}$. Let $\mathrm{Rep}(\mathcal{G})^{[\mathcal{T},1_{\triv}]}$ be the full subcategory of $\mathrm{Rep}(\mathcal{G})$ consisting of representations  whose irreducible subquotients are subquotients of the normalized parabolic induction $i_{\mathcal{B}}^{\mathcal{G}}(\rho)$ for some  $\rho \in X_{nr}(\mathcal{T})$   \cite{zbMATH03687666}.

For a representation $(\sigma,W)\in \mathrm{Rep}(\mathcal{G})$, write $W^{\mathcal{J}}$ for  the subspace of $\mathcal{J}$-fixed vectors in $W$. Let  $\mathrm{Rep}(\mathcal{G})^{(\mathcal{J},1_{\triv})}$ be the full subcategory of $\mathrm{Rep}(\mathcal{G})$ whose objects are those $(\sigma,W)$ such that $W$ is generated over $\mathcal{G}$ by $W^{\mathcal{J}}$. From Borel, we have equivalences of categories $\mathrm{Rep}(\mathcal{G})^{(\mathcal{J},1_{\triv})}\simeq \mathrm{Rep}(\mathcal{G})^{[\mathcal{T},1_{\triv}]}\simeq \mathcal{H}(\mathcal{J}\backslash\mathcal{G}/\mathcal{J})-\mathrm{Mod}$.

The disconnected group $\mathcal{G}^+$  has the topological group structure which comes from $\mathcal{G}$. We denote by $\mathrm{Rep}(\mathcal{G}^+)$ the category of smooth $\mathcal{G}^+$-representation. For  $(\pi,V)\in \mathrm{Rep}(\mathcal{G}^+)$ and $f\in \mathcal{H}(\mathcal{G}^+)$, we define:
\begin{equation}
\pi(f)v=\int_{\mathcal{G}^+}f(g)\pi(g)vdg
\end{equation}

Then $\pi(f_1\cdot f_2)=\pi(f_1)\circ \pi(f_2)$.

Let $e_{\mathcal{J}}$ be the characteristic function of $\mathcal{J}$. Then there is an algebra isomorphism $\mathcal{H}(\mathcal{J}\backslash\mathcal{G}^+/\mathcal{J})\simeq e_{\mathcal{J}}\cdot \mathcal{H}(\mathcal{G}^+) \cdot e_{\mathcal{J}}$ and we have $V^{\mathcal{J}}=\pi(e_{\mathcal{J}})V$ where $V^{\mathcal{J}}$ denotes the subspace of $\mathcal{J}$-fixed vectors in $V$.

\begin{proposition}\label{Type}
	A representation $(\pi, V)\in \mathrm{Rep}(\mathcal{G}^+)$ is generated over $\mathcal{G}^+$ by $V^{\mathcal{J}}$ if and only if $\mathrm{Res}^{\mathcal{G}^+}_{\mathcal{G}} V \in \mathrm{Rep}(\mathcal{G})^{(\mathcal{J},1_{\triv})}$.
\end{proposition}
\begin{proof}
	Suppose $V$ is generated over $\mathcal{G}^+$ by $V^{\mathcal{J}}$. Denote $V_1=\mathcal{G}\cdot  V^{\mathcal{J}}$ and $V_2=\mathcal{G}\cdot (1,\gamma)\cdot  V^{\mathcal{J}}$, then we have $V=V_1+V_2$. For $v\in  V^{\mathcal{J}}$, we have $(g,1)(1,\gamma)\cdot v=(1,\gamma)(\gamma(g),1)\cdot v$. So if $g\in \mathcal{J}$, we have $\gamma(g)\in \mathcal{J}$ and $(\gamma(g),1)\cdot v=v$, which means $(g,1)(1,\gamma)\cdot v=(1,\gamma)\cdot v$, i.e. $(1,\gamma)\cdot v\in V^{\mathcal{J}}$. As a result, $(1,\gamma)\cdot V^{\mathcal{J}}=V^{\mathcal{J}}$, hence $V=\mathcal{G}\cdot V^{\mathcal{J}}$ and $\mathrm{Res}^{\mathcal{G}^+}_{\mathcal{G}} V \in \mathrm{Rep}(\mathcal{G})^{(\mathcal{J},1_{\triv})}$.

	Conversely, if $\mathrm{Res}^{\mathcal{G}^+}_{\mathcal{G}} V \in Rep(\mathcal{G})^{(\mathcal{J},1_{\triv})}$, we have $V=\mathcal{G}\cdot V^{\mathcal{J}}= \mathcal{G}^+\cdot V^{\mathcal{J}}$.
\end{proof}

If $V'$ is a subquotient of $V$ as $\mathcal{G}^+$-representation, $V'$ is also a subquotient of $V$ as $\mathcal{G}$-representation. If $\mathrm{Res}^{\mathcal{G}^+}_{\mathcal{G}} V \in \mathrm{Rep}(\mathcal{G})^{(\mathcal{J},1_{\triv})}$, then $\mathrm{Res}^{\mathcal{G}^+}_{\mathcal{G}} V' \in \mathrm{Rep}(\mathcal{G})^{(\mathcal{J},1_{\triv})}$. We denote by $\mathrm{Rep}(\mathcal{G}^+)^{(\mathcal{J},1_{\triv})}$ the full subcategory of $\mathrm{Rep}(\mathcal{G}^+)$ consisting of those representations satisfying the condition in Proposition \ref{Type}. Then $\mathrm{Rep}(\mathcal{G}^+)^{(\mathcal{J},1_{\triv})}$ is closed under taking subquotient. From \cite{zbMATH01234305} we have the following.

\begin{proposition}\label{Type2}
	The functor
	\begin{equation}
		\begin{aligned}
		\Lambda_{\mathcal{J}}^+:\mathrm{Rep}(\mathcal{G}^+)^{(\mathcal{J},1_{\triv})}&\to \mathcal{H}(\mathcal{J}\backslash\mathcal{G}^+/\mathcal{J})-\mathrm{Mod}\\
		V &\mapsto V^{\mathcal{J}}
	\end{aligned}	
	\end{equation}

	is an equivalence of categories. Its inverse is given by \begin{align}
		(\Lambda_{\mathcal{J}}^+)^{-1}:V\mapsto\mathcal{H}(\mathcal{G}^+)\otimes_{\mathcal{H}(\mathcal{J}\backslash\mathcal{G}^+/\mathcal{J})} V
	\end{align}
\end{proposition}

We write $\mathrm{Rep}(\mathcal{G}^+)^{[\mathcal{T},1_{\triv}]}$ for the full subcategory of $\mathrm{Rep}(\mathcal{G}^+)$ whose restriction to $\mathrm{Rep}(\mathcal{G})$ lies in  $\mathrm{Rep}(\mathcal{G})^{[\mathcal{T},1_{\triv}]}$. Using $\mathrm{Rep}(\mathcal{G})^{(\mathcal{J},1_{\triv})}\simeq \mathrm{Rep}(\mathcal{G})^{[\mathcal{T},1_{\triv}]}$ together with Proposition \ref{Type}, we obtain \begin{align*}
	\mathrm{Rep}(\mathcal{G}^+)^{(\mathcal{J},1_{\triv})}\simeq \mathrm{Rep}(\mathcal{G}^+)^{[\mathcal{T},1_{\triv}]}
\end{align*}

We have the commutative diagram 
\begin{equation*}
\begin{tikzcd}
	\mathrm{Rep}(\mathcal{G}^+)^{(\mathcal{J},1_{\triv})}
	\arrow[r, "\simeq", "\Lambda_{\mathcal{J}}^+"']
	\arrow[d, "\mathrm{Res}^{\mathcal{G}^+}_{\mathcal{G}}"']
	&
	\mathcal{H}(\mathcal{J}\backslash\mathcal{G}^+/\mathcal{J})\text{-}\mathrm{Mod}
	\arrow[d, "\mathrm{Res}^{\mathcal{H}(\mathcal{J}\backslash\mathcal{G}^+/\mathcal{J})}_{\mathcal{H}(\mathcal{J}\backslash\mathcal{G}/\mathcal{J})}"]
	\\
	\mathrm{Rep}(\mathcal{G})^{(\mathcal{J},1_{\triv})}
	\arrow[r, "\simeq", "\Lambda_{\mathcal{J}}"']
	&
	\mathcal{H}(\mathcal{J}\backslash\mathcal{G}/\mathcal{J})\text{-}\mathrm{Mod}
\end{tikzcd}
\end{equation*}
where the right vertical restriction is induced by the embedding $\mathcal{H}(\mathcal{J}\backslash\mathcal{G}/\mathcal{J})\hookrightarrow \mathcal{H}(\mathcal{J}\backslash\mathcal{G}^+/\mathcal{J})\simeq \mathcal{H}(\mathcal{J}\backslash\mathcal{G}/\mathcal{J})\rtimes \Gamma$ which is identity on $\mathcal{H}(\mathcal{J}\backslash\mathcal{G}/\mathcal{J})$.

For $(\sigma_1,W_1)\in \mathrm{Rep}(\mathcal{G})$, define $\mathrm{ind}^{\mathcal{G}^+}_{\mathcal{G}}(W_1)=\{f:\mathcal{G}^+\to W_1| f(gg')=\sigma_1(g)f(g'),g\in \mathcal{G},g'\in \mathcal{G}^+\}$ is the representation of $\mathcal{G}^+$ given by right translation on the space. Taking $(\sigma_2,W_2)\in \mathcal{H}(\mathcal{J}\backslash\mathcal{G}/\mathcal{J})-\mathrm{Mod}$, $\mathrm{ind}^{\mathcal{H}(\mathcal{J}\backslash\mathcal{G}^+/\mathcal{J})}_{\mathcal{H}(\mathcal{J}\backslash\mathcal{G}/\mathcal{J})} (W_2)=\{f:{\mathcal{H}(\mathcal{J}\backslash\mathcal{G}^+/\mathcal{J})} \to W_2| f(gg')=\sigma_2(g)f(g'),g\in {\mathcal{H}(\mathcal{J}\backslash\mathcal{G}/\mathcal{J})}_,g'\in {\mathcal{H}(\mathcal{J}\backslash\mathcal{G}^+/\mathcal{J})}\}$  is regarded as a left $\mathcal{H}(\mathcal{J}\backslash\mathcal{G}^+/\mathcal{J})$-module via right translation. By the uniqueness of adjoint functors, we have the commutative diagram
\begin{equation}\label{IndComPG}
	\begin{tikzcd}
		\mathrm{Rep}(\mathcal{G}^+)^{(\mathcal{J},1_{\triv})}
		\arrow[r, "\simeq"', "\Lambda_{\mathcal{J}}^+"]
		&
		\mathcal{H}(\mathcal{J}\backslash\mathcal{G}^+/\mathcal{J})\text{-}\mathrm{Mod} \\
		\mathrm{Rep}(\mathcal{G})^{(\mathcal{J},1_{\triv})}
		\arrow[r, "\simeq"', "\Lambda_{\mathcal{J}}"]
		\arrow[u, "\mathrm{ind}^{\mathcal{G}^+}_{\mathcal{G}}"]
		&
		\mathcal{H}(\mathcal{J}\backslash\mathcal{G}/\mathcal{J})\text{-}\mathrm{Mod}
		\arrow[u, "\mathrm{ind}^{\mathcal{H}(\mathcal{J}\backslash\mathcal{G}^+/\mathcal{J})}_{\mathcal{H}(\mathcal{J}\backslash\mathcal{G}/\mathcal{J})}"']
	\end{tikzcd}
\end{equation}

\begin{remark}
	Note that $\mathcal{H}(\mathcal{J}\backslash\mathcal{G}/\mathcal{J})$ is a subalgebra of $\mathcal{H}(\mathcal{J}\backslash\mathcal{G}^+/\mathcal{J})$, so we can also define an induction $\mathrm{Ind}^{\mathcal{H}(\mathcal{J}\backslash\mathcal{G}^+/\mathcal{J})}_{\mathcal{H}(\mathcal{J}\backslash\mathcal{G}/\mathcal{J})} (V):= \mathcal{H}(\mathcal{J}\backslash\mathcal{G}^+/\mathcal{J})\otimes_{\mathcal{H}(\mathcal{J}\backslash\mathcal{G}/\mathcal{J})}V$. The two induction functors are naturally isomorphic given by \begin{align*}
	    \mathrm{ind}^{\mathcal{H}(\mathcal{J}\backslash\mathcal{G}^+/\mathcal{J})}_{\mathcal{H}(\mathcal{J}\backslash\mathcal{G}/\mathcal{J})} (V) &\to \mathrm{Ind}^{\mathcal{H}(\mathcal{J}\backslash\mathcal{G}^+/\mathcal{J})}_{\mathcal{H}(\mathcal{J}\backslash\mathcal{G}/\mathcal{J})} (V)\\
        f&\mapsto \sum_{a\in \Gamma}a\otimes f(a)
	\end{align*}
\end{remark}

For $(\sigma,W)\in \mathrm{Rep}(\mathcal{G})$, we define $(\gamma^*{\sigma}) (g):= \sigma(\gamma(g))$

From \cite{zbMATH01014663},  for an irreducible  representation $(\pi,V)\in \mathrm{Rep}(\mathcal{G}^+)$ either

$\bullet$ If $\mathrm{Res}^{\mathcal{G}^+}_{\mathcal{G}}(\pi)$ is irreducible, i.e. $\mathrm{Res}^{\mathcal{G}^+}_{\mathcal{G}}(\pi)=\sigma$,  then $\sigma \simeq \gamma^* \sigma$ and there are two extensions of $\sigma$ to $\mathcal{G}^+$.

or

$\bullet$ If $\mathrm{Res}^{\mathcal{G}^+}_{\mathcal{G}}(\pi)$ is reducible, i.e. $\mathrm{Res}^{\mathcal{G}^+}_{\mathcal{G}}(\pi)=\sigma_1\oplus\sigma_2$, then $\sigma_2\simeq  {\gamma^* \sigma_1}$ and there is only one extension of $\sigma_1$ 
(or of $\sigma_2$) to $\mathcal{G}^+$, and that extension satisfies $\pi\simeq i^{\mathcal{G}^+}_{\mathcal{G}}(\sigma_1)\simeq i^{\mathcal{G}^+}_{\mathcal{G}}(\sigma_2)$

\begin{remark}
    Recall $ \mathcal{H}(\mathcal{J}\backslash\mathcal{G}/\mathcal{J})\rtimes \Gamma\simeq \mathcal{H}(\mathcal{J}\backslash\mathcal{G}^+/\mathcal{J})$.
Denote by ${{\gamma}^*}V$ the $\mathcal{H}(\mathcal{J}\backslash\mathcal{G}/\mathcal{J})$ module with the same underlying vector space but with the action twisted by ${\gamma}$. From \eqref{GAct1} and \eqref{GAct2}, we have \begin{align}\label{TwiE}
    (\gamma^*\pi)^{\mathcal{J}}\simeq {\gamma}^*(\pi^{\mathcal{J}})
\end{align}
Hence from \cite{zbMATH02096630}, we get the same result as above.
\end{remark}

\subsection{Jacquet and induction functors}
Recall $\mathbb{B}$ denotes the Borel subgroup of upper triangular matrices. Denote by $\bar{\mathbb{B}}$ the opposite Borel subgroup, i.e., the lower triangular matrices. 
Let $\mathcal{P}'$ be a parabolic subgroup of $\mathcal{G}$ containing $\bar{\mathbb{B}}(\mathbb{Q}_p)$ and normarlized by $\gamma$. Write $\mathcal{P}'=\mathcal{M}\mathcal{U}'$ with Levi subgroup $\mathcal{M}$ and unipotent radical $\mathcal{U}'$. Then $\gamma$ normalizes $\mathcal{M}$ and $\mathcal{U}'$. Define $\mathcal{M}^+=\mathcal{M}\rtimes \Gamma$, $\mathcal{P}'^+=\mathcal{M}^+\mathcal{U}'$. Since $\mathcal{M}^+$ normalizes $\mathcal{U}'$, we can define functors $i_{\mathcal{P}'^+}$ and $r_{\mathcal{P}'^+}$ as in \cite[1.8]{zbMATH03639915}.

Recall that for any element $y\in Y^+$, we have $y\mapsto y(\varpi^{-1})$. Let $T_y^{\mathcal{G}}$ be the characteristic function of the double coset $\mathcal{J}y(\varpi^{-1})\mathcal{J}$, and  $\theta_y^{\mathcal{G}}=p^{-l(y)/2}T_{y}^{\mathcal{G}}=\operatorname{vol}(\mathcal{J}y(\varpi^{-1})\mathcal{J})^{-1/2}T_{y}^{\mathcal{G}}$ in $\mathcal{H}(\mathcal{J}\backslash\mathcal{G}/\mathcal{J})$. Then $y(\varpi^{-1})$ is positive relative to ($\mathbb{P}'(F)$, $\mathcal{J}$), and $T_y^{\mathcal{G}}$ is an invertible element of $\mathcal{H}(\mathcal{J}\backslash\mathcal{G}/\mathcal{J})$.

Let $\mathcal{J}_{\mathcal{M}}=\mathcal{J}\cap \mathcal{M}$. Then we have the affine Hecke algebra $\mathcal{H}(\mathcal{J}_{\mathcal{M}}\backslash\mathcal{M}/\mathcal{J}_{\mathcal{M}})$ of $\mathcal{M}$ with respect to $\mathcal{J}_{\mathcal{M}}$. By \cite[7.9]{zbMATH01234305}, there is an injective homomorphism
\begin{equation}
	t_{\mathcal{P}'}:\mathcal{H}(\mathcal{J}_{\mathcal{M}}\backslash\mathcal{M}/\mathcal{J}_{\mathcal{M}})  \hookrightarrow \mathcal{H}(\mathcal{J}\backslash\mathcal{G}/\mathcal{J})
\end{equation}
such that for a simple reflection $s_\alpha\in W^{\mathcal{M}}\subset W^{\mathcal{G}}$:
\begin{equation}\label{embedhecke1}
	t_{\mathcal{P}'}(T_{s_\alpha}^{\mathcal{M}})=T_{s_\alpha}^{\mathcal{G}}
\end{equation} 
and for $y\in Y^+$:
\begin{equation}
	t_{\mathcal{P}'}(T_y^{\mathcal{M}})=\frac{\operatorname{vol}(\mathcal{J}_{\mathcal{M}}y(\varpi^{-1})\mathcal{J}_{\mathcal{M}})^{1/2}}{\operatorname{vol}(\mathcal{J}y(\varpi^{-1})\mathcal{J})^{1/2}}T_y^{\mathcal{G}}
\end{equation}

i.e. 
\begin{equation}\label{embedhecke2}
	t_{\mathcal{P}'}(\theta_y^{\mathcal{M}})=\theta_y^{\mathcal{G}}
\end{equation}

Then we obtain an injective homomorphism:
\begin{equation}
	t_{\mathcal{P}'}^0:\mathcal{H}(\mathcal{J}_{\mathcal{M}}\backslash\mathcal{M}/\mathcal{J}_{\mathcal{M}}) \hookrightarrow \mathcal{H}(\mathcal{J}\backslash\mathcal{G}^+/\mathcal{J})
\end{equation}

via the embedding $\mathcal{H}(\mathcal{J}\backslash\mathcal{G}/\mathcal{J})\hookrightarrow \mathcal{H}(\mathcal{J}\backslash\mathcal{G}^+/\mathcal{J})\simeq \mathcal{H}(\mathcal{J}\backslash\mathcal{G}/\mathcal{J})\rtimes \Gamma$. Consequently, there is a  $\Gamma$-action on $\mathcal{H}(\mathcal{J}_{\mathcal{M}}\backslash\mathcal{M}/\mathcal{J}_{\mathcal{M}})$. Using the same method as before, we also have an embedding  $\mathcal{H}(\mathcal{J}_{\mathcal{M}}\backslash\mathcal{M}/\mathcal{J}_{\mathcal{M}})\hookrightarrow \mathcal{H}(\mathcal{J}_{\mathcal{M}}\backslash\mathcal{M}^+/\mathcal{J}_{\mathcal{M}})$, 
which induces another $\Gamma$-action on the same algebra. One easily checks that these two $\Gamma$-actions coincide. Hence we obtain an injective homomorphism:
\begin{equation}\label{embedhecke}
	t_{\mathcal{P}'^+}:\mathcal{H}(\mathcal{J}_{\mathcal{M}}\backslash\mathcal{M}^+/\mathcal{J}_{\mathcal{M}})\hookrightarrow \mathcal{H}(\mathcal{J}\backslash\mathcal{G}^+/\mathcal{J})
\end{equation}

satisfying the same relations \eqref{embedhecke1} and \eqref{embedhecke2} as above, together with
\begin{equation}
	t_{\mathcal{P}'^+}(T_{\gamma}^{\mathcal{M}^+})=T_{\gamma}^{\mathcal{G}^+}
\end{equation}

where $T_{\gamma}^{\mathcal{M}^+}$ is the characteristic function of the double coset $\mathcal{J}_{\mathcal{M}}(1,\gamma)\mathcal{J}_{\mathcal{M}}$ and $T_{\gamma}^{\mathcal{G}^+}$ is the characteristic function of the double coset $\mathcal{J}(1,\gamma)\mathcal{J}$.

For $\mathcal{M}^+$, we denote by $\mathrm{Rep}(\mathcal{M}^+)^{(\mathcal{J},1_{\triv})}$  the full subcategory of $\mathrm{Rep}(\mathcal{M}^+)$ whose restriction to $\mathrm{Rep}(\mathcal{M})$ lies in $\mathrm{Rep}(\mathcal{M})^{(\mathcal{J},1_{\triv})}$. 
By the same arguments as in Propositions \ref{Type} and Propositions \ref{Type2}, we obtain an equivalence of categories
\begin{equation}\label{MCatE}
	\begin{aligned}
		{\Lambda^+_{\mathcal{J}_{\mathcal{M}}}}:\mathrm{Rep}(\mathcal{M}^+)^{(\mathcal{J}_{\mathcal{M}},1_{\triv})}&\to \mathcal{H}(\mathcal{J}_{\mathcal{M}}\backslash\mathcal{M}^+/\mathcal{J}_{\mathcal{M}})-\mathrm{Mod}\\
		V &\mapsto V^{\mathcal{J}_{\mathcal{M}}}
	\end{aligned}
\end{equation}

\begin{proposition}
    There is a commutative diagram:
\begin{equation}
	\xymatrix{
		\mathrm{Rep}(\mathcal{G}^+)^{(\mathcal{J},1_{\triv})} \ar[r]^{\simeq}_{\Lambda_{\mathcal{J}}^+} \ar[d]_{r_{\mathcal{P}'^+}} &\mathcal{H}(\mathcal{J}\backslash\mathcal{G}^+/\mathcal{J})-\Mod \ar[d]^{t_{\mathcal{P}'^+}^*} \\
		\mathrm{Rep}(\mathcal{M}^+)^{(\mathcal{J}_{\mathcal{M}},1_{\triv})} \ar[r]^{\simeq}_{\Lambda^+_{\mathcal{J}_{\mathcal{M}}}} & \mathcal{H}(\mathcal{J}_{\mathcal{M}}\backslash\mathcal{M}^+/\mathcal{J}_{\mathcal{M}})-\Mod }
\end{equation}

where the right vertical restriction is via (\ref{embedhecke}) $t_{\mathcal{P}'^+}:\mathcal{H}(\mathcal{J}_{\mathcal{M}}\backslash\mathcal{M}^+/\mathcal{J}_{\mathcal{M}}) \to \mathcal{H}(\mathcal{J}\backslash\mathcal{G}^+/\mathcal{J})$
\end{proposition}
\begin{proof}
Let $(\pi,V)\in \Rep(\mathcal{G}^+)^{(\mathcal{J},1_{\triv})} $, and write $(\pi_u,V_u)$ for the Jacquet module of $(\pi,V)$ relative to the Jacquet functor $r_{\mathcal{P}'^+}$.
From the definition of the Jacquet functor, $r_{\mathcal{P}'} V$ and $r_{\mathcal{P}'^+}V$ have the same underlying space. Let $q$ denote the quotient map $V\to V_u$.

From the proof in \cite[7.9]{zbMATH01234305}. We have $q(V^{\mathcal{J}})=V_u^{\mathcal{J}_{\mathcal{M}}}$.
    We define a map:
\begin{equation}\label{JIso1}
    \begin{aligned}
        q: V^\mathcal{J} & \to V_u^{\mathcal{J}_{\mathcal{M}}}\\
        v&\mapsto q(v)
    \end{aligned}
\end{equation}

    Hence $q$ is an isomorphism, and for $v\in V^\mathcal{J}$, $f\in \mathcal{H}(\mathcal{J}_{\mathcal{M}}\backslash\mathcal{M}/\mathcal{J}_{\mathcal{M}})$, we have
    \begin{align}\label{ComP1}
        q(\pi(t_{\mathcal{P}'}f)v)=\pi_u(f)q(v)
    \end{align}

Now let $f=T_{\gamma}^{\mathcal{M}^+}$. We have \begin{align*}
    \pi_u(T_{\gamma}^{\mathcal{M}^+})\cdot q(v)&=\int_{\mathcal{M}^+}T_{\gamma}^{\mathcal{M}^+}(h)\pi_u(h)\cdot q(v) dh\\
&=\int_{\mathcal{M}^+}T_{\gamma}^{\mathcal{M}^+}(h)q(\pi (h)v) dh\\
&=q\circ \int_{\mathcal{M}^+}T_{\gamma}^{\mathcal{M}^+}(h)\pi(h)\cdot v dh\\
&= q\circ \int_{\MJ_{\mathcal{M}}(1,\gamma)\MJ_{\mathcal{M}}}\pi(h)\cdot v dh
\end{align*}
Denote $\MJ_u=\MJ\cap \mathcal{U}'$ and $\MJ_u=\MJ\cap \mathcal{U}'_{\op}$ where $\mathcal{U}'_{\op}$ is opposite of $\mathcal{U}'$. We have
\begin{align*}
\pi(t_{\mathcal{P}'^+}T_{\gamma}^{\mathcal{M}^+})v&=\int_{\mathcal{G}^+}T_{\gamma}^{\mathcal{G}^+}(h)\pi(h)\cdot v dh\\
&=\int_{\mathcal{J}(1,\gamma)\mathcal{J}}T_{\gamma}^{\mathcal{G}^+}(h)\pi(h)\cdot v dh\\
&=\int_{\MJ_u\mathcal{J}_{\mathcal{M}}(1,\gamma)\mathcal{J}_{\mathcal{M}}\MJ_l}T_{\gamma}^{\mathcal{G}^+}(h_uh_{\mathcal{M}}h_l)\pi(h_uh_{\mathcal{M}}h_l)\cdot v dh_ldh_{\mathcal{M}}dh_u\\
&= \int_{\MJ_u}\int_{\MJ_{\mathcal{M}}(1,\gamma)\MJ_{\mathcal{M}}}\pi(h)\cdot v dhdh_u
\end{align*}
We apply $q$ to this equation. Combining \eqref{ComP1}, for $f\in \mathcal{H}(\mathcal{J}_{\mathcal{M}}\backslash\mathcal{M}^+/\mathcal{J}_{\mathcal{M}})$, we obtain \begin{equation*}
q(\pi(t_{\mathcal{P}'^+}f)v)=\pi_u (f)q(v)
\end{equation*}
\end{proof}

Let $\mathcal{P}$ denote the opposite of $\mathcal{P}'$, which contains ${\mathbb{B}}(\mathbb{Q}_p)$(the upper triangular Borel subgroup) and ${\mathcal{M}}$. Define $\mathcal{P}^+=\mathcal{P}\rtimes \Gamma$. Then from \cite[5.5]{zbMATH02164641}:
\begin{proposition}[\cite{zbMATH02164641}]
	If $(\pi,V)\in 	\mathrm{Rep}(\mathcal{G}^+)$ is admissible, the contragredient of $r_{\mathcal{P}'^+}(\pi)$ is isomorphic to $r_{\mathcal{P}^+}(\tilde{\pi})$ as $\mathcal{M}^+$-representations, i.e. 
	$\widetilde{r_{\mathcal{P}'^+}(\pi)}\simeq r_{\mathcal{P}^+}(\tilde{\pi})$, where $\widetilde{(\cdot)}$ denotes the contragredient representation.
\end{proposition}

\begin{remark}
	In fact, the above isomorphism holds more generally, not only for admissible representations. Using the $\mathcal{M}$-equivariant non‑degenerate pairing introduced in \cite{Bern} (see also \cite{zbMATH05665785} or \cite{zbMATH01818673} for further details), one notes that the original pairing is $\mathcal{G}^+$-equivariant. Choosing a $\gamma$-invariant open compact subgroup, the pairing becomes $\mathcal{M}^+$-equivariant. In our case  $(\pi,V)\in 	\mathrm{Rep}(\mathcal{G}^+)^{(\mathcal{J},1_{\triv})} $, it can be described as $\mathcal{H}(\mathcal{J}_{\mathcal{M}}\backslash\mathcal{M}^+/\mathcal{J}_{\mathcal{M}})$-equivariant via ${\Lambda^+_{\mathcal{J}_{\mathcal{M}}}}$ for simplicity.
\end{remark}

\begin{proposition}
	If $(\pi,V)\in 	\mathrm{Rep}(\mathcal{G}^+)^{(\mathcal{J},1_{\triv})}$, the contragredient of $r_{\mathcal{P}'^+}(\pi)$ is isomorphic to $r_{\mathcal{P}^+}(\tilde{\pi})$ as $\mathcal{M}^+$-representations, i.e. 
	$\widetilde{r_{\mathcal{P}'^+}(\pi)}\simeq r_{\mathcal{P}^+}(\tilde{\pi})$, where $\widetilde{(\cdot)}$ denotes the contragredient representation.
\end{proposition}
\begin{proof}
    If $(\pi,V)\in 	\mathrm{Rep}(\mathcal{G}^+)^{(\mathcal{J},1_{\triv})}$, we have $(\widetilde{\pi},\widetilde{V})\in 	\mathrm{Rep}(\mathcal{G}^+)^{(\mathcal{J},1_{\triv})}$ and a canonical pairing:\begin{equation}\label{BiL2}
        \begin{aligned}
        \widetilde{V}\times V &\to \mathbb{C}\\
        (\tilde{v},v)& \mapsto \langle\tilde{v},v\rangle
    \end{aligned}
    \end{equation}

    From \cite{zbMATH01818673}, we have \begin{align*}
\langle\tilde{v},\pi(f)v\rangle=\langle\widetilde{\pi}(\tilde{f})\tilde{v},v\rangle
    \end{align*}
    where $\tilde{f}:g\mapsto f(g^{-1})$, $g\in \mathcal{G}^+$, and $\mathcal{H}(\mathcal{J}_{\mathcal{M}}\backslash\mathcal{M}/\mathcal{J}_{\mathcal{M}})$-invariant nondegenerate bilinear form
    \begin{align}\label{BiLi1}
\tilde{V}_l^{\mathcal{J}_{\mathcal{M}}}\times V_u^{\mathcal{J}_{\mathcal{M}}}\to \mathbb{C}
    \end{align}
    where $(\tilde{\pi}_l,\tilde{V}_l)$ is the Jacquet module of $(\tilde{\pi},\tilde{V})$ relative to the Jacquet functor $r_{\mathcal{P}^+}$. The nondegenerate bilinear form \eqref{BiLi1} comes from \eqref{BiL2} and \eqref{JIso1}. Since we have \begin{align*}
\langle\tilde{v},\pi(t_{\mathcal{P}'^+}T_{\gamma}^{\mathcal{M}^+})v\rangle=\langle\widetilde{\pi}(t_{\mathcal{P}^+}T_{\gamma}^{\mathcal{M}^+})\tilde{v},v\rangle
    \end{align*} 
    and \eqref{ComP1}, the nondegenerate bilinear form \eqref{BiLi1} is $\mathcal{H}(\mathcal{J}_{\mathcal{M}}\backslash\mathcal{M}^+/\mathcal{J}_{\mathcal{M}})$-invariant. Using the equivalence of categories \eqref{MCatE}, we obtain the proposition.
\end{proof}

Then we have Bernstein's second adjointness theorem.
\begin{theorem}
	The functor $r_{\mathcal{P}'^+}$ is right adjoint to $i_{\mathcal{P}^+}$, i.e.  $\Hom_{\mathcal{M}^+}(\pi_1,r_{\mathcal{P}'^+}(\pi_2))\simeq \Hom_{\mathcal{G}^+}(i_{\mathcal{P}^+}(\pi_1),\pi_2)$
\end{theorem}

\begin{proof}
	Using the same argument as in  Bernstein and \cite[1.9]{zbMATH03639915}, we have  $i_{\mathcal{P}'^+}(\tilde{\pi_1})=\widetilde{i_{\mathcal{P}'^+}({\pi_1})}$, and $r_{\mathcal{P}'^+}$ is left adjoint to $i_{\mathcal{P}'^+}$. Then 
	\begin{equation}
			\begin{aligned}
			\mathrm{Hom}_{\mathcal{G}^+}(i_{\mathcal{P}'^+}(\pi_1),\tilde{\pi_2}) &\simeq \mathrm{Hom}_{\mathcal{G}^+}(\pi_2,\widetilde{i_{\mathcal{P}'^+}(\pi_1)})\\
			&\simeq \mathrm{Hom}_{\mathcal{G}^+}(\pi_2,{i_{\mathcal{P}'^+}(\tilde{\pi_1}})\\
			&\simeq \mathrm{Hom}_{\mathcal{M}^+}(r_{\mathcal{P}'^+}(\pi_2),{\tilde{\pi_1}})\\
			&\simeq \mathrm{Hom}_{\mathcal{M}^+}({{\pi_1}},\widetilde{r_{\mathcal{P}'^+}(\pi_2)})\\
			&\simeq \mathrm{Hom}_{\mathcal{M}^+}({{\pi_1}},{r_{\mathcal{P}^+}(\tilde{\pi_2})})
		\end{aligned}
	\end{equation}

	If $\pi$ is admissible, then $\pi\simeq \tilde{\tilde{\pi}}$ and we get the theorem. For the general case, the same result follows by a diagram chase as in \cite{Bern}.
\end{proof}

Hence, by uniqueness of adjoint functors, the following diagram commutes:
\begin{equation}\label{commutediagram2}
	\xymatrix{
		\mathrm{Rep}(\mathcal{G}^+)^{(\mathcal{J},1_{\triv})}
		\ar[r]^{\simeq}_{\Lambda_{\mathcal{J}}^+} &\mathcal{H}(\mathcal{J}\backslash\mathcal{G}^+/\mathcal{J})-\mathrm{Mod}  \\
		\mathrm{Rep}(\mathcal{M}^+)^{(\mathcal{J}_{\mathcal{M}},1_{\triv})}\ar[r]^{\simeq}_{\Lambda^+_{\mathcal{J}_{\mathcal{M}}}} \ar[u]^{i_{\mathcal{P}^+}} & \mathcal{H}(\mathcal{J}_{\mathcal{M}}\backslash\mathcal{M}^+/\mathcal{J}_{\mathcal{M}})-\mathrm{Mod} \ar[u]_{\mathcal{H}(\mathcal{J}\backslash\mathcal{G}^+/\mathcal{J})\otimes_{\mathcal{H}(\mathcal{J}_{\mathcal{M}}\backslash\mathcal{M}^+/\mathcal{J}_{\mathcal{M}})}(-)}
	}
\end{equation}
\begin{remark}
	In the connected case, ${\mathcal{P}}$ contains ${\mathbb{B}}(\mathbb{Q}_p)$ (the upper triangular Borel subgroup), we can use diagram \ref{commutediagram2}  to translate Bernstein-Zelevinsky classification into modules of the affine Hecke algebra.
\end{remark}

\section{Representations of twisted affine Hecke algebra}
Now we translate everything to the complex dual group and use the geometry of the twisted affine Hecke algebra and the twisted graded Hecke algebra to classify their irreducible modules and compute multiplicities.
\subsection{Twisted affine Hecke algebra}

Let $G=\GL_n(\mathbb{C})$ be the dual group of $\mathcal{G}$. Define the complex dual group $G^+:=G\rtimes \hat{\Gamma}$, where $\hat{\Gamma} = \langle\hat{\gamma}\rangle$  for $\hat{\gamma}(g)=J(^tg^{-1})J^{-1}$. Let $W^+:=N_{G^+}(T)$. Then $W^+=W\rtimes \hat{\Gamma}$. Denote the root datum of $G=\GL_n(\mathbb{C})$ by $\mathcal{R}(G,T)=(\widetilde{X},\widetilde{R},\widetilde{Y},\widetilde{R}^\vee,\widetilde{\Pi})$. Then we have  $\hat{\Gamma}$ action on $W$ and  $\widetilde{X}$.

\begin{definition}
	The twisted affine Hecke algebra $\mathcal{H}(G^+,v)$ is the vector space $\mathbb{C}[\widetilde{X}]\otimes\mathbb{C}[v,v^{-1}]\otimes\mathbb{C}[W]\otimes\mathbb{C}[\hat{\Gamma}]$ with the following multiplication rules:
	
	$\bullet$ $\mathbb{C}[\widetilde{X}]$, $\mathbb{C}[v,v^{-1}]$, $\mathbb{C}[\hat{\Gamma}]$ are embedded as subalgebras, where $v$ is an indeterminate.
	
	$\bullet$ $(T_{s_\alpha}+1)(T_{s_\alpha}-v^2)=0$, $\alpha\in 	\widetilde{\Pi}$ 
	
	$\bullet$ $T_{w_1}T_{w_2}=T_{w_1w_2}$, if $l(w_1w_2)=l(w_1)+l(w_2)$, $w_1,w_2\in W$

	$\bullet$ For $x\in \widetilde{X}$ and $\alpha\in \widetilde{\Pi}$:
	\begin{equation}
		\theta_xT_{s_\alpha}-T_{s_\alpha}\theta_{s_\alpha(x)}=(v^2-1)\frac{\theta_x-\theta_{s_\alpha(x)}}{1-\theta_{-\alpha}}
	\end{equation}

	$\bullet$ For $T_\gamma\in \hat{\Gamma}$, $w\in W$, and $x\in \widetilde{X}$:
	\begin{equation}
		\begin{aligned}
			T_\gamma^2&=1\\
			T_\gamma \theta_xT_w T_\gamma&=\theta_{\gamma(x)}T_{\gamma(w)}
		\end{aligned}
	\end{equation}

\end{definition}

From the definition, we have $\mathcal{H}(G^+,v) \simeq \mathcal{H}(G,v)\rtimes \hat{\Gamma}$.

Since $\mathcal{R}(G,T)=\mathcal{R}^\vee (\mathcal{G},\mathcal{T})=(Y,R^\vee,X,R,\Pi^\vee)=(\widetilde{X},\widetilde{R},\widetilde{Y},\widetilde{R}^\vee,\widetilde{\Pi})$, setting $v=p^{1/2}$, from Proposition \ref{IwaHec} there is  an isomorphism \begin{equation}
    \mathcal{H}(G^+,p^{1/2})\simeq \mathcal{H}(\mathcal{J}\backslash\mathcal{G}^+/\mathcal{J})
\end{equation}. 
Define $\mathcal{H}(G,v)=\mathbb{C}[\widetilde{X}]\otimes\mathbb{C}[v,v^{-1}]\otimes\mathbb{C}[W]$ as a subalgebra of $\mathcal{H}(G^+,v)$. Then $\mathcal{H}(G^+,v)=\mathcal{H}(G,v)\rtimes\hat{\Gamma}$ and $\mathcal{H}(G,p^{1/2})\simeq \mathcal{H}(\mathcal{J}\backslash\mathcal{G}/\mathcal{J})$. 

For a Levi subgroup $\mathcal{M}$ normalized by $\gamma$, let $M$ be a Levi subgroup of ${G}$ that is the complex dual of $\mathcal{M}$. Then $\hat{\gamma}$ normalizes $M$. Define ${M}^+={M}\rtimes \hat{\Gamma}$. Denote by $W_M$ the subgroup of $W$ corresponding to $M$. Define $\mathcal{H}(M^+,v)=\mathbb{C}[\widetilde{X}]\otimes\mathbb{C}[v,v^{-1}]\otimes\mathbb{C}[W_{M}]\otimes\mathbb{C}[\hat{\Gamma}]$, a subalgebra of $\mathcal{H}(G^+,v)$. Specializing $v = p^{1/2}$ gives $\mathcal{H}(M^+,p^{1/2})\simeq\mathcal{H}(\mathcal{J}_{\mathcal{M}}\backslash\mathcal{M}^+/\mathcal{J}_{\mathcal{M}})$ and the injective homomorphism (\ref{embedhecke}) $t_{\mathcal{P}'^+}$ is obtained by specializing the variable $v$ to $p^{1/2}$.

\begin{lemma}
	The algebra $\mathbb{C}[\widetilde{X}]^{W^+}\otimes\mathbb{C}[v,v^{-1}]$ is the center of $\mathcal{H}(G^+,v)$.
\end{lemma}
\begin{proof}
	From \cite{Lusztig1989AffineHA}, the central subalgebra of $\mathcal{H}(G,v)$ is $\mathbb{C}[\widetilde{X}]^{W}\otimes\mathbb{C}[v,v^{-1}]$. Taking into account the action of $\hat{\Gamma}$ on $\mathcal{H}(G, v)$, we find that $\mathbb{C}[\widetilde{X}]^{W^+} \otimes \mathbb{C}[v, v^{-1}]$ is the center of $\mathcal{H}(G^+, v)$.
\end{proof}

We  view $\mathscr{A}=\mathbb{C}[\widetilde{X}]\otimes\mathbb{C}[v,v^{-1}]$ as the ring of regular functions on $T\times \mathbb{C}^*$, where $T\simeq \widetilde{Y}\otimes \mathbb{C}^*$ is the algebraic torus with a natural action of $W^+$.

Let $\chi$ be  the central character of the center $\mathscr{A}^{W^+}$ for $\mathcal{H}(G^+,v)$ given by the $W^+$-orbit $(W^+\cdot t,v_0)$, where $(t,v_0)\in T\times \mathbb{C}^*$ is semisimple.
By Schur's lemma, we obtain a canonical partition $\mathrm{Irr}\mathcal{H}(G^+,v)=\sqcup \mathrm{Irr}_{W^+\cdot t,v_0} \mathcal{H}(G^+,v)$ obtained by specifying the action of the centre of $\mathcal{H}(G^+,v)$ on a simple $\mathcal{H}(G^+,v)$-module.

Let $L$ be a finite dimensional  left $\mathcal{H}(G,v)$ module. 
Define $\hat{\gamma} ^* L$ to be the same vector space $L$ with the twisted action $h \cdot l = \hat{\gamma}(h) l$ for $h \in \mathcal{H}(G, v)$ and $l \in L$.

We define
\begin{equation}
	\begin{aligned}
		&\mathscr{I}_{(t,v_0)}^+= \{f\in  \mathscr{A}^{W^+}|f(t,v_0)=0\}\\
		& \mathscr{I}_{(t,v_0)}= \{f\in  \mathscr{A}^{W}|f(t,v_0)=0\}
	\end{aligned}
\end{equation}

Denote by  $\mathcal{H}(G^+,v)-\mathrm{Mod}_{(t,v_0)}$ (resp. $\mathcal{H}(G,v)-\mathrm{Mod}_{(t,v_0)}$) the category of finite dimensional modules of $\mathcal{H}(G^+,v)$ (resp. $\mathcal{H}(G,v)$) annihilated by some power of $\mathscr{I}_{(t,v_0)}^+$ (resp. $\mathscr{I}_{(t,v_0)}$).

We denote by $\widehat{\mathscr{A}^{W^+}_{({t},v_0)}}$ the $\mathscr{I}_{({t},v_0)}^+$-adic completion of $\mathscr{A}^{W^+}$ and $\widehat{\mathscr{A}^{W}_{({t},v_0)}}$ the $\mathscr{I}_{({t},v_0)}$-adic completion of $\mathscr{A}^{W}$.

 We define 
\begin{equation}
	\begin{aligned}
		\widehat{\mathcal{H}}(G^+,v)_{(t,v_0)}&=\widehat{\mathscr{A}^{W^+}_{({t},v_0)}}\otimes_{\mathscr{A}^{W^+}}\mathcal{H}(G^+,v)\\
        \widehat{\mathcal{H}}(G,v)_{(t,v_0)}&=\widehat{\mathscr{A}^{W}_{({t},v_0)}}\otimes_{\mathscr{A}^{W}}\mathcal{H}(G,v)
	\end{aligned}
\end{equation}

If $\hat{\Gamma}\cdot t\not\subseteq G\cdot t$, take $t'=\hat{\gamma}\cdot t$.

\begin{lemma}\label{HeReL}
	The functor $L\mapsto \hat{\gamma}^* L$ is an equivalence between the categories $\mathcal{H}(G,v)-\mathrm{Mod}_{(t,v_0)}$ and $\mathcal{H}(G,v)-\mathrm{Mod}_{(t',v_0)}$.
\end{lemma}
\begin{proof}
	For $f\in \mathscr{A}^{W}$, $l\in \hat{\gamma}^* L$, we have action $f\cdot l=\hat{\gamma}(f)l=f(\hat{\gamma}((t,v_0)))l=f((t',v_0))l$. Moreover,  there is an isomorphism of $\mathcal{H}(G,v)$-module $\hat{\gamma}^* (\hat{\gamma}^* L) \simeq L$, and we obtain the lemma.
\end{proof}

\begin{theorem}
	For a simple module $N \in \mathcal{H}(G^+,v)-\mathrm{Mod}_{(t,v_0)}$, the restriction $\mathrm{Res}^{\mathcal{H}(G^+,v)}_{\mathcal{H}(G,v)} (N) = L_1 \oplus L_2$, where $L_1$, $L_2$ are simple modules of $\mathcal{H}(G,v)$, $L_1 \in \mathcal{H}(G,v)-\mathrm{Mod}_{(t,v_0)}$, $L_2 \in \mathcal{H}(G,v)-\mathrm{Mod}_{(t',v_0)}$, and $L_2\simeq \hat{\gamma}^* L_1$. Hence $N\simeq \mathrm{Ind}^{\mathcal{H}(G^+,v)}_{\mathcal{H}(G,v)} (L_1)\simeq \mathrm{Ind}^{\mathcal{H}(G^+,v)}_{\mathcal{H}(G,v)} (L_2)$.
\end{theorem}

\begin{proof}
For a simple module $N \in \mathcal{H}(G^+,v)-\mathrm{Mod}_{(t,v_0)}$, let $L$ be a simple $\mathcal{H}(G,v)$ submodule of $N$. Let $\chi$ denote the character of the center $\mathscr{A}^{W}$ on the simple module $L$. Hence for $f\in \mathscr{A}^{W^+}$, we have $\chi(f)=f((t,v_0))$. As a result, the central character $\chi$ is given by the $W$-orbit  $(W\cdot t,v_0)$ or  the $W$-orbit  $(W\cdot t',v_0)$. The simple $\mathcal{H}(G,v)$-submodule $L$ lies either in $\mathcal{H}(G,v)-\mathrm{Mod}_{(t,v_0)}$ or in $\mathcal{H}(G,v)-\mathrm{Mod}_{(t',v_0)}$. From \cite[ThmA.6]{zbMATH02096630} and Lemma \ref{HeReL}, we obtain the theorem.
\end{proof}

In fact, we can obtain a finer description of $\widehat{\mathcal{H}}(G^+,v)_{(t,v_0)}$ and its module category.

\begin{lemma}\label{morita}
	The algebra $\widehat{\mathcal{H}}(G^+,v)_{(t,v_0)}$ is isomorphic to $M_2(\widehat{\mathcal{H}}(G,v)_{(t,v_0)})$, the $2\times 2$ matrix algebra with coefficients in $\widehat{\mathcal{H}}(G,v)_{(t,v_0)}$.
\end{lemma}
\begin{proof}
We define \begin{align*}
    &\mathscr{Z}_{(t,v_0)}= \{f\in  \mathscr{A}|f(t,v_0)=0\} \quad \text{a maximal ideal of $\mathscr{A}$} \\
&\widehat{\mathscr{A}}_{\mathscr{Z}_{(t,v_0)}}= \text{$\mathscr{Z}_{(t,v_0)}$-adic completion of $\mathscr{A}$}\\
    &\widehat{\mathscr{A}^+_{({t},v_0)}}= \widehat{\mathscr{A}^{W^+}_{({t},v_0)}}\otimes_{\mathscr{A}^{W^+}}\mathscr{A}\\
    &\widehat{\mathscr{A}_{({t},v_0)}}= \widehat{\mathscr{A}^{W}_{({t},v_0)}}\otimes_{\mathscr{A}^{W}}\mathscr{A}
\end{align*}

We obtain an isomorphism of algebra\begin{align*}
\widehat{\mathscr{A}^+_{({t},v_0)}}&=\bigoplus_{s\in W^+\cdot t} \widehat{\mathscr{A}}_{\mathscr{Z}_{(s,v_0)}}\\
&=\bigoplus_{s\in W\cdot t} \widehat{\mathscr{A}}_{\mathscr{Z}_{(s,v_0)}}\bigoplus_{s\in W\cdot t'}\widehat{\mathscr{A}}_{\mathscr{Z}_{(s,v_0)}}\\
&=\widehat{\mathscr{A}_{({t},v_0)}}\bigoplus \widehat{\mathscr{A}_{({t'},v_0)}}
\end{align*}

    and
\begin{equation}
		\begin{aligned}
		\widehat{\mathcal{H}}(G^+,v)_{(t,v_0)}&=\widehat{\mathscr{A}^{W^+}_{({t},v_0)}}\otimes_{\mathscr{A}^{W^+}}\mathcal{H}(G^+,v)\\
		&= \widehat{\mathscr{A}^+_{({t},v_0)}}\otimes  \mathbb{C}[W]\otimes\mathbb{C}[\hat{\Gamma}]\\
        &=(\widehat{\mathscr{A}_{({t},v_0)}}\otimes  \mathbb{C}[W]\otimes\mathbb{C}[\hat{\Gamma}]) \bigoplus (\widehat{\mathscr{A}_{({t'},v_0)}}\otimes  \mathbb{C}[W]\otimes\mathbb{C}[\hat{\Gamma}]) \\
		&= \widehat{\mathcal{H}} (G,v)_{(t,v_0)}\otimes\mathbb{C}[\hat{\Gamma}]\bigoplus \widehat{\mathcal{H}}(G,v)_{(t',v_0)}\otimes\mathbb{C}[\hat{\Gamma}]
	\end{aligned}
\end{equation}

We define a homomorphism \begin{align*}
    M_2(\widehat{\mathcal{H}}(G,v)_{(t,v_0)}) &\to  \widehat{\mathcal{H}}(G^+,v)_{(t,v_0)}\\
   \begin{pmatrix}
        f_{11} &f_{12}\\
        f_{21} & f_{22}
    \end{pmatrix} &\mapsto f_{11}+f_{12}\cdot (1,\hat{\gamma})+(1,\hat{\gamma})\cdot f_{21}+(1,\hat{\gamma})\cdot f_{22} \cdot (1,\hat{\gamma})
\end{align*}

Since for $f\in \widehat{\mathcal{H}}(G,v)_{(t,v_0)}$, we have $\hat{\gamma}(f)\in \widehat{\mathcal{H}}(G,v)_{(t',v_0)}$, $(1,\hat{\gamma})\cdot f=\hat{\gamma}(f)\cdot (1,\hat{\gamma})$ and $(1,\hat{\gamma})\cdot f\cdot (1,\hat{\gamma})= \hat{\gamma}(f)$. Hence there is an isomorphism $M_2(\widehat{\mathcal{H}}(G,v)_{(t,v_0)})\simeq \widehat{\mathcal{H}}(G^+,v)_{(t,v_0)}$ and this is compatible with the ring structures, and we obtain the lemma.
\end{proof}

\begin{theorem}
	There is an equivalence of categories:\begin{align*}
	    \mathcal{H}(G,v)-\mathrm{Mod}_{(t,v_0)}& \xrightarrow[]{\simeq}\mathcal{H}(G^+,v)-\mathrm{Mod}_{(t,v_0)}\\
        L &\mapsto \mathcal{H}(G^+,v)\otimes_{\mathcal{H}(G,v)}L
	\end{align*}
\end{theorem}
\begin{proof}
	From Lemma \ref{morita}, $\widehat{\mathcal{H}}(G^+,v)_{(t,v_0)}$ and $\widehat{\mathcal{H}}(G,v)_{(t,v_0)}$ are Morita equivalent. In fact, it is given by $L\mapsto \mathcal{H}(G^+,v)\otimes_{\mathcal{H}(G,v)}L$ from $\mathcal{H}(G,v)-\mathrm{Mod}_{(t,v_0)}$ to $\mathcal{H}(G^+,v)-\mathrm{Mod}_{(t,v_0)}$.
\end{proof}

\subsection{Twisted graded Hecke algebra}\label{DefGra}
Let $G^+$ be a complex reductive algebraic group with the identity component $G$. Let $B=TU$ be a Borel subgroup of $G$ with Levi factor $T$ and unipotent radical $U$, and $\mathfrak{t}=\mathrm{Lie}(T)$. We denote the root datum by $\mathcal{R}(G,T)$. Let  $W=N_{G}(T)/T$, $W^+=N_{G^+}(T)/T$, $\mathfrak{R}=N_{G^+}(B,T)/T$ where $N_{G^+}(B,T)=N_{G^+}(B)\cap N_{G^+}(T)$, hence $W^+=W\rtimes \mathfrak{R}$, $W$ is Weyl group of $\mathcal{R}(G,T)$.
Let $\mathfrak{r}$ be an indeterminate, identified with the coordinate function on $\mathbb{C}$.
Let $\mathfrak{t}^*$ be the dual space of the Lie algebra $\mathfrak{t}=\mathrm{Lie}(T)$ and $S(\mathfrak{t}^*)$ be the symmetric algebra of $\mathfrak{t}^*$.

\begin{proposition}\label{GradedHecke}\cite[2.2]{Aubert2016GradedHA}
	There exists a unique associative algebra structure on $ S(\mathfrak{t}^*)\otimes\mathbb{C}[W^+]\otimes \mathbb{C}[\mathfrak{r}] $ such that 
	
	$\bullet$ the twisted group algebra $\mathbb{C}[W^+]$ is embedded as subalgebra
	
	$\bullet$ the algebra $S(\mathfrak{t}^*)\otimes \mathbb{C}[\mathfrak{r}]$ of polynomial functions on $\mathfrak{t}\oplus \mathbb{C} $ is embedded as a subalgebra.
	
	$\bullet$ $\mathbb{C}[\mathfrak{r}] $ is central
	
	$\bullet$ $N_{\alpha}\xi-s_\alpha(\xi)N_{\alpha}=2\mathfrak{r}(\xi-s_\alpha(\xi))/\alpha$ for $\xi\in S(\mathfrak{t}^*)$ and simple roots $\alpha$.
	
	$\bullet$ $N_\tau \xi N_\tau^{-1}=\tau(\xi)$ for $\xi\in S(\mathfrak{t}^*) $ and $\tau\in \mathfrak{R}$.
	
\end{proposition}

We denote the algebra of Proposition \ref{GradedHecke} by $\mathbb{H}(G^+)$. Define $\mathbb{H}(G) = S(\mathfrak{t}^*) \otimes \mathbb{C}[W] \otimes \mathbb{C}[\mathfrak{r}]$ as a subalgebra of $\mathbb{H}(G^+)$; then $\mathbb{H}(G^+) = \mathbb{H}(G) \rtimes \mathfrak{R}$.

\begin{lemma}
	The algebra $S(\mathfrak{t}^*)^{W^+}\otimes\mathbb{C}[\mathfrak{r}] $ is the center of $\mathbb{H}(G^+)$.
\end{lemma}
\begin{proof}
	By a result of Lusztig, the center of $\mathbb{H}(G)$ is $S(\mathfrak{t}^*)^{W}\otimes\mathbb{C}[\mathfrak{r}] $. Taking into account the action of $\mathfrak{R}$ on $\mathbb{H}(G)$, we obtain that $S(\mathfrak{t}^*)^{W^+}\otimes\mathbb{C}[\mathfrak{r}] $ is the center of $\mathbb{H}(G^+)$.
\end{proof}

Now we return to our specific case $G^+=G\rtimes \hat{\Gamma}$. Let $t$ be a semisimple element in $G$ such that $\hat{\Gamma}\cdot t\subseteq G\cdot t$. Let $t=t_ct_h$ be its polar decomposition,  where $t_c=t|t|^{-1}$ is the compact part and $t_h=|t|$ is the hyperbolic part. Then $Z_{G^+}(t_c)^0$ has root datum $(\widetilde{X},R_{t_c},\widetilde{Y},R_{t_c}^\vee, \Pi_{t_c})$ and we have a graded Hecke algebra $\mathbb{H}(Z_{G^+}(t_c))$.

From now on, assume $v_0>0$. Denote by $\bar{\chi}$ the central character of the center of $\mathbb{H}(Z_{G^+}(t_c))$ given by the $(W^+_{t_c}\cdot \mathrm{log}t_h, \mathrm{log}v_0)$, and denote by $\mathbb{H}(Z_{G^+}(t_c))-\mathrm{Mod}_{\bar{\chi}}$ the category of finite dimensional modules of $\mathbb{H}(Z_{G^+}(t_c))$ on which the center acts via the central character $\bar{\chi}$.

\begin{theorem}\label{Ind22}[\cite{aubert2017affine}]
	There is an equivalence $\Theta$ between the categories  $\mathcal{H}(G^+,v)-\mathrm{Mod}_{\chi}$ and $\mathbb{H}(Z_{G^+}(t_c))-\mathrm{Mod}_{\bar{\chi}}$, where $\chi$ is the central character of $\mathcal{H}(G^+,v)$ given by $(W^+\cdot t,v_0)$ and $\bar{\chi}$ is the central character of $\mathbb{H}(Z_{G^+}(t_c))$ given by the $(W^+_{t_c}\cdot \mathrm{log}t_h, \mathrm{log}v_0)$. 
	
	The functor $\Theta$ is compatible with parabolic induction, in the following sense. let $M$ be a Levi subgroup of ${G}$ which is normalized by $\hat{\gamma}$, and define ${M}^+={M}\rtimes \hat{\Gamma}$. Then $\Theta \circ \mathrm{Ind}_{\mathcal{H}(M^+,v)}^{\mathcal{H}(G^+,v)}= \mathrm{Ind}_{\mathbb{H}(Z_{M^+}(t_c))}^{\mathbb{H}(Z_{G^+}(t_c))} \circ \Theta$, where $\mathrm{Ind}_{\mathcal{H}(M^+,v)}^{\mathcal{H}(G^+,v)}$(resp. $\mathrm{Ind}_{\mathbb{H}(Z_{M^+}(t_c))}^{\mathbb{H}(Z_{G^+}(t_c))}$ ) is given by $\mathcal{H}(G^+,v)\otimes_{\mathcal{H}(M^+,v)}(-)$ (resp. $\mathbb{H}(Z_{G^+}(t_c))\otimes_{\mathbb{H}(Z_{M^+}(t_c))}(-)$).

\end{theorem}

\begin{remark}\label{SplitRed}
	For a semisimple element $t\in \GL_n(\mathbb{C})$ satisfying $\hat{\Gamma}\cdot t\subseteq G\cdot t$, the short exact sequence \begin{equation}\label{DisSplitEqu}
		1\to Z_G(t_c) \to Z_{G^+}(t_c) \to \hat{\Gamma}_{t_c} \to 1
	\end{equation}
	oes not split in some cases. However, there is always a homomorphism  $\hat{\Gamma}_{t_c}\to \mathrm{Out}(Z_G(t_c))\simeq\mathrm{Aut}(\mathcal{B}(Z_G(t_c)))$ \cite{gaetz2024disconnectedreductivegroups}, which induces an isomorphism $\hat{\Gamma}_{t_c}\simeq \mathfrak{R}$. Consequently, we obtain  $\mathbb{H}(Z_{G^+}(t_c))=\mathbb{H}(Z_{G}(t_c))\rtimes\mathfrak{R}$. We will discuss this in more detail in the section of Whittaker normalization.
\end{remark}

\section{Representations of twisted graded Hecke algebra}
We use a superscript $*$ to denote the dual of a vector space or of a local system.   If a group $G$ acts on a vector space $V$, we denote the space of $G$-invariant vectors by $V^G$. We write $S(V)=\bigoplus_{j \ge 0}S^j(V)$ for the symmetric algebra of $V$.
\subsection{Equivariant derived category}

Let $A$ be a complex linear algebraic group and let $X$ be an $A$-variety, that is a quasiprojective complex algebraic variety with an algebraic $A$-action. We consider  sheaves of $\mathbb{C}$-vector spaces on X with respect to the analytic topology. We denote by $D(X)$ the bounded derived category of constructible sheaves on $X$ and $D_A(X)$ the $A$-equivariant bounded derived category as defined in \cite{EquBeL}. This is a triangulated category with a $t$-structure and there is a forgetful functor $\operatorname{For}: D_A(X) \to D(X)$.

We now recall the construction $D_A(X)$. Let $I\subset \mathbb{Z}$ be a finite interval, and let $k\geq|I|$. Let $Sf_k(A)$ be the category  whose objects are smooth free $A$-varieties $Z_k$ that are $k$-acyclic
and whose morphisms are  smooth free $A$-morphisms. For a chosen $Z_k\in Sf_k(A)$, we have the diagram $X \stackrel{p_k}{\longleftarrow} X\times Z_k \stackrel{q_k}{\longrightarrow} A\backslash (X\times Z_k)$. An object $\mathcal{I}$ in $D_A^I(X,Z_k)$ consists of a triple $(\mathcal{I}_0,\mathcal{I}_k,\beta)$ with  $\mathcal{I}_0\in D^I(X)$, $\mathcal{I}_k\in D^I(A\backslash (X\times Z_k))$, and an isomorphism  $\beta:p_k^*(\mathcal{I}_0)\simeq q_k^*(\mathcal{I}_k)$. 

A morphism $\alpha:(\mathcal{I}_0,\mathcal{I}_k,\beta_I)\mapsto (\mathcal{H}_0,\mathcal{H}_k,\beta_H) $ in $D^I_{A}(X,Z_k)$ is a pair $\alpha=(\alpha_X,\bar{\alpha})$ where $\alpha_X:\mathcal{I}_0 \mapsto \mathcal{H}_0$ and $\bar{\alpha}:  \mathcal{I}_k\mapsto \mathcal{H}_k$ satisfy $\beta_H\cdot p_k^*(\alpha_X)=q_k^*(\bar{\alpha})\cdot \beta_I$

For another object $Z_{k'} \in Sf_{k'}(A)$ with $k'\geq|I|$, there is a canonical equivalence of categories $D^I_{A}(X,Z_k)\simeq D^I_{A}(X,Z_{k'})$. This yields a well defined category $D^I_{A}(X)$.   Finally we define $D_{A}(X)=\lim_I D^I_{A}(X) $.

\begin{remark}\label{ExsitEqu}
	Let $A\subseteq \GL_n(\mathbb{C})$ be a closed subgroup, and let $V_k$ be the Stiefel variety of $n$-tuples of linearly independent vectors in $\mathbb{C}^{n+k+1}$. Then $V_k\in Sf_k(A)$ . 
\end{remark}

\begin{remark}\label{Recover}
	In fact, as explained in \cite[2.3.2]{EquBeL}, for a given  $Z_k\in Sf_k(A)$ with $k\geq |I|$, we can recover $D^I_{A}(X)$ from a full subcategory of $D^I(A\backslash (X\times Z_k))$. 
\end{remark}
Let $X$, $Y$ be $A$-varieties and let $f:X\to  Y$ be an $A$-morphism. Then we have functors:
\begin{equation*}
	f^*,f^! : D_A(Y)\to D_A(X) \quad f_*,f_! : D_A(X)\to D_A(Y)
\end{equation*}

If $A' \hookrightarrow A$ is a subgroup, then an $A$-variety $X$ can also be regarded as an $A'$-variety, and there is a restriction functor $\operatorname{Res}^A_{A'} : D_A(X) \to D_{A'}(X)$.  The four functors above commute with the restriction functor $\Res^A_{A'}$ in the sense that there are canonical isomorphisms $\Res^A_{A'}\cdot f^*\simeq f^*\cdot \Res^A_{A'}$, and similarly for $f^!$, $f_*$, $f_!$.

We have the functor $\otimes$ and $R\mathscr{Hom}$ in the equivariant case. 
We also have the equivariant dualizing complex $\omega_X=a_X^!\mathbb{C}_{\mathrm{pt}}$ $(a_X:X\to \mathrm{pt})$ and Verdier duality $D(\mathcal{I}):=R\mathscr{H}om(\mathcal{I},\omega_X)$. The following natural functorial isomorphisms hold: \begin{align}
    R\mathscr{H}om(\mathcal{I}_1\otimes \mathcal{I}_2,\mathcal{I}_3)&\simeq R\mathscr{H}om(\mathcal{I}_1,R\mathscr{H}om(\mathcal{I}_2,\mathcal{I}_3))\label{DefHomo}\\
    f^*(\mathcal{I}_1\otimes \mathcal{I}_2)&\simeq f^*\mathcal{I}_1\otimes f^*\mathcal{I}_2\nonumber\\
    D\cdot f^*&\simeq f^!\cdot D\nonumber\\
    D\cdot f_*&\simeq f_!\cdot D \nonumber
\end{align}

For any $j\in\mathbb{Z}$, define a functor $H^j_A(X,-)$ from $D_A(X)$ to the category of finite dimensional $\mathbb{C}$-vector spaces. Let $\mathcal{I}=(\mathcal{I}_0,\mathcal{I}_k,\beta_I)\in D_A(X)$. Choose $k\geq j$, and set \begin{align*}
    H^j_A(X,\mathcal{I}):=H^j(A\backslash (X\times Z_k),\mathcal{I}_k)
\end{align*} The definition is independent of the choices. Moreover, if $f:X\to  Y$ is a $A$-morphism, then $H^j_A(X,\mathcal{I})=H^j_A(Y,f_*\mathcal{I})$.

 We define the equivariant homology by \begin{align*}
	H_j^A(X,\mathcal{I})=H_A^{j-2\mathrm{dim}X}(X,\mathcal{I}\otimes\omega_X))
\end{align*}

Let $\mathcal{L}$ be an $A$-equivariant local system on $X$. We have $\mathcal{L}^*=R\mathscr{H}om(\mathcal{L},\mathbb{C}_X)$. From \eqref{DefHomo}, we have $R\mathscr{H}om(\mathcal{L\otimes\omega}_X,\omega_X)\simeq  R\mathscr{H}om(\mathcal{L},\mathbb{C}_X)$. We obtain \begin{align*}
\mathcal{L\otimes\omega}_X\simeq D(\mathcal{L}^*)
\end{align*} Hence  \begin{align*}
	H_j^A(X,\mathcal{L})=H_A^{j-2\mathrm{dim}X}(X,D(\mathcal{L}^*))
\end{align*}
    
For $\mathcal{I}$, $\mathcal{I}'\in D_A(X)$, we define 
\begin{equation}\label{DefCup}
	\oplus_{j+j'=i}H^j_A(X,\mathcal{I})\times H^{j'}_A(X,\mathcal{I}')\to H^i_A(X,\mathcal{I}\otimes\mathcal{I}')
\end{equation}

It is given by the cup product $H^j(A\backslash (X\times Z_k),\mathcal{I}_k)\times H^{j'}(A\backslash (X\times Z_k),\mathcal{I}'_k)\to H^i(A\backslash (X\times Z_k),\mathcal{I}_k\otimes\mathcal{I}'_k)$ with large enough $k$.

We have canonically 	
\begin{equation}\label{EquHomDes}
	\mathrm{Hom}_{D_A(X)}^j(\mathcal{I},\mathcal{I}')=H^j_A(X,D(D\mathcal{I}'\otimes\mathcal{I}))
\end{equation}	
and
$H^j_A(X,D(D\mathcal{I}\otimes\mathcal{I}'))=H^j(A\backslash (X\times Z_k),D(D(\mathcal{I}'_k)\otimes \mathcal{I}_k))=\mathrm{Hom}(H_c^{-j}(A\backslash (X\times Z_k),D(\mathcal{I}'_k)\otimes \mathcal{I}_k),\mathbb{C})$.

Using \eqref{DefCup}, we note that $H^*_A(X,D(D\mathcal{I}'\otimes\mathcal{I}))$ is a $H^*_A(X,\mathbb{C})$-module.

Taking $\mathcal{I}=\mathcal{I}'=\mathbb{C}_X$ in \eqref{EquHomDes}, we deduce
\begin{equation}\label{ConstHom}
	\mathrm{Hom}_{D_A(X)}^j(\mathbb{C}_X,\mathbb{C}_X)=H^j_A(X,\mathbb{C})
\end{equation}

Hence 	$\mathrm{Hom}_{D_A(X)}^*(\mathcal{I},\mathcal{I}')$ can be viewed as $\mathrm{Hom}_{D_A(X)}^*(\mathbb{C}_X,\mathbb{C}_X)$-module as follows: the product of $\phi: \mathbb{C}_X\to \mathbb{C}_X[n]$ with $\psi: \mathcal{I}\to \mathcal{I}'[m]$ is $\phi\otimes\psi:\mathcal{I}\to \mathcal{I}'[n+m]$  using $\mathcal{I}=\mathbb{C}_X\otimes\mathcal{I}$, $\mathcal{I}'=\mathbb{C}_X\otimes\mathcal{I}'$. Then the isomorphism \eqref{EquHomDes} is an isomorphism of $H^*_A(X,\mathbb{C})$-modules, where the left-hand side is a module via the identification $H^*_A(X,\mathbb{C})\cong\operatorname{Hom}^*_{D_A(X)}(\mathbb{C}_X,\mathbb{C}_X)$ and the right-hand side via the cup product.

Let $A^+$ be an algebraic group, and let $A$ be a connected normal subgroup with $\mathfrak{R} :=A^+/A= \langle1,\tau\rangle$. For an $A^+$-variety $X$, we have a functor $D_{A^+}(X)\stackrel{\mathrm{Res}^{A^+}_A}{\longrightarrow} D_{A}(X)$.

We define the functor $\tau^*$ on $D_{A}(X)$. Take a representative $\bar{\tau}$ of $\tau$ in $A^+$. Since $X$ is $A^+$-variety, $\bar{\tau} : X\to X $ is twisted $A$-equivariant, i.e., $\bar{\tau}(ax)=\mathrm{Ad}_{\bar{\tau}}(a)\bar{\tau}(x)$ where $a\in A$, $x\in X$. Take $Z_k^+\in Sf_k(A^+)$, hence for $a
\in A$, $z\in Z_k^+$ we have ${\bar{\tau}}(az)=\mathrm{Ad}_{\bar{\tau}}(a)\bar{\tau}(z)$. Then we have $\bar{\tau}$ action on $X\times Z_k$ which is twisted $A$-equivariant. It induces a $\bar{\tau}$ action on $A\backslash (X\times Z_k)$. If we change another representative $\bar{\tau}'$, we have $\bar{\tau}'\cdot\bar{\tau}^{-1}=a\in A$, hence $\bar{\tau}'\cdot\bar{\tau}^{-1}$ induces the identity map on $A\backslash (X\times Z_k)$,  therefore the action $\bar{\tau}:A\backslash (X\times Z_k)\to A\backslash (X\times Z_k)$ does not depend on the choice of representative. Finally this $\bar{\tau}$ defines a functor $\tau^*$ on $D^I(A\backslash (X\times Z_k))$ by pullback the complexes via $\bar{\tau}$ action on the space. From \cite[6.5]{EquBeL}, it means we can always find a $k$-acyclic compatible resolution of $\bar{\tau}$, hence $\tau^*$ is well defined on $D_{A}(X)$.

\begin{lemma}\label{Inv}
There is an algebraic isomorphism\begin{align*}
    \mathrm{Hom}^*_{D_{A^+}(X)}(\mathcal{F},\mathcal{F}')\simeq\mathrm{Hom}_{D_{A}(X)}^*(\mathrm{Res}^{A^+}_A(\mathcal{F}),\mathrm{Res}^{A^+}_A(\mathcal{F}'))^\tau
\end{align*}
\end{lemma}
\begin{proof}
	Take $Z_k^+\in Sf_k(A^+)$, then $Z_k^+\in Sf_k(A)$ and we have the projection $\pi: A\backslash (X\times Z_k^+)\to A^+\backslash (X\times Z_k^+)$. We have the following commutative diagram
	\begin{equation*}
		\xymatrix{
			D^I_{A^+}(X)
			\ar@{^{(}->}[r]
			\ar[d]_{\mathrm{Res}^{A^+}_A} & D^I(A^+\backslash (X\times Z_k^+)) \ar[d]^{\pi ^*} \\
			D^I_{A}(X)  \ar@{^{(}->}[r] & D^I(A\backslash (X\times Z_k^+))
		}
	\end{equation*}
Let $\mathcal{F},\mathcal{F}'\in D_{A^+}(X)$ and consider a morphism 
   $f\in \mathrm{Hom}_{D_{A}(X)}^*(\mathrm{Res}^{A^+}_A(\mathcal{F}),\mathrm{Res}^{A^+}_A(\mathcal{F}'))$.
Via the bottom horizontal embedding, $f$ corresponds to a morphism 
    $f_k\in \mathrm{Hom}^*_{D^I(A\backslash (X\times Z_k^+))}(\pi^*\mathcal{F}_k,\pi^*\mathcal{F}'_k)$ for any $I$. 
	
	Since we have an action $\tau$ on $A\backslash (X\times Z_k^+)$ by left multiplication, we have $\tau^*$ action on $\mathrm{Hom}_{D^I(A\backslash (X\times Z_k^+))}(\pi^*\mathcal{F}_k,\pi^*\mathcal{F}'_k)$: 
	\begin{align}
		\mathrm{Hom}_{D^I(A\backslash (X\times Z_k^+))}(\pi^*\mathcal{F}_k,\pi^*\mathcal{F}'_k)&\to \mathrm{Hom}_{D^I(A\backslash (X\times Z_k^+))}(\tau^*\pi^*\mathcal{F}_k,\tau^*\pi^*\mathcal{F}'_k)\label{DefHomAct}\\
		&\simeq\mathrm{Hom}_{D^I(A\backslash (X\times Z_k^+))}(\pi^*\mathcal{F}_k,\pi^*\mathcal{F}'_k)\nonumber
	\end{align} 
	where the second identity is induced by $\pi\circ\tau=\pi $.
	
	Since $\pi: A\backslash (X\times Z_k^+)\to A^+\backslash (X\times Z_k^+)$ is the quotient by $\mathfrak{R}$ and $\mathfrak{R}$ action on  $A\backslash (X\times Z_k^+)$ is free,  we have a canonical equivalence $D^I(A^+\backslash (X\times Z_k^+))\simeq D^I_{\mathfrak{R}}(A\backslash (X\times Z_k^+))$. Since $\mathfrak{R}$ is discrete group, we have a canonical equivalence  $D^I_{\mathfrak{R}}(A\backslash (X\times Z_k^+))\simeq D^I(Sh_{\mathfrak{R}}(A\backslash (X\times Z_k^+)))$. Therefore $\pi^*$ induces the equivalence of categories: 
	    $D^I(A^+\backslash (X\times Z_k^+))\simeq D^I(Sh_{\mathfrak{R}}(A\backslash (X\times Z_k^+)))$

	We define the category $D^I(A\backslash (X\times Z_k^+))^{\mathfrak{R}}$ of $\mathfrak{R}$-equivariant objects in $D^I(A\backslash (X\times Z_k^+))$ as follows: an object in $D^I(A\backslash (X\times Z_k^+))^{\mathfrak{R}}$ is a pair $(\mathcal{I},\kappa_{\tau})$ where $\mathcal{I}\in D^I(A\backslash (X\times Z_k^+))$ and $\kappa_{\tau}:\mathcal{I}\to \tau^*\mathcal{I}$ is an isomorphism satisfying $\kappa_\tau\circ\tau^*(\kappa_\tau)=1$; a morphism in $D^I(A\backslash (X\times Z_k^+))^{\mathfrak{R}}$ is a morphism of $D^I(A\backslash (X\times Z_k^+))$, $\phi : \mathcal{I}\to \mathcal{I}'$ such that the following diagram commutes:
	\begin{equation*}
		\xymatrix{
			\mathcal{I}
			\ar[r]^{\kappa_\tau}
			\ar[d]_{\phi} & \tau^*\mathcal{I}  \ar[d]^{\tau^*(\phi)} \\
			\mathcal{I}'
			\ar[r]^{\kappa'_\tau} & 
			\tau^*\mathcal{I}'
		}
	\end{equation*}
	
	There is natural functor $D^I(Sh_{\mathfrak{R}}(A\backslash (X\times Z_k^+)))\to D^I(A\backslash (X\times Z_k^+))$. 
     By \cite{EquTria}, this functor is an equivalence of categories  $D^I(Sh_{\mathfrak{R}}(A\backslash (X\times Z_k^+)))\simeq D^I(A\backslash (X\times Z_k^+))^{\mathfrak{R}}$.
 Equivalently, by \cite{EquDes}, the pullback functor $\pi^*$ directly gives an equivalence  $D^I(A^+\backslash (X\times Z_k^+))\simeq D^I(A\backslash (X\times Z_k^+))^{\mathfrak{R}}$.

	Take $\mathcal{I}=\pi^*\mathcal{F}_k$, $\mathcal{I}'=\pi^*\mathcal{F}'_k[j]$, then $\kappa_\tau$ and $\kappa'_\tau$ are identity map. Let $f_k\in \mathrm{Hom}_{D^I(A^+\backslash (X\times Z_k^+))}(\pi^*\mathcal{F}_k,\pi^*\mathcal{F}'_k[j])=\mathrm{Hom}^j_{D^I(A^+\backslash (X\times Z_k^+))}(\pi^*\mathcal{F}_k,\pi^*\mathcal{F}'_k)$, then $\pi^*(f)$ induces the algebraic isomorphism  $\mathrm{Hom}^*_{D_{A^+}(X)}(\mathcal{F},\mathcal{F})\simeq\mathrm{Hom}^*_{D_{A}(X)}(\mathrm{Res}^{A^+}_A(\mathcal{F}),\mathrm{Res}^{A^+}_A(\mathcal{F}))^\tau$.
\end{proof}

In fact, we can define a descent datum to recover $D_{A^+}(X)$ inside $D_{A}(X)$.

From the definition of the functor $\tau^*$, we have $\tau^*\circ\tau^*=1$. We define the category $D_{A}(X)^{\mathfrak{R}}$ as follows: an object in $D_{A}(X)^{\mathfrak{R}}$ is a pair $(\mathcal{I},\kappa_{\tau})$ where $\mathcal{I}\in D_{A}(X)$ and $\kappa_{\tau}:\mathcal{I}\to \tau^*\mathcal{I}$ is an isomorphism satisfying $\kappa_\tau\circ\tau^*(\kappa_\tau)=1$; a morphism in $D_{A}(X)^{\mathfrak{R}}$ is a morphism of $D_{A}(X)$, $\phi : \mathcal{I}\to \mathcal{I}'$ such that the following diagram commutes:
	\begin{equation*}
		\xymatrix{
			\mathcal{I}
			\ar[r]^{\kappa_\tau}
			\ar[d]_{\phi} & \tau^*\mathcal{I}  \ar[d]^{\tau^*(\phi)} \\
			\mathcal{I}'
			\ar[r]^{\kappa'_\tau} & 
			\tau^*\mathcal{I}'
		}
	\end{equation*}
The shift functor and distinguished triangles on $D_A(X)^{\mathfrak{R}}$ are defined as in $D_A(X)$.	

\begin{theorem}
    The functor \begin{align*}
        {\mathrm{Res}^{A^+}_A} : D_{A^+}(X)\to D_{A}(X)^{\mathfrak{R}}
    \end{align*}
    is an equivalence of categories. The functor ${\mathrm{Res}^{A^+}_A}$ is a triangulated functor, and $D_{A}(X)^{\mathfrak{R}}$ is a triangulated category.
\end{theorem}
\begin{proof}
    We use the notation and the proof in Lemma \ref{Inv}. The category  $D_{A^+}(X)$ is the pullback of the following commutative diagram \begin{equation*}
		\xymatrix{
			D^I_{A^+}(X)
			\ar@{^{(}->}[r]
			\ar[d]_{\mathrm{Res}^{A^+}_A} & D^I(A^+\backslash (X\times Z_k^+))\simeq D^I(A\backslash (X\times Z_k^+))^{\mathfrak{R}} \ar[d]^{\pi ^*} \\
			D^I_{A}(X)  \ar@{^{(}->}[r] & D^I(A\backslash (X\times Z_k^+))
		}
	\end{equation*}

The statement then follows from the definition and Remark \ref{Recover}.
\end{proof}

Let $f_1$, $f_2\in\mathrm{Hom}_{D_{A}(X)}^*(\mathrm{Res}^{A^+}_A(\mathcal{F}),\mathrm{Res}^{A^+}_A(\mathcal{F}'))$. Because $\tau^*$ is a functor,  we have $\tau^*(f_1\circ f_2)=\tau^*(f_1)\circ\tau^*(f_2)$. Hence the  $\tau$-action on $\mathrm{Hom}_{D_{A}(X)}^*(\mathrm{Res}^{A^+}_A(\mathcal{F}),\mathrm{Res}^{A^+}_A(\mathcal{F}'))$ defined in \eqref{DefHomAct} satisfies $\tau(f_1\cdot f_2)=\tau(f_1)\cdot\tau(f_2)$.

Taking $\mathcal{F}=\mathcal{F}'=\mathbb{C}_X$, we obtain a $\tau$-action on $\mathrm{Hom}_{D_A(X)}^*(\mathbb{C}_X,\mathbb{C}_X)$ and $\tau(g_1\cdot g_2)=\tau(g_1)\cdot\tau(g_2)$ for $g_1$, $g_2\in \mathrm{Hom}_{D_A(X)}^*(\mathbb{C}_X,\mathbb{C}_X)$.

For any $g\in \mathrm{Hom}_{D_A(X)}^*(\mathbb{C}_X,\mathbb{C}_X)$, $f\in \mathrm{Hom}_{D_{A}(X)}^*(\mathrm{Res}^{A^+}_A(\mathcal{F}),\mathrm{Res}^{A^+}_A(\mathcal{F}'))$, we have $\tau^*(g\otimes f)=\tau^*(g)\otimes\tau^*(f)$.  Consequently, the $\tau$-action is compatible with $\mathrm{Hom}_{D_A(X)}^*(\mathbb{C}_X,\mathbb{C}_X)$-module structure on $\mathrm{Hom}_{D_{A}(X)}^*(\mathrm{Res}^{A^+}_A(\mathcal{F}),\mathrm{Res}^{A^+}_A(\mathcal{F}'))$ which means that $\tau(g\cdot f)=\tau(g)\cdot\tau(f)$.

For $H^*_A(X,\mathbb{C})$, we can define a $\tau$-action via the homomorphism $\tau^*:H^*_A(X,\mathbb{C})\to H^*_A(X,\tau^*\mathbb{C})$.

Note that $\tau^*$ sends $\phi\in \mathrm{Hom}_{D_A(X)}^*(\mathbb{C}_X,\mathbb{C}_X)$ to $\tau^*(\phi)\in \mathrm{Hom}_{D_A(X)}^*(\tau^*\mathbb{C}_X,\tau^*\mathbb{C}_X)$.Taking $p:X\to \mathrm{pt}$ we have 
\begin{align}
	\mathrm{Hom}_{D_A(X)}^*(\tau^*\mathbb{C}_X,\tau^*\mathbb{C}_X)&\simeq \mathrm{Hom}_{D_A(X)}^*(\mathbb{C}_X,\tau_*\tau^*\mathbb{C}_X)\nonumber\\
	&\simeq \mathrm{Hom}_{D_A(X)}^*(p^*\mathbb{C}_{\mathrm{pt}},\tau_*\tau^*\mathbb{C}_X)\nonumber\\
	&\simeq\mathrm{Hom}_{D_A(\mathrm{pt})}^*(\mathbb{C}_{\mathrm{pt}},p_*\tau_*\tau^*\mathbb{C}_X)\nonumber\\
	&\simeq H^*_A(X,\tau^*\mathbb{C}_X)\label{ActPt}
\end{align} 

Therefore \eqref{ConstHom} $\mathrm{Hom}_{D_A(X)}^*(\mathbb{C}_X,\mathbb{C}_X)=H^*_A(X,\mathbb{C})$ is compatible with the $\tau$-action.

Let $A$ be a reductive group. We write $H_A^*$ instead of $H_A^*(\mathrm{point})$ where the point is regarded as an $A$-variety in the obvious way. From \cite[1.11]{lusztig1}, we have:
\begin{equation*}
	H_A^*= S(\mathfrak{a}^*)^A
\end{equation*}
where $\mathfrak{a}$ is the Lie algebra of $A$, and $S(\mathfrak{a}^*)^A$ denotes the $A$-invariant subalgebra of the symmetric algebra $S(\mathfrak{a}^*)$. For any semisimple element $s\in \mathfrak{a}$, we have the algebra homomorphism $\chi_s:S(\mathfrak{a}^*)\to \mathbb{C}$  given by evaluating a polynomial function at $s$. We denote $\mathbb{C}_s$ as  $H^*_A$-algebra via the algebra homomorphism $\chi_s: H_A^*\to \mathbb{C}$.

Recall that $A$ is a connected normal subgroup of $A^+$ with $\mathfrak{R} :=A^+/A= \langle1,\tau\rangle$. From \eqref{ActPt}, we have the $\tau$-action on $H^*_A\simeq S(\mathfrak{a}^*)^A$. Let $T_A$ be a maximal torus of $A$, with Lie algebra $\mathfrak{t}_A\hookrightarrow\mathfrak{a}$. Then the diagram 

\begin{equation}\label{ActPtDia}
	\begin{tikzcd}
		H^*_A\simeq S(\mathfrak{a}^*)^A \ar[r,"\tau"] \ar[d]& H^*_A\simeq S(\mathfrak{a}^*)^A\ar[d]\\
		S(\mathfrak{t}_A^*) \ar[r,"(d)"] & S(\mathfrak{t}_A^*)
	\end{tikzcd}
\end{equation}

is commutative, where the left vertical map is induced by the natural map $\mathfrak{t}_A\hookrightarrow\mathfrak{a}$ and $(d)$ is induced by $H^*_{T_A}\stackrel{\tau}{\longrightarrow} H^*_{T_A}$ which is $S(\mathfrak{t}_A^*) \to S(\mathfrak{t}_A^*)$ induced by the adjoint action $\mathfrak{R}$ on $\mathfrak{t}_A$.

We have an algebra homomorphism $\mathrm{Hom}_{D_A(\mathrm{pt})}^*(\mathbb{C},\mathbb{C})\to\mathrm{Hom}_{D_A(X)}^*(\mathbb{C}_X,\mathbb{C}_X)$ induced by $X\stackrel{p}{\longrightarrow}\mathrm{pt}$, and since $p^*\circ\tau^*=\tau^*\circ p^*$, this algebra homomorphism is compatible with the $\tau$-action.

Note that $H^*_A\simeq \mathrm{Hom}_{D_A(\mathrm{pt})}^*(\mathbb{C},\mathbb{C})$, hence combining  all above we get a $H^*_A$-module structure on $\mathrm{Hom}_{D_{A}(X)}^*(\mathrm{Res}^{A^+}_A(\mathcal{F}),\mathrm{Res}^{A^+}_A(\mathcal{F}'))$ and 
\begin{proposition}\label{ModAct}
	The $\tau$ action on $H^*_A$, $\mathrm{Hom}_{D_{A}(X)}^*(\mathrm{Res}^{A^+}_A(\mathcal{F}),\mathrm{Res}^{A^+}_A(\mathcal{F}'))$ and the $H^*_A$-module structure of $\mathrm{Hom}_{D_{A}(X)}^*(\mathrm{Res}^{A^+}_A(\mathcal{F}),\mathrm{Res}^{A^+}_A(\mathcal{F}'))$ are compatible.  Explicitly, for all $h_1,h_2,h\in H^*_A$ $f_1,f_2,f\in \mathrm{Hom}_{D_{A}(X)}^*(\mathrm{Res}^{A^+}_A(\mathcal{F}),\mathrm{Res}^{A^+}_A(\mathcal{F}'))$, we have   $\tau(f_1\cdot f_2)=\tau(f_1)\cdot\tau(f_2)$, $\tau(h_1\cdot h_2)=\tau(h_1)\cdot\tau(h_2)$, $\tau(h\cdot f)=\tau(h)\tau(f)$, where $\tau$ action is given by \eqref{ActPtDia} and \eqref{Inv}.
\end{proposition}

\subsection{Geometry of graded Hecke algebra}

We preserve the setup of section \ref{DefGra}. Let $G^+$ be a complex reductive algebraic group with identity component $G$. Let $B=TU$ be a Borel subgroup of $G$ with Levi factor $T$ and unipotent radical $U$, and $\mathfrak{t}=\mathrm{Lie}(T)$. We denote the root datum by $\mathcal{R}(G,T)$. Let  $W=N_{G}(T)/T$, $W^+=N_{G^+}(T)/T$, $\mathfrak{R}=N_{G^+}(B,T)/T$, hence $W^+=W\rtimes \mathfrak{R}$, and $W$ is the Weyl group of $\mathcal{R}(G,T)$. Finally, we have a twisted graded Hecke algebra $\mathbb{H}(G^+)$.

Consider the variety

$\dot{\mathfrak{g}}^+ = \{(x,gB)\in \mathfrak{g} \times G^+/B :\mathrm{Ad}(g^{-1})x\in \mathfrak{b}\}$

We have a natural $G^+\times \mathbb{C}^*$-action on $\dot{\mathfrak{g}}^+$:
\begin{align}
	(g_1,\lambda):(x,gB)\mapsto (\lambda^{-2}\mathrm{Ad}(g_1)x,g_1gB)
\end{align}

Let $\mathbf{E}$ be the set of all isomorphism classes of $G^+\times \mathbb{C}^*$-equivariant line bundles over $G^+/B$.

\begin{lemma}\label{HeckeMap}
There is a $\mathbb{C}$-linear isomorphism\begin{align*}
    \mathbb{C}\otimes \mathbf{E}\simeq \mathfrak{t}^*\oplus \mathbb{C}
\end{align*}
\end{lemma}
\begin{proof}
	We use the same method as in \cite[section8]{lusztig2}.
	Clearly, $G^+/B$ is a homogeneous variety of $G^+\times \mathbb{C}^*$ for the conjugation action. Then the stabilizer in ${G^+\times \mathbb{C}^*}$ of $B\in G^+/B$ is $B\times \mathbb{C}^*$. Hence $\mathbf{E}=\mathrm{Hom}(B\times \mathbb{C}^*,\mathbb{C}^*)=\mathrm{Hom}(T\times \mathbb{C}^*,\mathbb{C}^*)$. Finally, we get $\mathbb{C}\otimes \mathbf{E}\simeq \mathfrak{t}^*\oplus \mathbb{C}$.
\end{proof}

Let $\dot{\mathcal{L}}^+$ be the $G^+\times \mathbb{C}^*$-equivariant constant sheaf on $\dot{\mathfrak{g}}^+$. Let  $f^+: \dot{\mathfrak{g}}^+\to {\mathfrak{g}}$ be the projection on the first coordinate and define
\begin{equation}
K^+:=(f^+)_!(\dot{\mathcal{L}}^+)
\end{equation}

Let $\ddot{\mathfrak{g}}^+=\dot{\mathfrak{g}}^+\times_{\mathfrak{g}}\dot{\mathfrak{g}}^+$, and let $\ddot{\mathcal{L}}^+$ be the restriction of $\dot{\mathcal{L}}^+\boxtimes (\dot{\mathcal{L}}^+)^*$. From \cite[section2]{Aubert2016GradedHA}, $H_*^{G^+\times\mathbb{C}^*}(\ddot{\mathfrak{g}}^+,\ddot{\mathcal{L}^+})$ is naturally a left $\mathbb{H}(G^+)$-module via $\Delta$; moreover, the map $\mathbb{H}(G^+)\to H_*^{G^+\times\mathbb{C}^*}(\ddot{\mathfrak{g}}^+,\ddot{\mathcal{L}^+})$ given by $\zeta \mapsto \Delta(\zeta)\mathbf{1}$ is an isomorphism of vector spaces.

We define a map 
\begin{equation}\label{EquLine}
	\mathbf{E}\to \mathrm{Hom}_{D_{G^+\times \mathbb{C}^*}(\mathfrak{g})}^2(K^+,K^+)
\end{equation}
as follows. Consider a $G^+\times \mathbb{C}^*$-equivariant line bundle $E\in \mathbf{E}$ on $G^+/B$ and let $\dot{E}$ be the inverse image of $E$ under the $G^+\times \mathbb{C}^*$-morphism $\dot{\mathfrak{g}}^+\to G^+/B$ which is projection on the second coordinate. By the construction in \cite[1.19]{lusztig2}, $\dot{E}$ gives rise to a morphism $\kappa_{\dot{E}}:\dot{\mathcal{L}}^+[-2]\to \dot{\mathcal{L}}^+$ in $D_{G^+\times \mathbb{C}^*}\dot{\mathfrak{g}}^+$. Applying the functor $(f^+)_!:D_{G^+\times \mathbb{C}^*}(\dot{\mathfrak{g}}^+)\to D_{G^+\times \mathbb{C}^*}({\mathfrak{g}})$, we obtain a morphism
$$(f^+)_!\kappa_{\dot{E}}: K^+[-2]\to K^+$$
in $D_{G^+\times \mathbb{C}^*}({\mathfrak{g}})$, which yields an element of $\mathrm{Hom}_{{D_{G^+\times \mathbb{C}^*}(\mathfrak{g})}}^2(K^+,K^+)$. By definition, the map \eqref{EquLine} is given by $E\mapsto (f^+)_!\kappa_{\dot{E}}$.

From \cite{Aubert2016GradedHA}, we have $H^*_{G^+\times \mathbb{C}^*}(\dot{\mathfrak{g}}^+)\simeq S(\mathfrak{t}^*)\otimes \mathbb{C}[\mathfrak{r}]$, which means $H^2_{G^+\times \mathbb{C}^*}(\dot{\mathfrak{g}}^+)\simeq   \mathfrak{t}^*\oplus \mathbb{C}$. We have a commutative diagram :

\begin{equation}\label{EquiLineTr}
	\xymatrix{
		\mathbb{C}\otimes \mathbf{E} \ar[d]_{(a)} \ar[rd]^{(b)}	& \\
		H^2_{G^+\times  \mathbb{C}^*}(\dot{\mathfrak{g}}^+) \ar[r]^\simeq & \mathfrak{t}^*\oplus \mathbb{C}
	}
\end{equation}

where $(b)$ comes from \eqref{HeckeMap} and $(a)$ is defined as follows. For $E\in \mathbf{E}$, we have $\dot{E}$ which is  a $G^+\times \mathbb{C}^*$-equivariant line bundle on  $\dot{\mathfrak{g}}^+$. From \cite[1.19]{lusztig2}, we can define the  equivariant Chern class $c(\dot{E})\in H^2_{G^+\times  \mathbb{C}^*}(\dot{\mathfrak{g}}^+)$. Finally, $(a)$ is given by $1\otimes E\mapsto c(\dot{E})$.

Let $\mathbb{H}'$ be the free  associative $\mathbb{C}$-algebra on the generators $s_\alpha$, $\tau$, $(E)$, where $s_\alpha\in W$ ($\alpha$ is simple root), $\tau\in \mathfrak{R}$, and  $E\in \mathbf{E}$. Then there is a unique surjective homomorphism of algebra with $1$
\begin{equation}\label{Thm3}
	\mathbb{H}'\to \mathbb{H}(G^+)
\end{equation}

which carries $s_\alpha$ to $s_\alpha$, $\tau$ to $\tau$, and $(E)$ to $\xi$, where $\xi \in \mathfrak{t}^*\oplus \mathbb{C}$ corresponds to $1\otimes E \in \mathbb{C}\otimes\mathbf{E}$ in (\ref{EquiLineTr}).

\begin{theorem}\label{IsoHecke}
	There is an algebraic isomorphism
	$\mathbb{H}(G^+) \cong End_{D_{G^+\times \mathbb{C}^*}(\mathfrak{g})}^*K^+$
\end{theorem}

\begin{proof}
	We  use the same method as in \cite[8.11]{lusztig2} adapted to our setting. From \cite[section1]{lusztig1}, we have a commutative diagram
	\begin{equation}\label{ComDiag}
		\scalebox{0.9}{
		\xymatrix{
			\mathrm{Hom}_{D_A(X)}^j(\mathcal{I},\mathcal{I}')
			\ar[r]
			\ar[d] &H^j_A(X,D(D\mathcal{I}'\otimes\mathcal{I})) \ar[d] \ar[r] &\mathrm{Hom}(H_c^{-j}(A\backslash (X\times Z_k),D(\mathcal{I}'_k)\otimes \mathcal{I}_k),\mathbb{C})\ar[d]\\
			\mathrm{Hom}_{D_A(X)}^j(\mathcal{I},\mathcal{I}'')  \ar[r] & H^j_A(X,D(D\mathcal{I}''\otimes\mathcal{I}))\ar[r] & \mathrm{Hom}(H_c^{-j}(A\backslash (X\times Z_k),D(\mathcal{I}''_k)\otimes \mathcal{I}_k),\mathbb{C})
		}}
	\end{equation}
	where $g\in \mathrm{Hom}_{D_A(X)}^j(\mathcal{I},\mathcal{I}')$, $h\in \mathrm{Hom}_{D_A(X)}^j(\mathcal{I}',\mathcal{I}'')$, the left vertical map sends $g$ to $hg$, the middle vertical map is induced by $D(Dh\otimes1):D(D\mathcal{I}'\otimes\mathcal{I})\to D(D\mathcal{I}''\otimes\mathcal{I})$, the right vertical map is induced by $Dh_k\otimes 1: D(\mathcal{I}''_k)\otimes \mathcal{I}_k\to D(\mathcal{I}'_k)\otimes \mathcal{I}_k$.

	Let $\mathcal{I}=\mathcal{I}'=\mathcal{I}''=K^+$ and $A=G^+\times \mathbb{C}^*$, then $\mathrm{Hom}_{D_A(X)}^j(K^+,K^+)\simeq H^j_A(\ddot{\mathfrak{g}}^+,D(DK^+\otimes K^+))\simeq \mathrm{Hom}(H_c^{-j}(A\backslash (X\times Z_k),D(K^+_k)\otimes K^+_k),\mathbb{C})\simeq H_j^{G^+\times \mathbb{C}^*}(\ddot{\mathfrak{g}}^+,\ddot{\mathcal{L}^+})$. Taking the direct sum over $j$, we obtain a canonical identification:
	\begin{equation}\label{Thm1}
		\mathrm{Hom}_{{D_{G^+\times \mathbb{C}^*}(\mathfrak{g})}}^*(K^+,K^+)\simeq H_*^{G^+\times \mathbb{C}^*}(\ddot{\mathfrak{g}}^+,\ddot{\mathcal{L}^+})
	\end{equation}
	as graded vector spaces. We denote by $\mathbf{1}$ the element of $H_0^{G^+\times \mathbb{C}^*}(\ddot{\mathfrak{g}}^+,\ddot{\mathcal{L}^+})$ that corresponds to the identity element in $\mathrm{Hom}_{{D_{G^+\times \mathbb{C}^*}(\mathfrak{g})}}^0(K^+,K^+)$.
	
	From \cite[(7)]{Aubert2016GradedHA}, we have $\mathrm{Hom}_{{D_{G^+\times \mathbb{C}^*}(\mathfrak{g})}}^0(K^+,K^+)\simeq \mathbb{C}[W^+]$.
	
	Taking $(s_\alpha,\tau)\in W^+$ where $\alpha$ is simple root and $\tau \in \mathfrak{R}$, view it as an element in $\mathrm{Hom}_{{D_{G^+\times \mathbb{C}^*}(\mathfrak{g})}}^0(K^+,K^+)$, then we have a commutative diagram:

\begin{equation}\label{ActComDia}
	\xymatrix@C=1pc{
		\mathrm{Hom}_{{D_{G^+\times \mathbb{C}^*}(\mathfrak{g})}}^j(K^+,K^+)
		\ar[r]
		\ar[d]_{(s_\alpha,\tau)}
		&
		\mathrm{Hom}(H_c^{-j}(A\backslash (X\times Z_k),D(K^+_k)\otimes K^+_k),\mathbb{C})
		\ar[d]
		\ar[r]
		&
		H_j^{G^+\times \mathbb{C}^*}(\ddot{\mathfrak{g}}^+,\ddot{\mathcal{L}^+})
		\ar[d]_{\Delta(s_\alpha)\Delta(\tau)}
		\\
		\mathrm{Hom}_{{D_{G^+\times \mathbb{C}^*}(\mathfrak{g})}}^j(K^+,K^+)
		\ar[r]
		&
		\mathrm{Hom}(H_c^{-j}(A\backslash (X\times Z_k),D(K^+_k)\otimes K^+_k),\mathbb{C})
		\ar[r]
		&
		H_j^{G^+\times \mathbb{C}^*}(\ddot{\mathfrak{g}}^+,\ddot{\mathcal{L}^+})
	}
\end{equation}
	
	Taking $E\in \mathbf{E}$, we get a $G^+\times \mathbb{C}^*$-equivariant line bundle $\dot{E}$ over $\dot{\mathfrak{g}}^+$ and $(f^+)_!\kappa_{\dot{E}}\in \mathrm{Hom}_{D_{G^+\times \mathbb{C}^*}(\mathfrak{g})}^0(K^+[-2],K^+)$. 
	
Let $\pi_{12}:\ddot{\mathfrak{g}}^+\to\dot{\mathfrak{g}}^+$ be the projection $(x,gB,g'B)\mapsto(x,gB)$ and  we have a functor $\pi_{12}^*: H^*_{G^+\times \mathbb{C}^*}(\dot{\mathfrak{g}}^+)\to H^*_{G^+\times \mathbb{C}^*}(\ddot{\mathfrak{g}}^+)$.  Pulling back $\dot{E}$ yields  an equivariant line bundle $E'$ on $\ddot{\mathfrak{g}}^+$.
Let $c(\dot{E})\in  H^*_{G^+\times \mathbb{C}^*}(\dot{\mathfrak{g}}^+)\simeq S(\mathfrak{t}^*)\otimes \mathbb{C}[\mathfrak{r}]$ be the equivariant Chern class of equivariant vector bundle $\dot{E}$ on $\dot{\mathfrak{g}}^+$, hence we have $c(E')=\pi_{12}^*(c(\dot{E}))\in H^*_{G^+\times \mathbb{C}^*}(\ddot{\mathfrak{g}}^+)$.  Cup product with  $c(E')=\pi_{12}^*(c(\dot{E}))$ defines the $\Delta(E)$ on $H_*^{G^+\times \mathbb{C}^*}(\ddot{\mathfrak{g}}^+,\ddot{\mathcal{L}^+})$.

	Let $\mathcal{I}=\mathcal{I}''=K^+$, $\mathcal{I}'=K^+[-2]$ and $A=G^+\times \mathbb{C}^*$. We have a commutative diagram:
	\begin{equation*}
			\scalebox{0.9}{
		\xymatrix@C=1pc{
			\mathrm{Hom}_{{D_{G^+\times \mathbb{C}^*}(\mathfrak{g})}}^j(K^+,K^+[-2])
			\ar[r]
			\ar[d]_{(f^+)_!\kappa_{\dot{E}}} 
			&\mathrm{Hom}(H_c^{-j+2}(A\backslash (X\times Z_k),D(K^+_k)\otimes K^+_k),\mathbb{C})\ar[d]\ar[r]
			&H_{j-2}^{G^+\times \mathbb{C}^*}(\ddot{\mathfrak{g}}^+,\ddot{\mathcal{L}^+}) \ar[d]_{\Delta(E)} \\
			\mathrm{Hom}_{{D_{G^+\times \mathbb{C}^*}(\mathfrak{g})}}^j(K^+,K^+)  \ar[r]  & \mathrm{Hom}(H_c^{-j}(A\backslash (X\times Z_k),D(K^+_k)\otimes K^+_k),\mathbb{C})\ar[r]
			& H_j^{G^+\times \mathbb{C}^*}(\ddot{\mathfrak{g}}^+,\ddot{\mathcal{L}^+})
		}}
	\end{equation*}

	From \cite{Aubert2016GradedHA}, $H_*^{G^+\times \mathbb{C}^*}(\ddot{\mathfrak{g}}^+,\ddot{\mathcal{L}^+})$ is a left $\mathbb{H}(G^+)$-module, moreover, the map
	\begin{equation}\label{Thm2}
		\mathbb{H}(G^+)\to H_*^{G^+\times \mathbb{C}^*}(\ddot{\mathfrak{g}}^+,\ddot{\mathcal{L}^+})
	\end{equation}
	
	given by $h\mapsto \Delta(h)\mathbf{1}$ is an isomorphism of vector spaces. Using the commutative diagrams above we see that the following diagram 
	commutes:
	\[
	\begin{tikzcd}
		\mathbb{H}' \ar[r,"(u)"] \ar[d,"{\eqref{Thm3}}"]& \mathrm{Hom}_{{D_{G^+\times \mathbb{C}^*}(\mathfrak{g})}}^*(K^+,K^+)\ar[d,"{\eqref{Thm1}}"]\\
		\mathbb{H}(G^+) \ar[r,"{\eqref{Thm2}}"] & H_*^{G^+\times \mathbb{C}^*}(\ddot{\mathfrak{g}}^+,\ddot{\mathcal{L}^+})
	\end{tikzcd}
	\]
	where the top horizontal map $(u)$ is the homomorphism of $\mathbb{C}$-algebras with unit which sends $(E)$ to $(f^+)_!\kappa_{\dot{E}}$ and $s_\alpha$, $\tau$ to the image of $\mathbb{C}[W^+]\simeq \mathrm{Hom}_{{D_{G^+\times \mathbb{C}^*}(\mathfrak{g})}}^0(K^+,K^+)$.
	
	Since \eqref{Thm1} and \eqref{Thm2} are isomorphisms of vector spaces, we can define an isomorphism of vector spaces $\mathbb{H}(G^+)\to \mathrm{Hom}_{{D_{G^+\times \mathbb{C}^*}(\mathfrak{g})}}^*(K^+,K^+)$ as the composition
	\begin{equation}
		\mathbb{H}(G^+)\to H_*^{G^+\times \mathbb{C}^*}(\ddot{\mathfrak{g}}^+,\ddot{\mathcal{L}^+})\to \mathrm{Hom}_{{D_{G^+\times \mathbb{C}^*}(\mathfrak{g})}}^*(K^+,K^+)
	\end{equation}
	where the first arrow is given by \eqref{Thm2} and the second arrow is the inverse of the one in \eqref{Thm1}. This is automatically an algebra isomorphism since \eqref{Thm3} and $(u)$ are surjective algebra homomorphisms.
	
\end{proof}

\begin{remark}
	In the proof of the theorem, we use the result of \cite[(7)]{Aubert2016GradedHA}: 
	\begin{equation}
		\mathrm{Hom}_{{D_{G^+\times \mathbb{C}^*}(\mathfrak{g})}}^0(K^+,K^+)\simeq \mathbb{C}[W^+]
	\end{equation}
	 Notice that it depends on the choice of isomorphism $\dot{\mathcal{L}}^+\simeq \tau^*\dot{\mathcal{L}}^+$ ($\tau\in \mathfrak{R}$). But in our case $\dot{\mathcal{L}}^+$ is the constant sheaf, so in fact we have a canonical choice, namely identity.
\end{remark}
Combining Lemma \ref{Inv}, we have
\begin{theorem}\label{GraHeG}
	There is an algebraic isomorphism\begin{align*}
	    \mathbb{H}(G^+) \cong \mathrm{Hom}_{D_{G^+\times \mathbb{C}^*}(\mathfrak{g})}^*(K^+,K^+)\cong\mathrm{Hom}_{D_{G\times \mathbb{C}^*}(\mathfrak{g})}^*(K^+,K^+)^\tau
	\end{align*}
\end{theorem}

\subsubsection{Localization}

From \cite[1.20]{lusztig2}, $\mathrm{Hom}_{D_{G^+\times \C^*}(\mathfrak{g})}^*(K^+,K^+)$ is naturally a $H_{G^+\times \C^*}^*$-algebra.  Concretely, there is a natural ring homomorphism $H_{G^+\times \C^*}^*\to \mathrm{Hom}_{D_{G^+\times \C^*}(\mathfrak{g})}^*(K^+,K^+)$ whose image is contained in the center of  $\mathrm{Hom}_{D_{G^+\times \C^*}(\mathfrak{g})}^*(K^+,K^+)$.

Using Theorem \ref{IsoHecke}, we know that $Z(\mathrm{Hom}_{D_{G^+\times \C^*}(\mathfrak{g})}^*(K^+,K^+))= Z(\mathbb{H}(G^+))=S(\mathfrak{t}^*)^{W^+}\otimes\mathbb{C}[\mathfrak{r}]$. The ring homomorphism $H_{G^+\times \C^*}^*\to \mathrm{Hom}_{D_{G^+\times \C^*}(\mathfrak{g})}^*(K^+,K^+)$  is obtained by composing the cup product action $\Delta$ with the pullback map
$H_{G^+\times \C^*}^*\to H_{G^+\times \C^*}^*(\dot{\mathfrak{g}}^+)$ induced by the map $\dot{\mathfrak{g}}^+\to \mathrm{point}$. From \cite[Proposition 3.5]{Aubert2016GradedHA}, since $H_{G^+\times \C^*}^*(\dot{\mathfrak{g}}^+)\simeq  H_{T\times \C^*}^*$, the map $H_{G^+\times \C^*}^*\to H_{G^+\times \C^*}^*(\dot{\mathfrak{g}}^+)$ is equivalent to the map $H_{G^+\times \C^*}^* \to H_{T\times \C^*}^*$ induced by the inclusion $T\times \C^*\to G^+\times \C^*$. 

From \cite[1.11]{lusztig1}, $H_{G^+\times \C^*}^*\simeq S(\mathfrak{g}^*)^{G^+}\otimes\mathbb{C}[\mathfrak{r}]$, $H_{T\times \C^*}^*\simeq S(\mathfrak{t}^*)\otimes\mathbb{C}$ and  the homomorphism $H_{G^+\times \C^*}^* \to H_{T\times \C^*}^*$ is induced by the natural map $\mathfrak{t} \to \mathfrak{g}$. Therefore, we have 
\begin{proposition}[{\cite[Proposition 3.5]{Aubert2016GradedHA}}]\label{CenterHecke}
	The image of the ring homomorphism $H_{G^+\times \C^*}^*\to \mathrm{Hom}_{D_{G^+\times \C^*}(\mathfrak{g})}^*(K^+,K^+)$ is precisely the center of $\mathrm{Hom}_{D_{G^+\times \C^*}(\mathfrak{g})}^*(K^+,K^+)$.
\end{proposition}

Take $\bar{t}\in \mathfrak{t}_{\mathbb{R}}$, $r_0\in \mathbb{C}$, $a=(\bar{t},r_0)$, and let $Z = Z_{G\times \mathbb{C}^*}(a)$ and $Z^+ = Z_{G^+\times \mathbb{C}^*}(a)$. 
 Assume that $Z^+ \neq Z$. Hence we have an exact sequence 
\begin{equation*}
	1\to Z\to Z^+ \to G^+/G \to 1
\end{equation*}
Recall that $G$ is connected and $\mathfrak{R}=G^+/G=\langle1,\tau\rangle$. Let $T_a$ be the smallest torus in $G^+\times \mathbb{C}^*$ whose Lie algebra contains $a$. Then we denote by  $\dot{\mathfrak{g}}^{+,a}$ and $\mathfrak{g}^{a}$ the fixed point set of $T_a$ on $\dot{\mathfrak{g}}^{+}$ and $\mathfrak{g}$, respectively. Thus, $\dot{\mathfrak{g}}^{+,a}=\{(x,gB)\in \dot{\mathfrak{g}}^{+}|[\bar{t},x]=2r_0x\}$, ${\mathfrak{g}}^{a}=\{x\in {\mathfrak{g}}|[\bar{t},x]=2r_0x\}$.

Then we have a $Z^+$-equivariant commutative diagram:
\begin{equation*}
	\xymatrix{
		\dot{\mathfrak{g}}^{+,a}
		\ar[r]^{\dot{j}^+}
		\ar[d]_{\tilde{f}^+} &\dot{\mathfrak{g}}^{+} \ar[d]^{f^+} \\
		\mathfrak{g}^{a}  \ar[r]^j & \mathfrak{g}
	}
\end{equation*}
Define 
\begin{equation}
	\widetilde{K}^+:=\tilde{f}^+_!(\dot{j}^+)^*\dot{\mathcal{L}}^+
\end{equation}
i.e. $\widetilde{K}^+\in D_{Z^+}(\mathfrak{g}^a)$ 

Recall  $\ddot{\mathfrak{g}}^+=\dot{\mathfrak{g}}^+\times_{\mathfrak{g}}\dot{\mathfrak{g}}^+$, and let $\ddot{\mathcal{L}}^+$ be the restriction of $\dot{\mathcal{L}}^+\boxtimes (\dot{\mathcal{L}}^+)^*$ via the embedding $i:\ddot{\mathfrak{g}}^+\hookrightarrow\dot{\mathfrak{g}}^+\times\dot{\mathfrak{g}}^+$. Hence  $\ddot{\mathcal{L}}^+$ is the constant sheaf in our case. 
\begin{lemma}\label{OddV}
	For constant sheaf on $\ddot{\mathfrak{g}}^+$, we have $H_c^{\mathrm{odd}}(\ddot{\mathfrak{g}}^+,\underline{\mathbb{C}})=0$.
\end{lemma}
\begin{proof}
	Denote $\dot{\mathfrak{g}} = \{(x,gB)\in \mathfrak{g} \times G/B :\mathrm{Ad}(g^{-1})x\in \mathfrak{b}\}$ and   $\ddot{\mathfrak{g}}:=\dot{\mathfrak{g}}\times_{\mathfrak{g}}\dot{\mathfrak{g}}$.
	Since $\ddot{\mathfrak{g}}^+$  is a disjoint union of four copies of $\ddot{\mathfrak{g}}$, we have  $H_c^{\mathrm{odd}}(\ddot{\mathfrak{g}}^+,\mathbb{C})=\oplus_{i=1}^4H_c^{\mathrm{odd}}(\ddot{\mathfrak{g}},\mathbb{C})$, so it suffices to show $H_c^{\mathrm{odd}}(\ddot{\mathfrak{g}},\mathbb{C})=0$.
	
	For $w\in W$, set $\ddot{\mathfrak{g}}^{\Omega(w)}:=\{(x,gB,g'B)\in\ddot{\mathfrak{g}}\mid g^{-1}g'\in BwB\}$. Then $\ddot{\mathfrak{g}}=\sqcup_{w\in W}\ddot{\mathfrak{g}}^{\Omega(w)}$. Consider the projection
    \begin{equation*}
\ddot{\mathfrak{g}}\stackrel{\mathrm{pr}_{23}}{\longrightarrow} G/B\times G/B
	\end{equation*}
	and denote $\mathcal{O}(w)=\mathrm{pr}_{23}\ddot{\mathfrak{g}}^{\Omega(w)}$. Hence $\mathcal{O}(w)$ is a $G$-orbit under the diagonal action and we have a poset of $G$-orbit closures in $G/B\times G/B$. Enumerate $\mathcal{O}(w)$ in such an order: $\mathcal{O}(1)$, $\mathcal{O}(2)$, $\cdots$, $\mathcal{O}(m)$ that  $\mathrm{dim}\mathcal{O}(1)\geq \mathrm{dim}\mathcal{O}(2)\geq \cdots \mathrm{dim}\mathcal{O}(m)$. Let $\ddot{\mathfrak{g}}_j=\sqcup_{i\geq j}\ddot{\mathfrak{g}}^{\Omega(i)}$. Then $\ddot{\mathfrak{g}}_1\supset\ddot{\mathfrak{g}}_2\supset\cdots\supset\ddot{\mathfrak{g}}_m$ is a filtration over \begin{equation*}
		\ddot{\mathfrak{g}} \stackrel{\mathrm{pr}_{3}}{\longrightarrow} G/B
	\end{equation*}
	Since each $\ddot{\mathfrak{g}}^{\Omega(w)}$ is a vector bundle over $G/B$ with restriction of $\mathrm{pr}_3$, we have \begin{equation*}
		H_c^{\mathrm{odd}}(\ddot{\mathfrak{g}}^{\Omega(w)})=0
	\end{equation*}
	Using the open-closed long exact sequence for each $j$:\begin{align*}
		\cdots\to H_c^i(\ddot{\mathfrak{g}}_j-\ddot{\mathfrak{g}}_{j+1})\to H_c^i(\ddot{\mathfrak{g}}_j)\to H_c^i(\ddot{\mathfrak{g}}_{j+1})\to H_c^{i+1}(\ddot{\mathfrak{g}}_j-\ddot{\mathfrak{g}}_{j+1})\to \cdots
	\end{align*}
	We have $H_c^{\mathrm{odd}}(\ddot{\mathfrak{g}}^+,\underline{\mathbb{C}})=0$ by induction.
\end{proof}

 Lemma \ref{OddV} implies that the condition in  \cite[4.6(f)]{lusztig2} is satisfied.  Using \cite{lusztig2} section 4-5, we have  $\mathbb{C}$-algebra isomorphism (sometimes we omit $\mathrm{Res}$ for convenience):
\begin{align}
	\mathbb{C}_{a}\otimes_{H^*_{G\times \mathbb{C}^*}}\mathrm{Hom}^*_{D_{G\times \mathbb{C}^*}(\mathfrak{g})}(K^+,K^+)&\simeq
	\mathbb{C}_{a}\otimes_{H^*_{Z}}\mathrm{Hom}^*_{D_{Z}(\mathfrak{g})}(K^+,K^+)\label{local1}\\&\simeq
	\mathbb{C}_{a}\otimes_{H^*_Z}\mathrm{Hom}^*_{{D_{Z}(\mathfrak{g}^a)}}(\widetilde{K}^+,\widetilde{K}^+)\label{local2}\\&\simeq \mathrm{Hom}^*(\widetilde{K}^+,\widetilde{K}^+)\label{local3}
\end{align}

We have 
\begin{align}
	\mathbb{C}_{a}\otimes_{H^*_{G^+\times \mathbb{C}^*}}\mathrm{Hom}^*_{{D_{G^+\times \mathbb{C}^*}(\mathfrak{g})}}(K^+,K^+) 
	&\simeq \mathbb{C}_{a}\otimes_{H^*_{G\times \mathbb{C}^*}}H^*_{G\times \mathbb{C}^*}\otimes_{H^*_{G^+\times \mathbb{C}^*}}\mathrm{Hom}^*_{{D_{G^+\times \mathbb{C}^*}(\mathfrak{g})}}(K^+,K^+) \nonumber\\
	&\simeq \mathbb{C}_{a}\otimes_{H^*_{G\times \mathbb{C}^*}}\mathrm{Hom}^*_{{D_{G\times \mathbb{C}^*}(\mathfrak{g})}}(\mathrm{Res}^{G^+\times \mathbb{C}^*}_{G\times \mathbb{C}^*} K^+, \mathrm{Res}^{G^+\times \mathbb{C}^*}_{G\times \mathbb{C}^*} K^+)^\tau \label{Inv1}
\end{align}
where $H^*_{G\times \mathbb{C}^*}$ is regarded as an $H^*_{G^+\times \mathbb{C}^*}$-algebra via the map $H^*_{G^+\times \mathbb{C}^*}\to H^*_{G\times \mathbb{C}^*}$ and \eqref{Inv1} follows from Lemma \ref{Inv}.

Recall that for a connected reductive group $A$, we have $H_A^*= S(\mathfrak{a}^*)^A$. We regard $\iota_a:H_A^*\to \mathbb{C}$ as the algebra homomorphism given by evaluating a polynomial function on  $\mathfrak{a}$ at the semisimple element $a$. We denote $\mathbb{C}_a$ as a $H_A^*$-algebra via this algebra homomorphism $\iota_a$.

Let $A$ be a connected normal subgroup of $A^+$ with $\mathfrak{R} :=A^+/A= \langle1,\tau\rangle$. There is a $\tau$-action on  $H_A^*$. Taking a semisimple element $a\in \mathfrak{a}$ satisfying $\mathfrak{R}\cdot a \subset A\cdot a$, we define  a $\tau$-action on $\mathbb{C}_a$ to be the trivial action. Since $\iota_a(\tau\cdot h)=h(\tau(a))=h(a)$ for $h\in H_A^*= S(\mathfrak{a}^*)^A$, the algebra homomorphism $\iota_a:H_A^*\to \mathbb{C}_a$ is compatible with the $\tau$-action.

Therefore we can define a $\tau$-action on
\begin{equation*}
	\mathbb{C}_{a}\otimes_{H_A^*}\mathrm{Hom}_{D_{A}(X)}^*(\mathrm{Res}^{A^+}_A(\mathcal{F}),\mathrm{Res}^{A^+}_A(\mathcal{F}'))
\end{equation*} 
by $\tau(c\otimes f)=c\otimes\tau(f)$. Combining Proposition \ref{ModAct}, it is well defined.

We will check that the isomorphism \eqref{local1}, \eqref{local2}, \eqref{local3}  commute  with respect to the corresponding $\tau$-actions.  In categorical language, we will verify that the functors appearing in \eqref{local1}, \eqref{local2}, \eqref{local3} are $\mathfrak{R}$-equivariant.

Firstly, we consider \eqref{local1} as follows.

Since we have the short exact sequence 	
\begin{equation*}
	1\to Z\to Z^+ \to G^+/G \to 1
\end{equation*}
which means that $Z^+/Z\simeq G^+/G\simeq \mathfrak{R}\simeq \langle1,\tau\rangle$, we can take a representative $\bar{\tau}$ to get a commutative diagram of group homomorphisms

\[
\begin{tikzcd}
	Z\arrow[r] \arrow[d,"{\mathrm{Ad}_{\bar{\tau}}}"]&G\times\mathbb{C}^*\arrow[d,"{\mathrm{Ad}_{\bar{\tau}}}"]\\
	Z\arrow[r]&G\times\mathbb{C}^*
\end{tikzcd}
\]

From \cite[6.6(2)]{EquBeL}, we have $\mathrm{Res}^{G\times \mathbb{C}^*}_Z\circ \tau^*= \tau^*\circ \mathrm{Res}^{G\times \mathbb{C}^*}_Z$.

Therefore, the isomorphism \eqref{local1} commutes with the $\tau$-action and we have 

\begin{equation}\label{prooflocal1}
	\mathbb{C}_{a}\otimes_{H^*_{G\times \mathbb{C}^*}}\mathrm{Hom}^*_{D_{G\times \mathbb{C}^*}(\mathfrak{g})}(K^+,K^+)^\tau\simeq
	\mathbb{C}_{a}\otimes_{H^*_{Z}}\mathrm{Hom}^*_{D_{Z}(\mathfrak{g})}(K^+,K^+)^\tau
\end{equation}

Secondly, we analyze \eqref{local2} using the interpretation in \cite[4.9]{fourier}. 

We denote by
\begin{equation*}
	(\dot{j}^+)_!(\dot{j}^+)^!\dot{\mathcal{L}}^+\stackrel{\mu}{\longrightarrow}\dot{\mathcal{L}}^+\stackrel{\nu}{\longrightarrow}(\dot{j}^+)_*(\dot{j}^+)^*\dot{\mathcal{L}}^+
\end{equation*}

There are canonical morphisms of $H_Z^*$-algebras
\begin{equation}\label{ProofLocal2}
	\begin{tikzcd}
		\mathrm{Hom}^*_{D_{Z}(\mathfrak{g})}\Bigl((f^+)_*\dot{\mathcal{L}}^+,(f^+)_*\dot{\mathcal{L}}^+\Bigr)\\
		\mathrm{Hom}^*_{D_{Z}(\mathfrak{g})}\Bigl((f^+)_*(\dot{j}^+)_*(\dot{j}^+)^*\dot{\mathcal{L}}^+,(f^+)_*(\dot{j}^+)_!(\dot{j}^+)^!\dot{\mathcal{L}}^+\Bigr)
		\arrow[d, "\Psi_2"']\arrow[u, "\Psi_1"] \\
		\mathrm{Hom}^*_{D_{Z}(\mathfrak{g})}((f^+)_*(\dot{j}^+)_*(\dot{j}^+)^*\dot{\mathcal{L}}^+,(f^+)_*(\dot{j}^+)_*(\dot{j}^+)^*\dot{\mathcal{L}}^+)
	\end{tikzcd}	
\end{equation}

defined by $\Psi_1(u)=(f^+)_*(\mu)\circ u \circ (f^+)_*(\nu)$, $\Psi_2(u)=(f^+)_*(\nu\mu)\circ u$.

We have $\tau^*\Psi_1(u)=\tau^*(f^+)_*(\mu)\circ \tau^*(u) \circ \tau^*(f^+)_*(\nu)$, $\tau^*\Psi_2(u)=\tau^*(f^+)_*(\nu\mu)\circ \tau^*(u)$.

In fact, $(f^+)_*(\mu)=(f^+)_*(\mathrm{Res}^{Z^+}_Z\mu)=\mathrm{Res}^{Z^+}_Z(f^+)_*(\mu)$, $(f^+)_*(\nu)=(f^+)_*(\mathrm{Res}^{Z^+}_Z\nu)=\mathrm{Res}^{Z^+}_Z(f^+)_*(\nu)$, $(f^+)_*(\nu\mu)=(f^+)_*(\mathrm{Res}^{Z^+}_Z\nu\mu)=\mathrm{Res}^{Z^+}_Z(f^+)_*(\nu\mu)$, hence from Lemma \ref{Inv}, we have $\tau^*(f^+)_*(\mu)=(f^+)_*(\mu)$, $\tau^*(f^+)_*(\nu)=(f^+)_*(\nu)$, $\tau^*(f^+)_*(\nu\mu)=(f^+)_*(\nu\mu)$, then $\tau^*\Psi_1(u)=\Psi_1(\tau^*u)$, $\tau^*\Psi_2(u)=\Psi_2(\tau^*u)$.

Note that $f^+$, $\tilde{f}^+=f^+\circ \dot{j}^+$  are proper, so $(f^+)_!=(f^+)_*$ and $(\tilde{f}^+)_!=(\tilde{f}^+)_*=(f^+)_*(\dot{j}^+)_*$. Applying $\mathbb{C}_{a}\otimes_{H^*_{Z}}$ to \eqref{ProofLocal2}, we get a $\mathbb{C}$-algebra isomorphism \eqref{local2} that commutes with the $\tau$-action. Therefore we have

\begin{equation}\label{prooflocal2}
	\mathbb{C}_{a}\otimes_{H^*_{Z}}\mathrm{Hom}^*_{D_{Z}(\mathfrak{g})}(K^+,K^+)^\tau\simeq
	\mathbb{C}_{a}\otimes_{H^*_Z}\mathrm{Hom}^*_{{D_{Z}(\mathfrak{g}^a)}}(\widetilde{K}^+,\widetilde{K}^+)^\tau
\end{equation}

\begin{remark}
	Alternatively, we can view  $\mathrm{Hom}^*_{D_{G\times \mathbb{C}^*}(\mathfrak{g})}(K^+,K^+)\simeq H^*_{G\times \mathbb{C}^*}(\ddot{\mathfrak{g}}^+,i^!(\dot{\mathcal{L}}^+\boxtimes D\dot{\mathcal{L}}^+))=H^*_{G\times \mathbb{C}^*}(\ddot{\mathfrak{g}}^+,{\mathbb{C}})$, $\mathrm{Hom}^*_{D_{G^+\times \mathbb{C}^*}(\mathfrak{g})}(K^+,K^+)\simeq H^*_{G^+\times \mathbb{C}^*}(\ddot{\mathfrak{g}}^+,i^!(\dot{\mathcal{L}}^+\boxtimes D\dot{\mathcal{L}}^+))=H^*_{G^+\times \mathbb{C}^*}(\ddot{\mathfrak{g}}^+,{\mathbb{C}})\simeq H^*_{G\times \mathbb{C}^*}(\ddot{\mathfrak{g}}^+,{\mathbb{C}})^\tau$ (\cite[1.9]{lusztig1}). Transferring the above equivariant $\operatorname{Hom}$ to equivariant cohomology and using the commutativity of $\operatorname{Res}$, we can also obtain the isomorphisms \eqref{prooflocal1} and \eqref{prooflocal2}.
\end{remark}

Thirdly, we analyze \eqref{local3}. In fact, we have an isomorphism of $H_{T_a}^*$-algebra \cite[5.1]{lusztig2}:
\begin{equation}\label{local31}
	H_{T_a}^*\otimes_{H^*_Z} \mathrm{Hom}^*_{D_{Z}(\mathfrak{g}^a)}(\widetilde{K}^+,\widetilde{K}^+)\simeq 	H_{T_a}^*\otimes_{\mathbb{C}} \mathrm{Hom}^*(\widetilde{K}^+,\widetilde{K}^+)
\end{equation}

and an isomorphism of $\mathbb{C}$-algebras by applying $\mathbb{C}_{a}\otimes_{H^*_{T_a}}$:
\begin{equation*}
	\mathbb{C}_a\otimes_{H^*_Z} \mathrm{Hom}^*_{D_{Z}(\mathfrak{g}^a)}(\widetilde{K}^+,\widetilde{K}^+)\simeq  \mathrm{Hom}^*(\widetilde{K}^+,\widetilde{K}^+)
\end{equation*}

Note that \eqref{local31} is given by  $H_{T_a}^*$-algebra isomorphisms:
\begin{equation*}
	H_{T_a}^*\otimes_{\mathbb{C}} \mathrm{Hom}^*(\widetilde{K}^+,\widetilde{K}^+)\stackrel{\Upsilon_1}{\longrightarrow}
	\mathrm{Hom}^*_{D_{T_a}(\mathfrak{g}^a)}(\widetilde{K}^+,\widetilde{K}^+)
	\stackrel{\Upsilon_2}{\longleftarrow}	H_{T_a}^*\otimes_{H^*_Z} \mathrm{Hom}^*_{D_{Z}(\mathfrak{g}^a)}(\widetilde{K}^+,\widetilde{K}^+)
\end{equation*}

The group homomorphism $T_a\to \{1\}$ gives rise to a functor $D(\mathfrak{g}^a)\to D_{T_a}(\mathfrak{g}^a)$ since $T_a$ acts trivially on $\mathfrak{g}^a$. This gives rise to a $\mathbb{C}$-algebra homomorphism
\begin{equation}\label{local32}
	\mathrm{Hom}^*(\widetilde{K}^+,\widetilde{K}^+)\stackrel{\Upsilon'_1}{\longrightarrow} \mathrm{Hom}^*_{D_{T_a}(\mathfrak{g})}(\widetilde{K}^+,\widetilde{K}^+)	 
\end{equation}
which extends uniquely to $\Upsilon_1$ by cup product $\Upsilon_1((h,f))=h\cdot \Upsilon'_1(f)$.

Since $Z^+/Z\simeq\mathfrak{R}\simeq \langle1,\tau\rangle$, we can take a representative ${\bar{\tau}}$ to get a commutative diagram of group homomorphisms  that is compatible with their group actions on $\mathfrak{g}^a$:

\begin{equation}\label{TAct}
	\begin{tikzcd}
		1\arrow[d,"{\mathrm{Ad}_{\bar{\tau}}}"]&T_a\arrow[r]\arrow[l] \arrow[d,"{\mathrm{Ad}_{\bar{\tau}}}"]&Z\arrow[d,"{\mathrm{Ad}_{\bar{\tau}}}"]\\
		1&T_a\arrow[l]\arrow[r]&Z
	\end{tikzcd}
\end{equation}

Using the left square and \cite{EquBeL}, we have $\Upsilon'_1\circ \tau^*=\tau^*\circ \Upsilon'_1$, hence  $\Upsilon_1(\tau\cdot(h,f))=\Upsilon_1(\tau(h),\tau(f))=\tau(h)\cdot\Upsilon'_1(\tau(f))=\tau(h)\cdot\tau(\Upsilon'_1(f))$,  which means that $\Upsilon_1$  commutes with $\tau$. Notice that we choose $\bar{\tau}$ such that $\bar{\tau}$ stabilizes $T_a$ and $\mathrm{Ad}_{\bar{\tau}}( a)=a$.   Therefore we have 
\begin{align}
	\mathrm{Hom}^*(\widetilde{K}^+,\widetilde{K}^+)^\tau &\simeq \mathbb{C}_a\otimes_{H^*_{T_a}}	H_{T_a}^*\otimes_{\mathbb{C}} \mathrm{Hom}^*(\widetilde{K}^+,\widetilde{K}^+)^\tau \nonumber\\
	&\simeq\mathbb{C}_a\otimes_{H^*_{T_a}} 	\mathrm{Hom}^*_{D_{T_a}(\mathfrak{g}^a)}(\widetilde{K}^+,\widetilde{K}^+)^\tau \label{prooflocal31}
\end{align}

Using the right square we get $\mathrm{Res}^{Z}_{T_a}\circ \tau^*= \tau^*\circ \mathrm{Res}^{Z}_{T_a}$. Since $\Upsilon_2$ is induced by $\mathrm{Res}^{Z}_{T_a}$, we have
\begin{equation}\label{prooflocal32}
	\mathbb{C}_a\otimes_{H^*_{T_a}}\mathrm{Hom}^*_{D_{T_a}(\mathfrak{g}^a)}(\widetilde{K}^+,\widetilde{K}^+)^\tau\simeq	\mathbb{C}_a\otimes_{H^*_{Z}}\mathrm{Hom}^*_{D_{Z}(\mathfrak{g}^a)}(\widetilde{K}^+,\widetilde{K}^+)^\tau
\end{equation}

Combining \eqref{prooflocal31} and \eqref{prooflocal32}, we obtain an isomorphism of $\mathbb{C}$-algebras
\begin{equation}
	\mathrm{Hom}^*(\widetilde{K}^+,\widetilde{K}^+)^\tau\simeq \mathbb{C}_a\otimes_{H^*_Z} \mathrm{Hom}^*_{D_{Z}(\mathfrak{g}^a)}(\widetilde{K}^+,\widetilde{K}^+)^\tau
\end{equation}

Finally, combining the three steps above, we obtain the isomorphism:

\begin{align}
	\mathbb{C}_{a}\otimes_{H^*_{G^+\times \mathbb{C}^*}}\mathrm{Hom}^*_{{D_{G^+\times \mathbb{C}^*}(\mathfrak{g})}}(K^+,K^+) &\simeq
	\mathbb{C}_{a}\otimes_{H^*_{G\times \mathbb{C}^*}}\mathrm{Hom}^*_{D_{G\times \mathbb{C}^*}(\mathfrak{g})}(K^+,K^+)^\tau\nonumber\\
	&\simeq
	\mathbb{C}_{a}\otimes_{H^*_{Z}}\mathrm{Hom}^*_{D_{Z}(\mathfrak{g})}(K^+,K^+)^\tau \nonumber\\&\simeq
	\mathbb{C}_{a}\otimes_{H^*_Z}\mathrm{Hom}^*_{{D_{Z}(\mathfrak{g}^a)}}(\widetilde{K}^+,\widetilde{K}^+)^\tau\nonumber\\&\simeq \mathrm{Hom}^*(\widetilde{K}^+,\widetilde{K}^+)^\tau\label{HomRed1}
\end{align}

Combining Proposition \ref{CenterHecke}, we have a $\mathbb{C}$-algebra isomorphism:
\begin{equation}\label{HomRedF}
	\mathbb{C}_{a}\otimes_{Z(\mathbb{H}(G^+))}\mathbb{H}(G^+)\simeq \mathrm{Hom}^*(\widetilde{K}^+,\widetilde{K}^+)^\tau 	
\end{equation}

Applying the equivariant decomposition theorem  to the  $Z^+$-equivariant  proper map  $\tilde{f}^+: \dot{\mathfrak{g}}^{+,a}\to \mathfrak{g}^a$, we get 
\begin{equation*}
	\widetilde{K}^+=\bigoplus_{i\in \mathbb{Z},\phi^+=(\mathbb{O},\chi)} L_{\phi^+}(i)\otimes \IC_{\phi^+}[i]
\end{equation*}
where $\phi^+=(\mathbb{O},\chi)$ runs over pairs consisting of a $Z^+$-orbit $\mathbb{O}$ and an irreducible $Z^+$-equivariant local system $\chi$ on $\mathbb{O}$.

Denote by $\phi^+$ and $\psi^+$ the $Z^+$-equivariant local systems on an unspecified $Z^+$-orbit of $\mathfrak{g}^a$.
Recall that an object $\mathcal{F}\in \mathrm{Perv}_{Z^+}(\mathfrak{g}^a)$ if $\mathrm{For}(\mathcal{F})\in \mathrm{Perv}(\mathfrak{g}^a)$, where $\mathrm{For}$ is the forgetful functor $\mathrm{For}=\mathrm{Res}^{Z^+}_{\{\mathbf{1}\}}$. Hence
\begin{equation*}
	\mathrm{Hom}^n(\IC_{\phi^+},\IC_{\psi^+})=0 \quad \mathrm{if} \quad n< 0
\end{equation*}

Let $L_{\phi^+}=\oplus_iL_{\phi^+}(i)$ and 
\begin{equation*}
	\widetilde{K}^{+'}=\bigoplus_{\phi^+=(\mathbb{O},\chi)} L_{\phi^+}\otimes \IC_{\phi^+}
\end{equation*}
Then
\begin{equation}
\begin{aligned}
	\mathrm{Hom}^*(\widetilde{K}^+,\widetilde{K}^+)&\simeq \mathrm{Hom}^*(\widetilde{K}^{+'},\widetilde{K}^{+'})\\
	&=\bigoplus_{k\geq0,\phi^+,\psi^+}\mathrm{Hom}_{\mathbb{C}}(L_{\phi^+},L_{\psi^+})\otimes\mathrm{Hom}^k(\IC_{\phi^+},\IC_{\psi^+})
\end{aligned} 
\end{equation}

Since $Z$ is connected, we have 
\begin{equation}\label{ConIso}
	\mathrm{Hom}^0_{D_{Z}(\mathfrak{g}^a)}(\IC_{\phi^+},\IC_{\psi^+})=\mathrm{Hom}^0(\IC_{\phi^+},\IC_{\psi^+})
\end{equation}
which is given by $\mathrm{For}$. Taking a representative $\bar{\tau}$ such that the diagram \ref{TAct} is commutative and using $\mathrm{For}\circ\tau^*=\tau^*\circ\mathrm{For}$, we have

\begin{equation}\label{ProofSim1}
	\mathrm{Hom}^0(\IC_{\phi^+},\IC_{\psi^+})^\tau=\mathrm{Hom}^0_{D_{Z}(\mathfrak{g}^a)}(\IC_{\phi^+},\IC_{\psi^+})^\tau=\mathrm{Hom}^0_{D_{Z^+}(\mathfrak{g}^a)}(\IC_{\phi^+},\IC_{\psi^+})
\end{equation}

Since $\IC_{\phi^+}$, $\IC_{\psi^+}$ are simple objects in $\mathrm{Perv}_{Z^+}(\mathfrak{g}^a)$, 
\begin{equation}\label{ProofSim2}
	\mathrm{Hom}^0_{D_{Z^+}(\mathfrak{g}^a)}(\IC_{\phi^+},\IC_{\psi^+})=\mathbb{C}\cdot \delta_{\phi^+,\psi^+}
\end{equation}

Combining \eqref{ProofSim1} and \eqref{ProofSim2}, we have
\begin{align}
	\mathrm{Hom}^0(\widetilde{K}^{+'},\widetilde{K}^{+'})^\tau
	&=\bigoplus_{\phi^+,\psi^+}\mathrm{Hom}_{\mathbb{C}}(L_{\phi^+},L_{\psi^+})\otimes\mathrm{Hom}^0(\IC_{\phi^+},\IC_{\psi^+})^\tau\nonumber\\
	&=\bigoplus_{\phi^+,\psi^+}\mathrm{Hom}_{\mathbb{C}}(L_{\phi^+},L_{\psi^+})\otimes\mathrm{Hom}^0_{D_{Z^+}(\mathfrak{g}^a)}(\IC_{\phi^+},\IC_{\psi^+})\nonumber\\
	&=\bigoplus_{\phi^+}\mathrm{End}_{\mathbb{C}}L_{\phi^+}\label{SimMod1}
\end{align}

Therefore, we have 
\begin{align}
	\mathrm{Hom}^*(\widetilde{K}^+,\widetilde{K}^+)^\tau&\simeq \mathrm{Hom}^*(\widetilde{K}^{+'},\widetilde{K}^{+'})^\tau\nonumber\\
	&=\bigoplus_{k\geq0,\phi^+,\psi^+}\mathrm{Hom}_{\mathbb{C}}(L_{\phi^+},L_{\psi^+})\otimes\mathrm{Hom}^k(\IC_{\phi^+},\IC_{\psi^+})^\tau\nonumber\\
	&\to 	\mathrm{Hom}^0(\widetilde{K}^{+'},\widetilde{K}^{+'})^\tau=\bigoplus_{\phi^+}\mathrm{End}_{\mathbb{C}}L_{\phi^+}\label{SimMod2}
\end{align}
where the kernel of this homomorphism is a two-sided nilpotent ideal. This nilpotent ideal is the radical of the finite dimensional algebra $\mathrm{Hom}^*(\widetilde{K}^+,\widetilde{K}^+)^\tau$, since $\mathrm{Hom}^0(\widetilde{K}^{+'},\widetilde{K}^{+'})^\tau$ is a semisimple algebra.

Denote
\begin{equation}
	\mathbf{P}_{\phi^+}:=\mathrm{Hom}^0(\IC_{\phi^+},\widetilde{K}^{+'})^\tau
\end{equation}  Then 
\begin{align}
	\mathbf{P}_{\phi^+}&=\mathrm{Hom}^0(\IC_{\phi^+},\widetilde{K}^{+'})^\tau\nonumber\\
	&=\bigoplus_{\psi^+}L_{\phi^+}\otimes\mathrm{Hom}^0(\IC_{\phi^+},\IC_{\psi^+})^\tau\nonumber\\
	&=\bigoplus_{\psi^+}L_{\phi^+}\otimes\mathrm{Hom}^0_{D_{Z^+}(\mathfrak{g}^a)}(\IC_{\phi^+},\IC_{\psi^+})\nonumber\\
	&=L_{\phi^+}
\end{align}

Using \eqref{SimMod1}, $\mathbf{P}_{\phi^+}=L_{\phi^+}$ is a complete collection of mutually non-isomorphic simple $\mathrm{Hom}^0(\widetilde{K}^{+'},\widetilde{K}^{+'})^\tau$-modules. These modules can be regarded as $\mathrm{Hom}(\widetilde{K}^+,\widetilde{K}^+)^\tau$-module via the algebra homomorphism \eqref{SimMod2} above, and they provide a complete collection of mutually non-isomorphic simple $\mathrm{Hom}(\widetilde{K}^+,\widetilde{K}^+)^\tau$-modules.

Therefore, we need to determine what kind of $\IC$-sheaves do occur in decomposition of $\widetilde{K}^+$.

We have a $Z^+$-equivariant isomorphism $Z^+\times_Z\dot{\mathfrak{g}}^a\stackrel{\simeq}{\longrightarrow} \dot{\mathfrak{g}}^{+,a}$ given by the action map: $\dot{m}:(z,x,(g,1)B)\mapsto (\Ad(z)x,z\cdot(g,1)B)$. Let $m:Z^+\times_Z\mathfrak{g}^a \to \mathfrak{g}^a$ be the action map: $(z,x)\mapsto \Ad(z)x$. Hence we have the $Z^+$-equivariant commutative diagram: 
\begin{equation*}
	\begin{tikzcd}
		Z^+\times_Z\dot{\mathfrak{g}}^a\arrow[r,"\dot{m}"] \arrow[d,"p_1"]&\dot{\mathfrak{g}}^{+,a}\arrow[d,"\tilde{f}^+"]\\
		Z^+\times_Z\mathfrak{g}^a\arrow[r,"m"]&\mathfrak{g}^a
	\end{tikzcd}
\end{equation*}
where $p_1$ is the projection.
Then we have
\begin{align}
	\widetilde{K}^+&= \tilde{f}^+_!\underline{\mathbb{C}}
	=\tilde{f}^+_!\dot{m}_!\underline{\mathbb{C}}\nonumber\\
	&=(\tilde{f}^+\circ\dot{m})_!\underline{\mathbb{C}}
	=(m\circ p_1)_!\underline{\mathbb{C}}\nonumber\\
	&=m_!{p_1}_!\underline{\mathbb{C}}\label{Ind1}
\end{align}
From the natural inclusion $i:\mathfrak{g}^a\to 	Z^+\times_Z\mathfrak{g}^a$ and $\dot{i}:\dot{\mathfrak{g}}^a\to Z^+\times_Z\dot{\mathfrak{g}}^a$, we have the induction equivalences:
\begin{equation*}
	i^*:D_{Z^+}(Z^+\times_Z\mathfrak{g}^a)\stackrel{\simeq}{\longrightarrow} D_Z(\mathfrak{g}^a)
\end{equation*}
\begin{equation*}
	\dot{i}^*:D_{Z^+}(Z^+\times_Z\dot{\mathfrak{g}}^a )\stackrel{\simeq}{\longrightarrow} D_Z(\dot{\mathfrak{g}}^a)
\end{equation*}
Using the commutative diagram
\begin{equation}
	\xymatrix{
		Z^+\times_Z\mathfrak{g}^a
		& Z^+\times_Z\dot{\mathfrak{g}}^a   	\ar[l]_{p_1}\\
		\mathfrak{g}^a	\ar[u]_{i} & \dot{\mathfrak{g}}^a \ar[l]_{\tilde{f}}\ar[u]^{\dot{i}}
	}
\end{equation}
and the fact that the induction equivalence commutes with $\tilde{f}_!$, we have 
\begin{equation*}
	i^*{p_1}_!\underline{\mathbb{C}}=\tilde{f}_!\dot{i}^*\underline{\mathbb{C}}=\tilde{f}_!\underline{\mathbb{C}}
\end{equation*}
which means 
\begin{equation}\label{Ind2}
	{p_1}_!\underline{\mathbb{C}}=	(i^*)^{-1}\tilde{f}_!\underline{\mathbb{C}}
\end{equation}
Denote 
\begin{equation*}
	\widetilde{K}:=\tilde{f}_!\underline{\mathbb{C}}
\end{equation*}
Combining \eqref{Ind1} and \eqref{Ind2}, we have 
\begin{align}\label{GeoInd}
	\widetilde{K}^+&= m_!{p_1}_!\underline{\mathbb{C}}
	=m_!(i^*)^{-1}\tilde{f}_!\underline{\mathbb{C}}\\
	&=m_!(i^*)^{-1}\widetilde{K} = \mathrm{Ind}_Z^{Z^+}\widetilde{K}
\end{align}
where $\mathrm{Ind}_Z^{Z^+}=\mathrm{Ind}_{Z*}^{Z^+}=\mathrm{Ind}_{Z!}^{Z^+}$ because of $m_!=m_*$ and $m^*=m^!$.

Denote the stabilizer of $y\in \mathfrak{g}^a$ in $Z=Z_{G\times \mathbb{C}^*}(a)$  by $Z_y=\{(g,r)\in G\times \mathbb{C}^*|\Ad(g)\bar{t}=\bar{t}, \Ad(g)y=r^2y\}$, where $a=(\bar{t},r_0)$. Now we assume that $Z_y$ is connected for all $y$.

Applying the equivariant decomposition theorem  to the $Z$-equivariant  proper map  $\tilde{f}: \dot{\mathfrak{g}}^{a}\to \mathfrak{g}^a$, we get 
\begin{equation*}
	\widetilde{K}=\bigoplus_{i\in \mathbb{Z},\phi=(\mathbb{O},1)} L_{\phi}(i)\otimes \IC_{\phi}[i]
\end{equation*}
where $\phi=(\mathbb{O},1)$ runs over  $Z$-orbits $\mathbb{O}$.

Using the non-vanishing result of $\widetilde{K}$, we know that all $Z$-orbit $\IC_{\phi}$ does occur in the decomposition of $\widetilde{K}$. Hence we have 
\begin{align}
	\mathrm{Ind}_Z^{Z^+}\widetilde{K} &=\bigoplus_{i\in \mathbb{Z},\phi=(\mathbb{O},1)} L_{\phi}(i)\otimes \mathrm{Ind}_Z^{Z^+}\IC_{\phi}[i]
\end{align}

\begin{lemma}
	Take a $Z$-orbit $\mathbb{O}$ and $y\in \mathbb{O}$. \\If $Z^+\cdot \mathbb{O}=\mathbb{O}$, we have $Z^+_y/Z_y\simeq \mathfrak{R}$ and $\pi_0(Z^+_y)\simeq\mathfrak{R}$, hence $\mathrm{Ind}_Z^{Z^+}\IC_{(\mathbb{O},1)}=\IC_{(\mathbb{O},1)^+}\oplus \IC_{(\mathbb{O},-1)^+}$, where $\IC_{(\mathbb{O},1)^+}$ and $\IC_{(\mathbb{O},-1)^+}$ denote the equivariant intersection cohomology sheaf on the $Z^+$-orbit $\mathbb{O}$ corresponding to the trivial and sign representations of $\pi_0(Z^+_y)\simeq\mathfrak{R}$, respectively.\\
	If $Z^+\cdot \mathbb{O}=\mathbb{O}\sqcup\mathbb{O}'$, we have $Z^+_y=Z_y$ and $\pi_0(Z^+_y)=\pi_0(Z_y)=1$, hence $\mathrm{Ind}_Z^{Z^+}\IC_{(\mathbb{O},1)}=\IC_{(\mathbb{O}\sqcup\mathbb{O}',1)^+}$, where $\IC_{(\mathbb{O}\sqcup\mathbb{O}',1)^+}$ denotes the equivariant intersection cohomology sheaf on the $Z^+$-orbit $\mathbb{O}\sqcup\mathbb{O}'$ corresponding to the trivial representation of $\pi_0(Z^+_y)$.
\end{lemma}

\begin{proof}
	
	The induction equivalence
	\begin{equation*}
		i^*:D_{Z^+}(Z^+\times_Z\mathfrak{g}^a)\stackrel{\simeq}{\longrightarrow} D_Z(\mathfrak{g}^a)
	\end{equation*}
	sends $\IC(Z^+\times_Z\mathbb{O},1)$ to $\IC(\mathbb{O},1)$.
	
	Recall the action map $m:Z^+\times_Z\mathfrak{g}^a \to \mathfrak{g}^a$.

    If $Z^+\cdot \mathbb{O}=\mathbb{O}$, then $\mathrm{Supp}(m_!(\IC(Z^+\times_Z\mathbb{O},1)))\subset\bar{\mathbb{O}}$ and  $m_!(\IC(Z^+\times_Z\mathbb{O},1))|_{\mathbb{O}}=\mathcal{L}_1[\mathrm{dim}\mathbb{O}]\oplus\mathcal{L}_{-1}[\mathrm{dim}\mathbb{O}]$, where $\mathcal{L}_1$ and $\mathcal{L}_{-1}$ are the local systems on $\mathbb{O}$ corresponding to the trivial and sign representation of $\pi_0(Z^+_y)\simeq\mathfrak{R}$. Since the action map $m$ is proper and small, we have $m_!(\IC(Z^+\times_Z\mathbb{O},1))=\IC_{(\mathbb{O},1)^+}\oplus \IC_{(\mathbb{O},-1)^+}$.
	
	If $Z^+\cdot \mathbb{O}=\mathbb{O}\sqcup\mathbb{O}'$, then $\mathrm{dim}\mathbb{O}=\mathrm{dim}\mathbb{O}'=\mathrm{dim}\mathbb{O}\sqcup\mathbb{O}'$, $\mathrm{Supp}(m_!(\IC(Z^+\times_Z\mathbb{O},1)))\subset\bar{\mathbb{O}}\sqcup\bar{\mathbb{O}'}$ and  $m_!(\IC(Z^+\times_Z\mathbb{O},1))|_{\mathbb{O}\sqcup\mathbb{O}'}=\mathcal{L}_1[\mathrm{dim}\mathbb{O}]$ where $\mathcal{L}_1$ is the constant local system on $\mathbb{O}\sqcup\mathbb{O}'$ corresponding to the trivial representation of $\pi_0(Z^+_y)=1$. Since the action map $m$ is proper and small, we have $m_!(\IC(Z^+\times_Z\mathbb{O},1))=\IC_{(\mathbb{O}\sqcup\mathbb{O}',1)^+}$

	Using $\mathrm{Ind}_Z^{Z^+}\IC(\mathbb{O},1)=m_!(i^*)^{-1}\IC(\mathbb{O},1)$, we get the lemma.	
\end{proof}
Thus we get the lemma:
\begin{lemma}
	The semisimple complex $\widetilde{K}^+$ consists of all simple $Z^+$-equivariant perverse sheaves on $\mathfrak{g}^a$
\end{lemma}
and the theorem
\begin{theorem}
	There is a bijection between simple $Z^+$-equivariant perverse sheaves on $\mathfrak{g}^a$ and simple $\mathrm{Hom}^*(\widetilde{K}^+,\widetilde{K}^+)^\tau$-modules, given by
	\begin{equation}
		\begin{aligned}
			\IC_{\phi^+}\mapsto \mathbf{P}_{\phi^+}&:=\mathrm{Hom}^0(\IC_{\phi^+},\widetilde{K}^{+'})^\tau\\
			&=L_{\phi^+}
		\end{aligned}
	\end{equation}
	where $\phi^+=(\mathbb{O},\chi)$ runs over all  pairs consisting of  a $Z^+$-orbit $\mathbb{O}$ in $\mathfrak{g}^a$ and an irreducible $Z^+$-equivariant local system $\chi$ on $\mathbb{O}$.
\end{theorem}
\begin{remark}
	One can prove that $\mathrm{Res}^{Z^+}_Z\mathrm{Ind}_Z^{Z^+}\widetilde{K}=\widetilde{K}\oplus\tau^*\widetilde{K}$ and that there is an embedding of subalgebras  $\mathrm{Hom}^*(\widetilde{K},\widetilde{K})\hookrightarrow 	\mathrm{Hom}^*(\widetilde{K}^+,\widetilde{K}^+)^\tau$ given by $f\mapsto f+\tau^*f$. There is an isomorphism  $\mathrm{Hom}^0(\IC_{\phi^+},\widetilde{K}) \to \mathrm{Hom}^0(\IC_{\phi^+},\widetilde{K}^{+'})^\tau$ given by $h\mapsto h+\tau^*h$, which is an isomorphism of $\operatorname{Hom}^*(\widetilde{K},\widetilde{K})$-modules. 
    
    Hence if $Z^+\cdot \mathbb{O}=\mathbb{O}$, then $\mathbf{P}_{(\mathbb{O},1)^+}$ and $\mathbf{P}_{(\mathbb{O},-1)^+}$ have the $\mathrm{Hom}^*(\widetilde{K},\widetilde{K})$-module structure as $\mathbf{P}_{(\mathbb{O},1)}$; if $Z^+\cdot \mathbb{O}=\mathbb{O}\sqcup\mathbb{O}'$, then $\mathbf{P}_{(\mathbb{O}\sqcup\mathbb{O}',1)^+}$ has the $\mathrm{Hom}^*(\widetilde{K},\widetilde{K})$-module structure as $\mathbf{P}_{(\mathbb{O},1)}\oplus\mathbf{P}_{(\mathbb{O}',1)}$.
\end{remark}
\subsubsection{Standard module}
We define $E(i_y^!\widetilde{K}^+):=\oplus_kH^k(i_y^!\widetilde{K}^+)$. Then $E(i_y^!\widetilde{K}^+)$ can be viewed as a $\mathrm{Hom}^*(\widetilde{K}^+,\widetilde{K}^+)^\tau$-module via:
\begin{equation}
\mathrm{Hom}^*(\widetilde{K}^+,\widetilde{K}^+)^\tau\to \mathrm{Hom}^*(\widetilde{K}^+,\widetilde{K}^+) \to  \mathrm{Hom}^*(i_y^!\widetilde{K}^+,i_y^!\widetilde{K}^+)
\end{equation}
Let $Z^+_y$ be the stabilizer of $y$ in $Z^+$. Then $Z^+_y$ acts naturally on the stalks $H^k(i_y^!\widetilde{K}^+)$ of each $Z^+$-equivariant IC sheaf in $\widetilde{K}^+$, and this action factors through $Z^+_y/(Z^+_y)^0$. This induces an action of $Z^+_y/(Z^+_y)^0$ on $E(i_y^!\widetilde{K}^+)$ which commutes with the $\mathrm{Hom}^*(\widetilde{K}^+,\widetilde{K}^+)^\tau$-module structure. Let $\rho$ be an irreducible representation of $Z^+_y/(Z^+_y)^0$, and denote by $E(i_y^!\widetilde{K}^+)^\rho$ the $\rho$-isotypical component. Hence $E(i_y^!\widetilde{K}^+)^\rho$ depends only on $\psi^+=(\mathbb{O},\rho)$, a $Z^+$-orbit $y\in \mathbb{O}$ and an irreducible $Z^+$-equivariant local system $\rho$ on $\mathbb{O}$. We call $\mathbf{E}_{\psi^+}:=E(i_y^!\widetilde{K}^+)^\rho$ a standard module.

Similarly the group $Z^+_y/(Z^+_y)^0$ acts naturally on the stalks $H^k(i_y^!\IC_{\phi^+})$, and we also have $H^k(i_y^!\IC_{\phi^+})^\rho$.

Now we  consider the Grothendieck group of $\mathrm{Hom}(\widetilde{K}^+,\widetilde{K}^+)^\tau$-modules of finite dimension over $\mathbb{C}$. There is a natural filtration on $E(i_y^!\widetilde{K}^+)$ by the cohomological degree $m$. The action $Z^+_y/(Z^+_y)^0$ commutes with $\mathrm{Hom}^*(\widetilde{K}^{+'},\widetilde{K}^{+'})^\tau$ and preserves the filtration.

\begin{proposition}\label{GeoMul}
	In the Grothendieck group of $\mathrm{Hom}^*(\widetilde{K}^+,\widetilde{K}^+)^\tau$-modules of finite dimension over $\mathbb{C}$, we have the equality $\mathbf{E}_{\psi^+}\simeq \bigoplus_{\phi^+}(\bigoplus_kH^k(i_y^!\IC_{\phi^+}))^\rho\otimes\mathbf{P}_{\phi^+}$, hence $m(\mathbf{P}_{\phi^+},\mathbf{E}_{\psi^+})=\sum_k\mathrm{dim}H^k(i_y^!\IC_{\phi^+})^\rho$.
\end{proposition}
\begin{proof}
	We have $(\mathrm{gr}H^k(i_y^!\IC_{\phi^+}))^\rho= \bigoplus_{\phi^+}(\bigoplus_kH^k(i_y^!\IC_{\phi^+}))^\rho\otimes L_{\phi^+}$, and the $\mathrm{Hom}^*(\widetilde{K}^+,\widetilde{K}^+)^\tau$-action on $(\mathrm{gr}H^k(i_y^!\IC_{\phi^+}))^\rho$ factors through the projection
	\begin{equation*}
		\mathrm{Hom}^*(\widetilde{K}^+,\widetilde{K}^+)^\tau\simeq\mathrm{Hom}^*(\widetilde{K}^{+'},\widetilde{K}^{+'})^\tau\to 	\mathrm{Hom}^0(\widetilde{K}^{+'},\widetilde{K}^{+'})^\tau=\bigoplus_{\phi^+}\mathrm{End}_{\mathbb{C}}L_{\phi^+}
	\end{equation*}
	Hence $\mathbf{P}_{\phi^+}=L_{\phi^+}$ occurs as a $\mathrm{Hom}^*(\widetilde{K}^+,\widetilde{K}^+)^\tau$-module  in $(\mathrm{gr}H^k(i_y^!\IC_{\phi^+}))^\rho$ exactly $\sum_k\mathrm{dim}H^k(i_y^!\IC_{\phi^+})^\rho$ times.
\end{proof}

Take $\phi^+=(\mathbb{O},\rho)$, a $Z^+$-orbit $y\in \mathbb{O}$ and an irreducible $Z^+$-equivariant local system $\rho$ on $\mathbb{O}$. We have a morphism 
\begin{equation}\label{QuoIrr}
	\mathcal{H}^*(i^!_{\mathbb{O}}\widetilde{K}^+)\to 	\mathcal{H}^*(i^*_{\mathbb{O}}\widetilde{K}^+)
\end{equation}
 obtained by applying the functor $i^*$ to the canonical morphism $i_!i^!\mathcal{F}\to \mathcal{F}$ and using that $i^*i_!i^!\mathcal{F}=i^!\mathcal{F}$.

This morphism \eqref{QuoIrr} is compatible with the $\mathrm{Hom}^*(\widetilde{K}^+,\widetilde{K}^+)$ action. Denote 
\begin{equation*}
	L_{y,\rho}=\mathrm{Im}(\mathcal{H}^*_y(i^!_{\mathbb{O}}\widetilde{K}^+)_\rho\to 	\mathcal{H}^*_y(i^*_{\mathbb{O}}\widetilde{K}^+)_{\rho})
\end{equation*}  

\begin{proposition}\label{QuoStand}
	The isomorphism $L_{\phi^+}\simeq L_{y,\rho}$ intertwines the $\mathrm{Hom}(\widetilde{K}^+,\widetilde{K}^+)^\tau$-action.
\end{proposition}
\begin{proof}
	Denote $\psi^+=(\mathbb{O}_{\psi^+},\rho_{\psi^+})$, $\phi^+=(\mathbb{O}_{\phi^+},\rho_{\phi^+})$. If $\mathbb{O}_{\psi^+}\neq \mathbb{O}_{\phi^+}$, the map  $i^!_{{\mathbb{O}}_{\psi^+}}\widetilde{K}^+ \to 	i^*_{{\mathbb{O}}_{\phi^+}}\widetilde{K}^+$ vanishes. Using the decomposition of $\widetilde{K}^+$, we have $L_{\phi^+}\simeq L_{y,\rho}$ as vector spaces. Since $\mathrm{Hom}^*(\widetilde{K}^{+'},\widetilde{K}^{+'})^\tau$ preserves the filtration of $\mathcal{H}^*_y(i^!_{\mathbb{O}}\widetilde{K}^+)$ and a similar result holds in the $i^*_{\mathbb{O}}$-case, the action $\mathrm{Hom}^*(\widetilde{K}^{+'},\widetilde{K}^{+'})^\tau$ on $L_{y,\rho}$ factors through $\mathrm{Hom}^*(\widetilde{K}^{+'},\widetilde{K}^{+'})^\tau\to 	\mathrm{Hom}^0(\widetilde{K}^{+'},\widetilde{K}^{+'})^\tau=\bigoplus_{\phi^+}\mathrm{End}_{\mathbb{C}}L_{\phi^+}$.
\end{proof}

Hence we can view all simple modules $\mathbf{P}_{\phi^+}\simeq L_{\phi^+}\simeq L_{y,\rho}$ as a simple quotient of the standard module $\mathbf{E}_{\phi^+}$.

\subsubsection{Induction theorem}

Recall $\tilde{f}^+: \dot{\mathfrak{g}}^{+,a}\to \mathfrak{g}^a$, and denote by $\mathcal{P}_y^{+,a}$ the inverse image $(\tilde{f}^+)^{-1}(y)$ of $y$. Consider the Cartesian diagram:
\begin{equation*}
	\xymatrix{
		\mathcal{P}_y^{+,a}
		\ar[r]^{\tilde{i}}
		\ar[d]_{\tilde{f}^+_y} &\dot{\mathfrak{g}}^{+,a} \ar[d]^{\tilde{f}^+} \\
		y  \ar[r]^{\tilde{i}_y} & \mathfrak{g}^a
	}
\end{equation*}
Then $H^*(\tilde{i}_y^!\widetilde{K}^+)=H^*(\tilde{i}_y^!\tilde{f}^+_*\tilde{\dot{\mathcal{L}}}^+)=H^*((\tilde{f}^+_y)_*\tilde{i}^!\tilde{\dot{\mathcal{L}}}^+)=H^*(\mathcal{P}_y^{+,a},\tilde{i}^!\tilde{\dot{\mathcal{L}}}^+)=H^*(\mathcal{P}_y^{+,a},D\tilde{i}^*D\tilde{\dot{\mathcal{L}}}^+)=H^*(\mathcal{P}_y^{+,a},D\tilde{i}^*(\tilde{\dot{\mathcal{L}}}^+)^*)=H^*(\mathcal{P}_y^{+,a},D(\tilde{\dot{\mathcal{L}}}^+)^*)=H_*^{\{1\}}(\mathcal{P}_y^{+,a},\tilde{\dot{\mathcal{L}}}^+)$.

We also have $f^+: \dot{\mathfrak{g}}^{+}\to \mathfrak{g}$ and $f: \dot{\mathfrak{g}}\to \mathfrak{g}$. Denote by $\mathcal{P}_y^{+}$ the inverse image $({f}^+)^{-1}(y)$ of y, and by $\mathcal{P}_y$ the inverse image ${f}^{-1}(y)$ of y. Then in our case $\mathcal{P}_y^{+,a}=\mathcal{P}_y^{+}$. Since $H_c^{\mathrm{odd}}(\mathcal{P}_y,\dot{\mathcal{L}})=0$, we have $H_c^{\mathrm{odd}}(\mathcal{P}_y^+,\dot{\mathcal{L}}^+)=0$. Then 
\begin{align}
	H_*^{\{1\}}(\mathcal{P}_y^{+,a},\tilde{\dot{\mathcal{L}}}^+)&\simeq\mathbb{C}_a\otimes_{H^*_{T_a}} H_*^{T_a}(\mathcal{P}_y^{+,a},\tilde{\dot{\mathcal{L}}}^+)\nonumber\\
	&\simeq \mathbb{C}_a\otimes_{H^*_{M_y(a)^0}}H_*^{M_y(a)^0}(\mathcal{P}_y^{+,a},\tilde{\dot{\mathcal{L}}}^+)\nonumber\\
	&\simeq \mathbb{C}_a\otimes_{H^*_{M_y(a)^0}}H_*^{M_y(a)^0}(\mathcal{P}_y^{+},\tilde{\dot{\mathcal{L}}}^+)\nonumber\\ &\simeq\mathbb{C}_a\otimes_{H^*_{M_y^0}}H^{M_y^0}_*(\mathcal{P}_y^{+},{\dot{\mathcal{L}}}^+)\label{ModRed}
\end{align}

where $M_y$ is the stabilizer of $y$ in $G^+\times\mathbb{C}^*$, $a=(\bar{t},r_0)$, and $M_y(a)$ is the stabilizer of $a$ in $M_y$.

In \cite{Aubert2016GradedHA}, we also have a $\mathbb{H}(G^+)\times(Z_{G^+\times\mathbb{C}^*}(a,y)/Z_{G^+\times\mathbb{C}^*}^0(a,y))$-module structure on  $\mathbb{C}_a\otimes_{H_{M_y^0}}H^{M_y^0}_*(\mathcal{P}_y^{+},{\dot{\mathcal{L}}}^+)$.

\begin{theorem}\label{StandIso}
	There is an isomorphism of $\mathbb{H}(G^+)\times(Z_{G^+\times\mathbb{C}^*}(a,y)/Z_{G^+\times\mathbb{C}^*}^0(a,y)) $-representation between $H^*(i_y^!\widetilde{K}^+)$ and $\mathbb{C}_a\otimes_{H_{M_y^0}}H^{M_y^0}_*(\mathcal{P}_y^{+},{\dot{\mathcal{L}}}^+)$.
\end{theorem}
\begin{proof}
	For any reductive group $A$, $A$-variety $N$ and $\mathcal{F}\in D_A(N)$, we have $H^*_A(N,\mathcal{F})\simeq \mathrm{Hom}^*_A(\mathbb{C}_N,\mathcal{F})$, so $H^*_A(N,\mathcal{F})$ has a $\mathrm{Hom}^*_A(\mathcal{F},\mathcal{F})$-module structure given by 
	\begin{equation}\label{ActTransHom1}
		\mathrm{Hom}^*_A(\mathbb{C}_N,\mathcal{F})\times\mathrm{Hom}^*_A(\mathcal{F},\mathcal{F})\to \mathrm{Hom}^*_A(\mathbb{C}_N,\mathcal{F})
	\end{equation}
	Let $A'$ be a reductive group and $\epsilon:A'\to A$ a group homomorphism. For an $A'$-variety $N'$ and an $\epsilon$-map $s: N'\to N$, we can define the functor $s^*:D_A'(N')\to D_A(N)$. We have the commutative diagram 
	\begin{equation}\label{ActBasCh1}
		\begin{tikzcd}
			\mathrm{Hom}^*_A(s_!\mathbb{C}_N,\mathcal{F})\times\mathrm{Hom}^*_A(\mathcal{F},\mathcal{F}) \ar[r] \ar[d]& \mathrm{Hom}^*_A(s_!\mathbb{C}_N,\mathcal{F})\ar[d]\\
			\mathrm{Hom}^*_{A'}(\mathbb{C}_{N'},s^!\mathcal{F})\times\mathrm{Hom}^*_{A'}(s^!\mathcal{F},s^!\mathcal{F}) \ar[r] & \mathrm{Hom}^*_{A'}(\mathbb{C}_{N'},s^!\mathcal{F})
		\end{tikzcd}
	\end{equation}

	Let	$\widetilde{\mathcal{O}}=(G^+\times\mathbb{C}^*)/M_y^0\stackrel{h}{\longrightarrow}\mathfrak{g}$ be defined by $(g_1,\lambda)\mapsto\lambda^{-2}\mathrm{Ad}(g_1)y$	and $\dot{\widetilde{\mathcal{O}}}=(G^+\times\mathbb{C}^*\times\mathcal{P}_y^{+})/M_y^0\stackrel{\dot{h}}{\longrightarrow}\dot{\mathfrak{g}}^+$ be defined by $(g_1,\lambda,gB)\mapsto(\lambda^{-2}\mathrm{Ad}(g_1)y,g_1gB)$.
	
	We have the Cartesian diagram:
	\begin{equation}
		\xymatrix{
			\dot{\widetilde{\mathcal{O}}}
			\ar[r]^{\dot{h}}
			\ar[d]_{{f}^+_{\mathcal{O}}} &\dot{\mathfrak{g}}^{+} \ar[d]^{{f}^+} \\
			\widetilde{\mathcal{O}}  \ar[r]^{h} & \mathfrak{g}
		}
	\end{equation}
    Then
	\begin{align*}
		H^{G^+\times\mathbb{C}^*}_*(\dot{\widetilde{\mathcal{O}}},\dot{h}^*{\dot{\mathcal{L}}}^+)&=H_{G^+\times\mathbb{C}^*}^*(\dot{\widetilde{\mathcal{O}}},D\dot{h}^*({\dot{\mathcal{L}}}^+)^*)=H_{G^+\times\mathbb{C}^*}^*(\dot{\widetilde{\mathcal{O}}},\dot{h}^!{\dot{\mathcal{L}}}^+)\\
		&=H_{G^+\times\mathbb{C}^*}^*({\widetilde{\mathcal{O}}},({f}^+_{\mathcal{O}})_*\dot{h}^!{\dot{\mathcal{L}}}^+)=H_{G^+\times\mathbb{C}^*}^*({\widetilde{\mathcal{O}}},h^!(f^+)_*{\dot{\mathcal{L}}}^+)=H_{G^+\times\mathbb{C}^*}^*({\widetilde{\mathcal{O}}},h^!K^+)
	\end{align*}
	Hence we have a  $\mathrm{Hom}^*_{G^+\times\mathbb{C}^*}(K^+,K^+)$-action on $H^{G^+\times\mathbb{C}^*}_*(\dot{\widetilde{\mathcal{O}}},\dot{h}^*{\dot{\mathcal{L}}}^+)\simeq H_{G^+\times\mathbb{C}^*}^*({\widetilde{\mathcal{O}}},h^!K^+)$ given by 
	\begin{equation}
		\mathrm{Hom}^*_{D_{G^+\times\mathbb{C}^*}(\mathfrak{g})}(K^+,K^+)\stackrel{h^!}{\longrightarrow}\mathrm{Hom}^*_{D_{G^+\times\mathbb{C}^*}({\widetilde{\mathcal{O}}})}(h^!K^+,h^!K^+)
	\end{equation}
	
	Since $\mathcal{P}_y^{+}$ is a closed $M_y^0$-stable subvariety of the $G^+\times\mathbb{C}^*$-variety $\dot{\widetilde{\mathcal{O}}}$ and the map $M_y^0\backslash G^+\times\mathbb{C}^*\times\mathcal{P}_y^{+}\simeq \dot{\widetilde{\mathcal{O}}}$ is an isomorphism of $G^+\times\mathbb{C}^*$-variety, we have an induction equivalence $j_1^*:D_{ G^+\times\mathbb{C}^*}(\dot{\widetilde{\mathcal{O}}})\to D_{M_y^0}(\mathcal{P}_y^{+})$, where  $\dot{j_1}:\mathcal{P}_y^{+}\hookrightarrow\dot{\widetilde{\mathcal{O}}}$.
	
	We have the Cartesian diagram:
	\begin{equation}
		\xymatrix{
			\mathcal{P}_y^{+}
			\ar[r]^{\dot{j_1}}
			\ar[d]_{{f}^+_y} &	\dot{\widetilde{\mathcal{O}}} \ar[d]^{{f}^+_{\mathcal{O}}} \\
			y  \ar[r]^{j_1} & \widetilde{\mathcal{O}}
		}
	\end{equation}
	Using the same argument, we have a $\mathrm{Hom}^*_{G^+\times\mathbb{C}^*}(K^+,K^+)$-action on $H^{M_y^0}_*(\mathcal{P}_y^{+},j_1^*\dot{h}^*{\dot{\mathcal{L}}}^+)=H_{M_y^0}^*(y,j_1^!{h}^!K^+)$ given by 
	
	\begin{align}
		\mathrm{Hom}^*_{D_{G^+\times\mathbb{C}^*}(\mathfrak{g})}(K^+,K^+)&\stackrel{h^!}{\longrightarrow}\mathrm{Hom}^*_{D_{G^+\times\mathbb{C}^*}({\widetilde{\mathcal{O}}})}(h^!K^+,h^!K^+)\nonumber\\
		&\stackrel{j_1^!}{\longrightarrow}\mathrm{Hom}^*_{D_{M_y^0}(y)}(j_1^!h^!K^+,j_1^!h^!K^+)\label{EquActGeo1}
	\end{align}
	
	Using the induction equivalence $j_1^*:D_{ G^+\times\mathbb{C}^*}(\dot{\widetilde{\mathcal{O}}})\to D_{M_y^0}(\mathcal{P}_y^{+})$, we have $H^{M_y^0}_*(\mathcal{P}_y^{+},j_1^*\dot{h}^*{\dot{\mathcal{L}}}^+)\simeq H^{G^+\times\mathbb{C}^*}_*(\dot{\widetilde{\mathcal{O}}},\dot{h}^*{\dot{\mathcal{L}}}^+)$. Combining with  \eqref{ActComDia} we can reformulate the $\Delta$ action on $H^{M_y^0}_*(\mathcal{P}_y^{+},{\dot{\mathcal{L}}}^+)$ in \cite[Section3]{Aubert2016GradedHA} or \cite[8.1]{lusztig1} given by $H^{M_y^0}_*(\mathcal{P}_y^{+},j_1^*\dot{h}^*{\dot{\mathcal{L}}}^+)\simeq H_{M_y^0}^*(y,j_1^!{h}^!K^+)$ and \eqref{EquActGeo1}.
	
	Denote $j_y=j_1\circ h$. Using \eqref{ActTransHom1}, we transfer the $\mathrm{Hom}^*_{D_{M_y^0}(y)}(j_y^!K^+,j_y^!K^+)$-action on $H_{M_y^0}^*(y,j_y^!K^+)$:
	\begin{equation}
		\mathrm{Hom}^*_{D_{M_y^0}(y)}(\mathbb{C}_y,j_y^!K^+)\times\mathrm{Hom}^*_{D_{M_y^0}(y)}(j_y^!K^+,j_y^!K^+)\to \mathrm{Hom}^*_{D_{M_y^0}(y)}(\mathbb{C}_y,j_y^!K^+)
	\end{equation}
	Using \eqref{ActBasCh1}, we can transfer this action to 
	\begin{equation}\label{HomP1}
		\mathrm{Hom}^*_{D_{G^+\times\mathbb{C}^*}(\mathfrak{g})}((j_y)_!\mathbb{C},K^+)\times\mathrm{Hom}^*_{D_{G^+\times\mathbb{C}^*}(\mathfrak{g})}(K^+,K^+)\to \mathrm{Hom}^*_{D_{G^+\times\mathbb{C}^*}(\mathfrak{g})}(j_y)_!\mathbb{C},K^+)
	\end{equation}
	
	Recall that we denote $a=(\bar{t},r_0)$ and $Z=Z_{G\times \mathbb{C}^*}(a)$. Applying $\mathrm{Res}^{G\times \mathbb{C}^*}_{Z}$ to \eqref{HomP1}, we have 
	
	\begin{equation*}
		\begin{tikzcd}[column sep=small]
			\mathrm{Hom}^*_{D_{G\times \mathbb{C}^*}(\mathfrak{g})}((j_y)_!\mathbb{C},{K}^+)\times\mathrm{Hom}^*_{D_{G\times \mathbb{C}^*}(\mathfrak{g})}({K}^+,{K}^+) \ar[r] \ar[d,"\mathrm{Res}^{G\times \mathbb{C}^*}_{Z}"]& \mathrm{Hom}^*_{D_{G\times \mathbb{C}^*}(\mathfrak{g})}((j_y)_!\mathbb{C},{K}^+)\ar[d,"\mathrm{Res}^{G\times \mathbb{C}^*}_{Z}"]\\
			\mathrm{Hom}^*_{D_{Z}(\mathfrak{g})}((j_y)_!\mathbb{C},{K}^+)\times\mathrm{Hom}^*_{D_{Z}(\mathfrak{g})}({K}^+,{K}^+) \ar[r]&  \mathrm{Hom}^*_{D_{Z}(\mathfrak{g})}((j_y)_!\mathbb{C},{K}^+)
		\end{tikzcd}
	\end{equation*}

	Recall the set up of \eqref{ProofLocal2}.
	
	We have
	\begin{equation*}
		(\dot{j}^+)_!(\dot{j}^+)^!\dot{\mathcal{L}}^+\stackrel{\mu}{\longrightarrow}\dot{\mathcal{L}}^+\stackrel{\nu}{\longrightarrow}(\dot{j}^+)_*(\dot{j}^+)^*\dot{\mathcal{L}}^+
	\end{equation*}
	
	and canonical morphisms of $H_Z^*$-algebras
	\begin{equation*}
		\begin{tikzcd}
			\mathrm{Hom}^*_{D_{Z}(\mathfrak{g})}\Bigl((f^+)_*\dot{\mathcal{L}}^+,(f^+)_*\dot{\mathcal{L}}^+\Bigr)\\
			\mathrm{Hom}^*_{D_{Z}(\mathfrak{g})}\Bigl((f^+)_*(\dot{j}^+)_*(\dot{j}^+)^*\dot{\mathcal{L}}^+,(f^+)_*(\dot{j}^+)_!(\dot{j}^+)^!\dot{\mathcal{L}}^+\Bigr)
			\arrow[d, "\Psi_2"']\arrow[u, "\Psi_1"] \\
			\mathrm{Hom}^*_{D_{Z}(\mathfrak{g})}((f^+)_*(\dot{j}^+)_*(\dot{j}^+)^*\dot{\mathcal{L}}^+,(f^+)_*(\dot{j}^+)_*(\dot{j}^+)^*\dot{\mathcal{L}}^+)
		\end{tikzcd}	
	\end{equation*}

	defined by $\Psi_1(u)=(f^+)_*(\mu)\circ u \circ (f^+)_*(\nu)$, $\Psi_2(u)=(f^+)_*(\nu\mu)\circ u$.

	Hence we can define the morphism:
	\begin{equation*}
		\mathrm{Hom}^*_{D_{Z}(\mathfrak{g})}((j_y)_!\mathbb{C},(f^+)_*\dot{\mathcal{L}}^+)\stackrel{\Psi}{\longrightarrow}\mathrm{Hom}^*_{D_{Z}(\mathfrak{g})}((j_y)_!\mathbb{C},(f^+)_*(\dot{j}^+)_*(\dot{j}^+)^*\dot{\mathcal{L}}^+)	
	\end{equation*}
 by $\Psi(u')=(f^+)_*(\nu)\circ u'$.
	
	Thus we have the commutative diagram:

	\begin{equation*}
		\begin{tikzcd}[column sep=small, row sep=scriptsize]
			\mathrm{Hom}^*_{D_{Z}(\mathfrak{g})}((j_y)_!\mathbb{C},K^+)\times\mathrm{Hom}^*_{D_{Z}(\mathfrak{g})}(K^+,K^+) \ar[r] & \mathrm{Hom}^*_{D_{Z}(\mathfrak{g})}((j_y)_!\mathbb{C},K^+)\ar[dd,"\Psi"]\\
			\mathrm{Hom}^*_{D_{Z}(\mathfrak{g})}((j_y)_!\mathbb{C},K^+)\times\mathrm{Hom}^*_{D_{Z}(\mathfrak{g})}\Bigl(K^+,(f^+)_*(\dot{j}^+)_!(\dot{j}^+)^!\dot{\mathcal{L}}^+\Bigr) \ar[u,"1\times\Psi_1"] \ar[d,"\Psi\times\Psi_2"']&  \\
			\mathrm{Hom}^*_{D_{Z}(\mathfrak{g})}((j_y)_!\mathbb{C},\widetilde{K}^+)\times\mathrm{Hom}^*_{D_{Z}(\mathfrak{g})}(\widetilde{K}^+,\widetilde{K}^+) \ar[r]& 	\mathrm{Hom}^*_{D_{Z}(\mathfrak{g})}((j_y)_!\mathbb{C},\widetilde{K}^+)
		\end{tikzcd}
	\end{equation*}
	because $\Psi(\Psi_1(u)\circ u')=(f^+)_*(\nu)\circ (f^+)_*(\mu)\circ u \circ (f^+)_*(\nu) \circ u'=(f^+)_*(\nu\mu)\circ u\circ (f^+)_*(\nu) \circ u'=\Psi_2(u)\circ \Psi(u')$.
	
	Using the group homomorphism $T_a\to Z$ and $T_a\to 1$ which give rise to a functor $D(\mathfrak{g}^a)\stackrel{\Upsilon}{\longrightarrow} D_{T_a}(\mathfrak{g}^a)$, we get 
	
	\begin{equation*}
		\begin{tikzcd}[column sep=small]
			\mathrm{Hom}^*_{D_{Z}(\mathfrak{g})}((j_y)_!\mathbb{C},\widetilde{K}^+)\times\mathrm{Hom}^*_{D_{Z}(\mathfrak{g})}(\widetilde{K}^+,\widetilde{K}^+) \ar[r] \ar[d,"\mathrm{Res}^Z_{T_a}"]& \mathrm{Hom}^*_{D_{Z}(\mathfrak{g})}((j_y)_!\mathbb{C},\widetilde{K}^+)\ar[d,"\mathrm{Res}^Z_{T_a}"]\\
			\mathrm{Hom}^*_{D_{T_a}(\mathfrak{g})}((j_y)_!\mathbb{C},\widetilde{K}^+)\times\mathrm{Hom}^*_{D_{T_a}(\mathfrak{g})}(\widetilde{K}^+,\widetilde{K}^+) \ar[r]&  \mathrm{Hom}^*_{D_{T_a}(\mathfrak{g})}((j_y)_!\mathbb{C},\widetilde{K}^+)\\
			\mathrm{Hom}^*_{D_{1}(\mathfrak{g})}((j_y)_!\mathbb{C},\widetilde{K}^+)\times\mathrm{Hom}^*_{D_{1}(\mathfrak{g})}(\widetilde{K}^+,\widetilde{K}^+)\ar[u,"\Upsilon"] \ar[r]& 	\mathrm{Hom}^*_{D_{1}(\mathfrak{g})}((j_y)_!\mathbb{C},\widetilde{K}^+)\ar[u,"\Upsilon"]
		\end{tikzcd}
	\end{equation*}
	
	Using \eqref{ActBasCh1}, we can transfer \begin{equation*}
		\begin{tikzcd}
			\mathrm{Hom}^*_{D_{1}(\mathfrak{g})}((j_y)_!\mathbb{C},\widetilde{K}^+)\times\mathrm{Hom}^*_{D_{1}(\mathfrak{g})}(\widetilde{K}^+,\widetilde{K}^+) \ar[r]& 	\mathrm{Hom}^*_{D_{1}(\mathfrak{g})}((j_y)_!\mathbb{C},\widetilde{K}^+)
		\end{tikzcd}
	\end{equation*}
	to 
	\begin{equation*}
		\begin{tikzcd}
			\mathrm{Hom}^*_{D(y)}(\mathbb{C},j_y^!\widetilde{K}^+)\times\mathrm{Hom}^*_{D(y)}(j_y^!\widetilde{K}^+,j_y^!\widetilde{K}^+) \ar[r]& 	\mathrm{Hom}^*_{D(y)}((\mathbb{C},j_y^!\widetilde{K}^+)
		\end{tikzcd}
	\end{equation*}

	Using all the commutative diagrams above together with \eqref{HomRed1}, \eqref{HomRedF} and \eqref{ModRed}, we get the theorem.
	
\end{proof}

\begin{remark}
	Recall $a=(\bar{t},r_0)$ and $Z^+=Z_{G^+\times\mathbb{C}^*}(a)=Z_{G^+}(\bar{t})\times\mathbb{C}^*$. For a nilpotent element $y$, it is possible to choose an algebraic homomorphism $\varphi_y:\operatorname{SL}_2(\mathbb{C})\to G$ such that $d_{\varphi_y}\begin{pmatrix}
		0 &1\\
		0& 0
	\end{pmatrix}=y$. There is a homomorphism $Z_{G^+}(y)\times \mathbb{C}^*\to Z_{G^+\times\mathbb{C}^*}(y)$ given by $(g,r)\mapsto g\varphi_y\begin{pmatrix}
		r &0\\
		0& r^{-1}
	\end{pmatrix}$. Hence $\pi_0(Z^+_y)=\pi_0(Z_{G^+\times\mathbb{C}^*}(a,y))=Z_{G^+\times\mathbb{C}^*}(a,y)/Z_{G^+\times\mathbb{C}^*}^0(a,y)\simeq Z_{G^+}(\bar{t},y)/Z_{G^+}^0(\bar{t},y)=\pi_0(Z_{G^+}(\bar{t},y))$.
\end{remark}

\begin{remark}
	We use notation similar to that in  \cite{Aubert2016GradedHA}, $E_{y,\bar{t},r_0}^+=\mathbb{C}_a\otimes_{H_{M_y^0}}H^{M_y^0}_*(\mathcal{P}_y^{+},{\dot{\mathcal{L}}}^+)$, where $a=(\bar{t},r_0)$, and $E_{y,\bar{t},r_0,\rho}^+=\mathrm{Hom}_{\pi_0(Z^+_y)}(\rho,E_{y,\bar{t},r_0}^+)$. 
\end{remark}

Therefore, from Theorem \ref{StandIso} we have 
\begin{theorem}\label{StandIso2}
	Take $\tau\in \mathfrak{R}$ and $\bar{t}$ satisfying $\tau\cdot \bar{t} \subset G\cdot \bar{t}$ and set $a=(\bar{t},r_0)$. Let $\psi^+=(\mathbb{O},\rho)$, a $Z^+$-orbit $\mathbb{O}$ on $\mathfrak{g}^a$ and an irreducible $Z^+$-equivariant local system $\rho$ on $\mathbb{O}$, and $y\in \mathbb{O}$.
	
	There is an isomorphism of $\mathbb{H}(G^+)$-representation between $\mathbf{E}_{\psi^+}$ and $E_{y,\bar{t},r_0,\rho}^+$.
\end{theorem}

Take a subgroup $Q^+$ of $G^+$, and let $	\mathbf{E}_{\psi^+(Q^+)}$ be the analog of the standard module of $\mathbb{H}(Q^+)$.

\begin{theorem}\label{Ind31}[{\cite[Thm3.4]{Aubert2016GradedHA}}]
	Take $t_0\in \mathfrak{t}_{\mathbb{R}}$, $r_0> 0$, $a=(t_0,r_0)$, and $t_0':=t_0-d_{\varphi_y}(\begin{pmatrix}
		r_0 &0\\
		0& -r_0
	\end{pmatrix})$ satisfying  $\alpha(t_0')\leq 0$ for all positive roots $\alpha$.  Then
	\begin{equation}
	\mathbb{H}(G^+)\otimes_{\mathbb{H}(Q^+)}\mathbf{E}_{\psi^+(Q^+)}\simeq \mathbf{E}_{\psi^+(G^+)}
	\end{equation}
	
	as an $\mathbb{H}(G^+)$-module.
\end{theorem}

Recall that $G$ is the identity component group of $G^+$ and we assume that $Z_y$ is connected for all $y$. Then we have the standard module $E_{y,\bar{t},r_0}$ of graded Hecke algebra $\mathbb{H}(G)$, where $[\bar{t},y]=2r_0y$.

\begin{proposition}\label{GeoActM}[{\cite[Lemma3.18]{Aubert2016GradedHA}}]
	Take $\tau\in \mathfrak{R}$ and $\bar{t}$ satisfying $\tau\cdot \bar{t} \subset G\cdot \bar{t}$, and set $a=(\bar{t},r_0)$. Take a $Z$-orbit $\mathbb{O}$ on $\mathfrak{g}^a$ and $y\in \mathbb{O}$.

	 Denote by $\tau^* E_{y,\bar{t},r_0,1}$ the ${\mathbb{H}(G)}$-module which has the same underlying vector space $E_{y,\bar{t},r_0,1}$ but with ${\mathbb{H}(G)}$-action given by $\tau ^*(h)\cdot m=\tau(h)\cdot m$.
	
	If $Z_y^+=Z_y=1$, then $Z^+\cdot \mathbb{O}=\mathbb{O}\sqcup\mathbb{O}'$, and we have $E_{y,\bar{t},r_0,1}^+=\mathrm{Ind}_{\mathbb{H}(G)}^{\mathbb{H}(G^+)}E_{y,\bar{t},r_0,1}$

	If $Z^+_y/Z_y\simeq \mathfrak{R}$, then $Z^+\cdot \mathbb{O}=\mathbb{O}$, and we can define a geometric normalization $J^\tau$ as in \cite[Proposition3.15]{Aubert2016GradedHA} such that  $J^\tau$ is an isomorphism of $\mathbb{H}(G)$-modules between $E_{y,\bar{t},r_0,1}$ and  $\tau^* E_{y,\bar{t},r_0,1}$. Hence there is an isomorphism of $\mathbb{H}(G^+)$-modules $E_{y,\bar{t},r_0,\rho}^+\simeq  E_{y,\bar{t},r_0,1}\otimes V_\rho$ where $E_{y,\bar{t},r_0,1}\otimes V_\rho$ has the $\mathbb{H}(G^+)=\mathbb{H}(G)\rtimes\mathfrak{R}$-action given by
	\begin{equation*}
		(h,\tau)\cdot (m\otimes n)=h\cdot J^\tau(m) \otimes \rho(\tau)\cdot n \quad h\in \mathbb{H}(G), \tau \in \mathfrak{R}, m\in E_{y,\bar{t},r_0,1}, n\in V_\rho
	\end{equation*} 
\end{proposition}
\begin{proposition}[{\cite[Theorem3.20]{Aubert2016GradedHA}}]\label{StandMod2}
	Take $r_0> 0$, $\tau\in \mathfrak{R}$ and $\bar{t}$ satisfying $\tau\cdot \bar{t} \subset G\cdot \bar{t}$, and set $a=(\bar{t},r_0)$. Take a $Z$-orbit $\mathbb{O}$ on $\mathfrak{g}^a$ and $y\in \mathbb{O}$.
	
	If $Z_y^+=Z_y=1$, then $E_{y,\bar{t},r_0,1}^+\simeq\mathrm{Ind}_{\mathbb{H}(G)}^{\mathbb{H}(G^+)}E_{y,\bar{t},r_0,1}$ has a unique irreducible quotient $\mathbb{H}(G^+)$-module $\mathrm{Ind}_{\mathbb{H}(G)}^{\mathbb{H}(G^+)}M_{y,\bar{t},r_0,1}$, where $M_{y,\bar{t},r_0,1}$ is the unique irreducible quotient $\mathbb{H}(G)$-module of $E_{y,\bar{t},r_0,1}$. Denote $M_{y,\bar{t},r_0,1}^+=\mathrm{Ind}_{\mathbb{H}(G)}^{\mathbb{H}(G^+)}M_{y,\bar{t},r_0,1}$.

	If $Z^+_y/Z_y\simeq \mathfrak{R}$, then $E_{y,\bar{t},r_0,\rho}^+\simeq  E_{y,\bar{t},r_0,1}\otimes V_\rho$ has a unique irreducible quotient $\mathbb{H}(G^+)$-module $M_{y,\bar{t},r_0,1}\otimes V_\rho$,  where $M_{y,\bar{t},r_0,1}$ is the unique irreducible quotient $\mathbb{H}(G)$-module of $E_{y,\bar{t},r_0,1}$ and $M_{y,\bar{t},r_0,1}\otimes V_\rho$ has the $\mathbb{H}(G^+)=\mathbb{H}(G)\rtimes\mathfrak{R}$-action given by
	\begin{equation*}
		(h,\tau)\cdot (m\otimes n)=h\cdot J^\tau(m) \otimes \rho(\tau)\cdot n \quad h\in \mathbb{H}(G), \tau \in \mathfrak{R}, m\in M_{y,\bar{t},r_0,1}, n\in V_\rho
	\end{equation*} We denote this $\mathbb{H}(G^+)$-module by $M_{y,\bar{t},r_0,\rho}^+$.
\end{proposition}

Using $\mathbf{E}_{\phi^+}\simeq E_{y,\bar{t},r_0,\rho}^+$ from Theorem \ref{StandIso2}, the fact that $\mathbf{P}_{\phi^+}\simeq L_{\phi^+}\simeq L_{y,\rho}$ is a simple quotient of the standard module $\mathbf{E}_{\phi^+}$ in Proposition \ref{QuoStand} and the uniqueness in Proposition \ref{StandMod2}, we have 
\begin{theorem}
	There is an isomorphism of $\mathbb{H}(G^+)$-modules $\mathbf{P}_{\phi^+}\simeq L_{\phi^+}\simeq L_{y,\rho}\simeq M_{y,\bar{t},r_0,1}^+$.
\end{theorem}

\section{Langlands parameters}
\subsection{Connected case}
In $\GL_n$, we can write a Levi subgroup as $\mathbb{M}=\GL_{n_1}\times\cdots\times \GL_{n_k}$. For $\pi_i\in \mathrm{Rep}(\GL_{n_i}(\mathbb{Q}_p))$, we write $\pi_1\times \cdots \times \pi_k \in \mathrm{Rep}(\GL_n(\mathbb{Q}_p))$ for the normalized parabolic induction $i_{\mathbb{P}(\mathbb{Q}_p)}(\pi_1\otimes\cdots\otimes\pi_r)$ via a parabolic subgroup $\mathbb{P}(\mathbb{Q}_p)$ containing $\mathbb{M}(\mathbb{Q}_p)$ and ${\mathbb{B}}(\mathbb{Q}_p)$.

Let $F=\mathbb{Q}_p$, $\varrho_0^i=|\cdot|_F^{z_i}$, $z_i\in \mathbb{C}$, $\varrho_1^i=\varrho_0|\cdot|_F^{\frac{1-n_i}{2}}$, $\varrho_2^i=\varrho_0|\cdot|_F^{\frac{3-n_i}{2}}$, $\cdots$, $\varrho_{n_i}^i=\varrho_0|\cdot|_F^{\frac{n_i-1}{2}}$, which are one dimensional representations of $\GL_1(F)=F^\times$. Let $\Delta_i=[\varrho_1^i,\varrho_{n_i}^i]$ be a segment. 
Then we have the unique irreducible quotient  $\St\langle \Delta_i\rangle$ of $\varrho_1\times\varrho_2\times\cdots\times\varrho_{n_i}$ in $\Rep(\GL_{n_i}(F))$.

A sequence of segments $(\Delta_1,\cdots,\Delta_k)$ is said to be ordered if for every $1\leq i < j \leq k$, $\Delta_i$ does not precede $\Delta_j$. 

Let $\pi \in \mathrm{Rep}(\GL_n(\mathbb{Q}_p))$. Recall that the socle (resp., cosocle) of $\pi$, denoted by $\mathrm{soc}(\pi)$ (resp., $\mathrm{cos}(\pi)$), is the largest semisimple subrepresentation (resp., quotient) of $\pi$. 

For any ordered form $(\Delta_1,\cdots,\Delta_k)$ of a multisegment $\mathfrak{m}$, the representation $\St\langle \Delta_1\rangle \times \cdots \times \St\langle \Delta_k\rangle$ has a unique irreducible quotient \begin{align}
	\mathrm{cos}(\St\langle \Delta_1\rangle \times \cdots \times \St\langle \Delta_k\rangle)
\end{align} 
Then we get the Langlands classification of  all irreducible representations in $Rep(\mathcal{G})^{[\mathcal{T},1_{triv}]}$.

Denote $\mathcal{G}=\GL_n(\Qp)$ and $G=\GL_n(\mathbb{C})$,  the dual group of $\mathcal{G}$.
Using 
\begin{equation}\label{PRepEq1}
	\Lambda_{\mathcal{J}}: 	\mathrm{Rep}(\mathcal{G})^{[\mathcal{T},1_{triv}]}\simeq\mathrm{Rep}(\mathcal{G})^{(\mathcal{J},1_{triv})}\simeq  \mathcal{H}(\mathcal{J}\backslash\mathcal{G}/\mathcal{J})-\mathrm{Mod}
\end{equation}
\begin{equation}\label{PRepEq2}
	\mathcal{H}(\mathcal{J}\backslash\mathcal{G}/\mathcal{J})\simeq \mathcal{H}(G,p^{1/2})\end{equation}
\begin{equation}\label{PRepEq3}
	\Theta:	\mathbb{H}(Z_{G}(t_c))-\mathrm{Mod}_{(\mathrm{log}t_h, \mathrm{log}p^{1/2})}\simeq\mathcal{H}(G,v)-\mathrm{Mod}_{(t,p^{1/2})}
\end{equation} 

where $t=t_ct_h\in G$ is semisimple, we can classify all irreducible representations in $\mathrm{Rep}(\mathcal{G})^{(\mathcal{J},1_{triv})}$. 

We have the Iwahori-Matsumoto involution $\IM$ on $\mathbb{H}(Z_{G}(t_c))$ given by 
\begin{equation}
\begin{aligned}
	\IM(N_w)&=\mathrm{sign}(w)N_w\quad &(w\in W) \\
	\IM(\xi)&=-\xi \quad &(\xi\in \mathfrak{t}^*)
\end{aligned}
\end{equation}

For a representation $(\pi,V)$ of  $\mathbb{H}(Z_{G}(t_c))$, denote by $(\IM^*\pi,V)$ the representation twisted by $\IM$:\begin{align}
	(\IM^*\pi)(g)\cdot v:=\pi(\IM(g))\cdot v
\end{align}

Moreover, using the compatibility between these equivalences and induction, we can obtain a correspondence between the standard representations in $\mathrm{Rep}(\mathcal{G})^{(\mathcal{J},1_{triv})}$ and the standard modules in $\mathbb{H}(Z_{G}(t_c))-\mathrm{Mod}$.

Denote by $\nu_{n_i}$  the irreducible $n_i$-dimensional  representation of $\mathrm{SL_2(\mathbb{C})}$ and set $\nu=\oplus_{i=1}^k\nu_{n_i}$. Let $y=d_\nu\begin{pmatrix}
	0 &1\\
	0& 0
\end{pmatrix}$, $t'=\times_{{n_1}}(p^{z_1})\times\cdots\times\times_{{n_k}}(p^{z_k})$, $t'_h=|t'|$, $\bar{t}=-\mathrm{log}t'_h+d_\nu\begin{pmatrix}
	r_0 &0\\
	0& -r_0
\end{pmatrix}$, where $r_0=\mathrm{log}p^{1/2}$.

Hence the standard module 
$\IM^*E_{y,\bar{t},r_0,1}$ of $\mathbb{H}(Z_{G}(t_c))$ corresponds to $\St \langle \Delta_1\rangle \times \cdots \times \St\langle \Delta_k\rangle$ in $\mathrm{Rep}(\mathcal{G})^{[\mathcal{T},1_{triv}]}$. The irreducible quotient module $\IM^*M_{y,\bar{t},r_0,1}$  corresponds to $\mathrm{cos}(\St\langle \Delta_1\rangle \times \cdots \times \St\langle \Delta_k\rangle)$ through \eqref{PRepEq1} and \eqref{PRepEq3}.

 Denote by $^L\mathcal{G}=G\times W_F$ the $L$-group of $G$, where $W_F$ is the Weil group of $F$.

\begin{definition}
	A  Langlands parameter of $\mathcal{G}$ is a continuous group homomorphism
	\begin{equation*}
		\phi: W_F\times \mathrm{SL}_2(\mathbb{C}) \to {^L\mathcal{G}} 
	\end{equation*}
	such that 
	
	(1) it respects the projections to $W_F$ for both  $W_F\times \mathrm{SL}_2(\mathbb{C})$ and ${^L\mathcal{G}}$;
	
	(2) its restriction to $\mathrm{SL}_2(\mathbb{C})$ is algebraic;
	
	(3) the image of $W_F$ consists of semisimple elements in ${^L\mathcal{G}}$.
\end{definition}
Let $P({^L\mathcal{G}})$ be the set of Langlands parameters of $\mathcal{G}$. We say two Langlands parameters are equivalent if they are conjugate under $G$, and we denote the set of equivalence classes of Langlands parameters of $\mathcal{G}$ by $\Phi(\mathcal{G})$.

An infinitesimal parameter of $\mathcal{G}$ is a continuous homomorphism\begin{equation*}
	\lambda:W_F\to {^L\mathcal{G}}
\end{equation*} 
such that it satisfies the condition (1) and (3) of the definition of Langlands parameters.
 
If $\phi\in P({^L\mathcal{G}})$, we can associate it with an infinitesimal parameter by 
\begin{equation*}
	\lambda_\phi(\omega)=\phi(\omega,\begin{pmatrix}
		|\omega|^{1/2} & \\ &|\omega|^{-1/2}
	\end{pmatrix}), \quad \omega\in W_F
\end{equation*}
If we fix an infinitesimal parameter $\lambda$, let
\begin{equation*}
	P(\lambda,{^L\mathcal{G}}):=\{\phi\in P({^L\mathcal{G}})|\lambda_\phi=\lambda\}
\end{equation*}

and \begin{equation*}
	\Xi(\lambda,{^L\mathcal{G}}):=\{(\phi,\rho)|\phi\in P(\lambda,{^L\mathcal{G}})/Z_{{G}}(\lambda),\rho\in \mathrm{Irr}(\pi_0(Z_{{G}}(\phi)))\}
\end{equation*}

By varying over all conjugacy classes of infinitesimal parameter of $\mathcal{G}$, we get 
\begin{equation*}
	\Xi({^L\mathcal{G}}):=\{(\phi,\rho)|\phi\in \Phi(\mathcal{G}), \rho\in \mathrm{Irr}(\pi_0(Z_{{G}}(\phi)))\}
\end{equation*}

Hence, from the local Langlands correspondence, there is a natural bijection between $\mathrm{Rep}(\mathcal{G})^{[\mathcal{T},1_{\triv}]}$ and  the subset of $\Xi(^LG)$ whose Langlands parameter is trivial on the inertial group $I_F$.

We can choose $\phi\in \Phi(\mathcal{G})$ such that $t'=\phi(\mathrm{Fr})=\times_{n_1}(p^{z_1})\times\cdots\times\times_{n_k}(p^{z_k})$ satisfies 
\begin{condition}\label{IndCond}
	$|p^{z_1}|\geq |p^{z_2}| \geq \cdots \geq |p^{z_k}|$
\end{condition}

Let $y=d_\phi\begin{pmatrix}
	0 &1\\
	0& 0
\end{pmatrix}$, $t=\lambda_\phi(\mathrm{Fr})$, $\bar{t}=-\mathrm{log}t'_h+d_\nu\begin{pmatrix}
	r_0 &0\\
	0& -r_0
\end{pmatrix}=-\mathrm{log}|t'|+d_\nu\begin{pmatrix}
	r_0 &0\\
	0& -r_0
\end{pmatrix}$ and  $r_0=\mathrm{log}p^{1/2}$.

Denote 
\begin{equation}
\begin{aligned}
	E(\phi,1)&:=(\Lambda_{\mathcal{J}})^{-1}\circ\Theta\circ\IM^*E_{y,\bar{t},r_0,1}\\
	M(\phi,1)&:=(\Lambda_{\mathcal{J}})^{-1}\circ\Theta\circ \IM^*M_{y,\bar{t},r_0,1}
\end{aligned}
\end{equation}

Let $\varrho_0=|\cdot|_F^{z_i}$, $\varrho_1=\varrho_0|\cdot|_F^{\frac{1-n_i}{2}}$,  $\cdots$, $\varrho_{n_i}=\varrho_0|\cdot|_F^{\frac{n_i-1}{2}}$, and $\Delta_i=[\varrho_1,\varrho_{n_i}]$ be a segment. Then $(\Delta_1,\cdots,\Delta_k)$ is an ordered sequence of segments.  We have 
\begin{equation}
	\begin{aligned}
		E(\phi,1)&\simeq \St \langle \Delta_1\rangle \times \cdots \times \St\langle \Delta_k\rangle\\
		M(\phi,1)&\simeq \mathrm{cos}(\St\langle \Delta_1\rangle \times \cdots \times \St\langle \Delta_k\rangle)
	\end{aligned}
\end{equation}

\subsection{Disconnected case}
Recall $\mathcal{G}^+=\mathcal{G}\rtimes \Gamma$, where $\Gamma = \langle\gamma\rangle$, and $G^+:=G\rtimes \hat{\Gamma}$, where $\hat{\Gamma} = \langle\hat{\gamma}\rangle$.
Now we define \begin{equation*}
	\Xi(\lambda,{^L\mathcal{G}^+}):=\{(\phi,\rho)|\phi\in P(\lambda,{^L\mathcal{G}})/Z_{{G^+}}(\lambda),\rho\in \mathrm{Irr}(\pi_0(Z_{{G^+}}(\phi)))\}
\end{equation*}

By varying over all conjugacy classes of infinitesimal parameters of $\mathcal{G}$, we  get $\Xi({^L\mathcal{G}^+})$.

Note that $t=\lambda_\phi(\mathrm{Fr})$, $t'=\phi(\mathrm{Fr})$, $\bar{t}=-\mathrm{log}|t'|+d_\nu\begin{pmatrix}
	r_0 &0\\
	0& -r_0
\end{pmatrix}$, $r_0=\mathrm{log}p^{1/2}$, and $Z_{{G^+}}(\phi)=Z_G(t,y)= Z_{Z_G(t_c)}(\bar{t},y)$. Using what have been proved in  section 2 of this paper, we have

\begin{equation}\label{EqCat1}
	\Lambda_{\mathcal{J}}^+:\mathrm{Rep}(\mathcal{G}^+)^{[\mathcal{T},1_{\triv}]}\simeq\mathrm{Rep}(\mathcal{G}^+)^{(\mathcal{J},1_{\triv})}\simeq  \mathcal{H}(\mathcal{J}\backslash\mathcal{G}^+/\mathcal{J})-\mathrm{Mod}
\end{equation}
\begin{equation}
	\mathcal{H}(\mathcal{J}\backslash\mathcal{G}^+/\mathcal{J})\simeq \mathcal{H}(G^+,p^{1/2})\end{equation}
If $\hat{\Gamma}\cdot t\not\subseteq G\cdot t$, then
\begin{equation}
	\mathrm{Ind}_{\mathcal{H}(G,v)}^{\mathcal{H}(G^+,v)}:	\mathcal{H}(G,v)-\mathrm{Mod}_{(t,p^{1/2})}\simeq \mathcal{H}(G^+,v)-\mathrm{Mod}_{(t,p^{1/2})}
\end{equation}
If $\hat{\Gamma}\cdot t\subseteq G\cdot t$, then
\begin{equation}\label{EqCatDTC1}
	\Theta^+:	 \mathbb{H}(Z_{G^+}(t_c))-\mathrm{Mod}_{(\mathrm{log}t_h, \mathrm{log}p^{1/2})}\simeq \mathcal{H}(G^+,v)-\mathrm{Mod}_{(t,p^{1/2})}
\end{equation} 
 Using the classification of all simple modules of $\mathbb{H}(Z_{G^+}(t_c))$ obtained in Section 3, we obtain all irreducible representations in $\mathrm{Rep}(\mathcal{G}^+)^{[\mathcal{T},1_{\triv}]}$ which are parameterized by $(\phi,\rho)\in \Xi({^L\mathcal{G}^+})$ satisfying $\phi(I_F)=1$.

We have the Iwahori-Matsumoto involution $\IM$ on $\mathbb{H}(Z_{G^+}(t_c))$ given by \begin{equation}
	\begin{aligned}
		\IM(N_w)&=\mathrm{sign}(w)N_w\quad &(w\in W) \\
		\IM(N_\tau)&=N_\tau\quad &(\tau\in \mathfrak{R})\\
		\IM(\xi)&=-\xi \quad &(\xi\in \mathfrak{t}^*)
	\end{aligned}
\end{equation}

Denote \begin{align}
	&	M^+(\phi,1):= \mathrm{ind}_{\mathcal{G}}^{\mathcal{G}^+}M(\phi,1) \quad &\hat{\Gamma}\cdot t\not\subseteq G\cdot t \\
	& M^+(\phi,\rho):=(\Lambda_{\mathcal{J}}^+)^{-1}\circ(\Theta^+)\circ\IM^*M_{y,\bar{t},r_0,\rho}^+ \quad &\hat{\Gamma}\cdot t\subseteq G\cdot t \label{TwiStanM2}
\end{align}

where  $y=d_\phi\begin{pmatrix}
	0 &1\\
	0& 0
\end{pmatrix}$, $t=\lambda_\phi(\mathrm{Fr})$, $t'=\phi(\mathrm{Fr})$, $\bar{t}=-\mathrm{log}|t'|+d_\nu\begin{pmatrix}
	r_0 &0\\
	0& -r_0
\end{pmatrix}$, and $\rho\in \mathrm{Irr}(\pi_0(Z_{Z_G(t_c)}(\bar{t},y)))=\mathrm{Irr}(\pi_0(Z_{{G^+}}(\phi)))$.

Hence we get a Langlands classification of $\mathrm{Irr}(\mathrm{Rep}(\mathcal{G}^+)^{[T,1_{\triv}]})$ using enhanced Langlands parameters that are  trivial on $I_F$.

If $\mathrm{Res}^{\mathcal{G}^+}_{\mathcal{G}}M^+(\phi,\rho)$ is reducible, we have $M^+(\phi,\rho)\simeq \mathrm{ind}_{\mathcal{G}}^{\mathcal{G}^+} M(\phi,1)$.  We define the twisted standard representation
\begin{equation}
	E^+(\phi,1):=\mathrm{ind}_{\mathcal{G}}^{\mathcal{G}^+} E(\phi,1)
\end{equation}

Hence if $\hat{\Gamma}\cdot t\not\subseteq G\cdot t$, we have
\begin{equation}
\begin{aligned}
	E^+(\phi,1)&=\mathrm{ind}_{\mathcal{G}}^{\mathcal{G}^+} E(\phi,1)\\
	&\simeq \mathrm{ind}_{\mathcal{G}}^{\mathcal{G}^+} \circ(\Lambda_{\mathcal{J}})^{-1}\circ\Theta\circ\IM^*E_{y,\bar{t},r_0,1}\\
	&\simeq 	(\Lambda_{\mathcal{J}}^+)^{-1}\circ \mathrm{Ind}_{\mathcal{H}(G,v)}^{\mathcal{H}(G^+,v)}(\Theta\circ \IM^*E_{y,\bar{t},r_0,1})
\end{aligned}
\end{equation}

If $\hat{\Gamma}\cdot t\subseteq G\cdot t$, we have
\begin{equation}\label{StaModI1}
	\begin{aligned}
		E^+(\phi,1)&=\mathrm{ind}_{\mathcal{G}}^{\mathcal{G}^+} E(\phi,1)\\
		&\simeq \mathrm{ind}_{\mathcal{G}}^{\mathcal{G}^+}\circ(\Lambda_{\mathcal{J}})^{-1}\circ\Theta\circ\IM^*E_{y,\bar{t},r_0,1}\\
		&\simeq(\Lambda_{\mathcal{J}}^+)^{-1}\circ\Theta^+\circ\mathrm{Ind}_{\mathbb{H}(Z_{G}(t_c))}^{\mathbb{H}(Z_{G^+}(t_c))}\IM^*E_{y,\bar{t},r_0,1}\\
		&\simeq(\Lambda_{\mathcal{J}}^+)^{-1}\circ\Theta^+\circ \IM^*\mathrm{Ind}_{\mathbb{H}(Z_{G}(t_c))}^{\mathbb{H}(Z_{G^+}(t_c))}E_{y,\bar{t},r_0,1}\\
		&\simeq(\Lambda_{\mathcal{J}}^+)^{-1}\circ\Theta^+\circ\IM^*E_{y,\bar{t},r_0,1}^+
	\end{aligned} 
\end{equation}

If $\mathrm{Res}^{\mathcal{G}^+}_{\mathcal{G}}M^+(\phi,\rho)$ is irreducible, then $\mathrm{Res}^{\mathcal{G}^+}_{\mathcal{G}}M^+(\phi,\rho)=M(\phi,1)$ and $M(\phi,1)\simeq \gamma^*M(\phi,1)$, where $\gamma^*M(\phi,1)$ is given by twisting the representation $M(\phi,1)$ with the action $\gamma$.

We choose  $\phi\in \Phi(\mathcal{G})$ such that  $t'=\phi(\mathrm{Fr})$ satisfies Condition \ref{IndCond}. Then we have an ordered sequence of segments $(\Delta_1,\cdots,\Delta_k)$ such that  $M(\phi,1)\simeq \mathrm{cos}(\St\langle \Delta_1\rangle \times \cdots \times \St\langle \Delta_k\rangle)$
where  $\varrho_0=|\cdot|_F^{z_i}$, $\varrho_1=\varrho_0|\cdot|_F^{\frac{1-n_i}{2}}$,  $\cdots$, $\varrho_{n_i}=\varrho_0|\cdot|_F^{\frac{n_i-1}{2}}$, and a segment $\Delta_i=[\varrho_1,\varrho_{n_i}]$.  For convenience, we write  $\St\langle \Delta_i\rangle$ as $\St(z_i,n_i)$:
\begin{align}
	M(\phi,1)&\simeq \mathrm{cos}(\St(z_1,n_1)\times \cdots \times \St(z_k,n_k))
\end{align}

Recall that  $\hat{\Gamma}$ acts on $G=\GL_n(\mathbb{C})$, and  from the representation theory of the graded Hecke algebra (Proposition \ref{StandMod2}) we have ${\gamma^*( M(\phi,1))}\simeq M(\hat{\gamma}(\phi),1)$.

Then \begin{align}
	M(\hat{\gamma}(\phi),1)\simeq \mathrm{cos}(\St(-z_k,n_k)\times\cdots\times \St(-z_1,n_1))
\end{align}

Hence, if  $M(\phi,1)\simeq \gamma^*M(\phi,1)$, then for every $(z_i,n_i)$ either $p^{z_i}={(p^{z_i})}^{-1}$ or $p^{z_j}=(p^{z_i})^{-1}$ with $n_j=n_i$ for some $(z_j,n_j)$.

\begin{remark}
	From \cite[Theorem 7.3]{RepBernZel}, we know that if $\pi\in \mathrm{Rep}(\GL_n(\mathbb{Q}_p))$ is an irreducible representation, then ${\gamma^*(\pi)}$ is  the contragredient representation of $\pi$.
\end{remark}

Note that $t'=\phi(\mathrm{Fr})=\times_{n_1}(p^{z_1})\times\cdots\times\times_{n_k}(p^{z_k})$ satisfies  \begin{align}
	|p^{z_1}|\geq |p^{z_2}| \geq \cdots \geq |p^{z_k}|
\end{align} 

Then either $p^{z_i}=(p^{z_i})^{-1}$ or $p^{z_{k+1-i}}=(p^{z_{i}})^{-1}$ with $n_{k+1-i}=n_{i}$.

Pick those $(z_i,n_i)$ satisfying $p^{z_i}=\pm 1$ and write \begin{align} \mathcal{Q}^0=\GL_{n_{i_1}+\cdots+n_{i_l}}(\Qp)
\end{align}
Hence we have\begin{align}
	\mathcal{Q}=\GL_{n_1}(\mathbb{Q}_p)\times\cdots\times \mathcal{Q}^0\times\cdots\times \GL_{n_k}(\mathbb{Q}_p)
\end{align}
where  $n_{k+1-i}=n_{i}$. 

Then $\gamma$ normalizes $\mathcal{Q}$,  and we can define $\mathcal{Q}^+=\mathcal{Q}\rtimes \Gamma$. Hence we have an irreducible representation  $M^+(\phi_{\mathcal{Q}},\rho)$ of $\mathcal{Q}^+$

Take the twisted parabolic subgroup $\mathcal{P}^+$ which contains the upper triangular Borel subgroup and $\mathcal{Q}^+$.

We define the twisted standard representation \begin{equation}
	E^+(\phi,\rho):= i_{\mathcal{P}^+} M^+(\phi_{\mathcal{Q}},\rho)
\end{equation}

Denote by $\hat{\mathfrak{g}}$ the Lie algebra of $G$, and define the Vogan moduli space \begin{equation*}
	V(\lambda,{^L\mathcal{G}}):=\{x\in \hat{\mathfrak{g}}|\mathrm{Ad}(\lambda(\mathrm{Fr}))x=px\}
\end{equation*}

Let $\mathrm{Per}(\lambda,{^L\mathcal{G}}^+)$ be the category of $Z_{{G}^+}(\lambda)$-equivariant perverse sheaves on $V(\lambda,{^L\mathcal{G}})$. Then we can parametrize simple objects $\IC(\zeta^+)$ in this category by  $\zeta^+=(\phi(\zeta^+),\rho(\zeta^+))\in \Xi(\lambda,{^L\mathcal{G}^+})$.

\begin{theorem}\label{MainTh2}
	Taking $\xi^+,\zeta^+\in \Xi(\lambda,{^L\mathcal{G}^+})$, we have irreducible representation $M^+(\phi(\xi^+),\rho(\xi^+))$ and standard representation $E^+(\phi(\zeta^+),\rho(\zeta^+))$. Then the multiplicity in the Grothendieck group is given by  
	\begin{equation}
		m(M^+(\phi(\xi^+),\rho(\xi^+)),E^+(\phi(\zeta^+),\rho(\zeta^+)))=\sum_k\mathrm{dim}H^k(i_{y(\zeta^+)}^!\IC(\xi^+)))^{\rho(\zeta^+)}
	\end{equation}
	where $y=d_{\phi(\zeta^+)}\begin{pmatrix}
		0 &1\\
		0& 0
	\end{pmatrix}$.
\end{theorem} 
\begin{proof}
	Denote $t=\lambda_\phi(\mathrm{Fr})$. 
	
	If $\hat{\Gamma}\cdot t\not\subseteq G\cdot t$, we have $M^+(\phi,1):= \mathrm{ind}_{\mathcal{G}}^{\mathcal{G}^+}M(\phi,1)$ and $E^+(\phi,1):= \mathrm{ind}_{\mathcal{G}}^{\mathcal{G}^+}E(\phi,1)$. Using \eqref{IndComPG}, \eqref{EqCat1} and \eqref{EqCatDTC1}, we can reduce the disconnected case to the connected case, and get the theorem.
	
	If $\hat{\Gamma}\cdot t\subseteq G\cdot t$, we have \begin{align}
		M^+(\phi,\rho):=(\Lambda_{\mathcal{J}}^+)^{-1}\circ\Theta^+\circ\IM^*M_{y,\bar{t},r_0,\rho}^+
	\end{align}
	If $M(\phi,1)\simeq \gamma^*M(\phi,1)$, we have defined 
	\begin{align}
		E^+(\phi,\rho)= i_{\mathcal{P}^+} M^+(\phi_{\mathcal{Q}},\rho)
	\end{align}
	The irreducible representation $M^+(\phi_{\mathcal{Q}},\rho)$ of $\mathcal{Q}$ is given by  \begin{align}
		\Lambda^+_{\mathcal{J}_{\mathcal{Q}}}: \mathrm{Rep}(\mathcal{Q}^+)^{(\mathcal{J}_{\mathcal{Q}},1_{\triv})}  \xrightarrow[]{\simeq} \mathcal{H}(\mathcal{J}_{\mathcal{Q}}\backslash\mathcal{Q}^+/\mathcal{J}_{\mathcal{Q}})-\Mod
	\end{align}
	
	\begin{align}
		\mathcal{H}(\mathcal{J}_{\mathcal{Q}}\backslash\mathcal{Q}^+/\mathcal{J}_{\mathcal{Q}})\simeq \mathcal{H}(Q^+,p^{1/2})\simeq \mathcal{H}(Q,p^{1/2})\rtimes \Gamma
	\end{align}
	
	and
	\begin{equation*}
		\Theta^+_Q:	\mathcal{H}(Q^+,v)-\mathrm{Mod}_{(t,p^{1/2})}\simeq \mathbb{H}(Z_{Q^+}(t_c))-\mathrm{Mod}_{(\mathrm{log}t_h, \mathrm{log}p^{1/2})}
	\end{equation*} 
	which means \begin{align}
		M^+(\phi_{\mathcal{Q}},\rho)=(\Lambda_{\mathcal{J_Q}}^+)^{-1}\circ\Theta^+_Q\circ \IM^*M_{y,\bar{t},r_0,\rho}^{Q,+} 
	\end{align}
	
	From \eqref{commutediagram2} and Theorem \ref{Ind22}, we have
	\[
	\begin{tikzcd}
		\mathrm{Rep}(\mathcal{G}^+)^{(\mathcal{J},1_{\triv})} \ar[r,"{ \Lambda^+_{\mathcal{J}}}"] & \mathcal{H}(G^+,p^{1/2})-\Mod\\
		\mathrm{Rep}(\mathcal{Q}^+)^{(\mathcal{J}_{\mathcal{Q}},1_{\triv})} \ar[r,"{\Lambda^+_{\mathcal{J}_{\mathcal{Q}}}}"] \ar[u,"{i_{\mathcal{P}^+}}"]& \mathcal{H}(Q^+,p^{1/2})-\Mod\ar[u,"\Ind"]
	\end{tikzcd}
	\]
	and  \[
	\begin{tikzcd}
		\mathcal{H}(Q^+,v)-\mathrm{Mod}_{(t,p^{1/2})}\ar[r,"{ \Theta^+}"] & \mathbb{H}(Z_{Q^+}(t_c))-\mathrm{Mod}_{(\mathrm{log}t_h, \mathrm{log}p^{1/2})}\\
		\mathcal{H}(Q^+,v)-\mathrm{Mod}_{(t,p^{1/2})}\ar[r,"{\Theta^+_Q}"] \ar[u,"{\Ind}"]& \mathbb{H}(Z_{Q^+}(t_c))-\mathrm{Mod}_{(\mathrm{log}t_h, \mathrm{log}p^{1/2})}\ar[u,"\Ind"]
	\end{tikzcd}
	\]
	
	Hence we have \begin{align}
		E^+(\phi,\rho)&= i_{\mathcal{P}^+} M^+(\phi_{\mathcal{Q}},\rho)\notag\\
		&=i_{\mathcal{P}^+}(\Lambda_{\mathcal{J_Q}}^+)^{-1}\circ\Theta^+_Q\circ\IM^*M_{y,\bar{t},r_0,\rho}^{Q,+}\notag\\
		&\simeq (\Lambda_{\mathcal{J}}^+)^{-1}\circ\Theta^+\circ \Ind_{\mathbb{H}(Z_{Q^+}(t_c))}^{\mathbb{H}(Z_{G^+}(t_c))}\IM^*M_{y,\bar{t},r_0,\rho}^{Q,+}\notag\\
		&\simeq (\Lambda_{\mathcal{J}}^+)^{-1}\circ\Theta^+\circ \IM^*\Ind_{\mathbb{H}(Z_{Q^+}(t_c))}^{\mathbb{H}(Z_{G^+}(t_c))}M_{y,\bar{t},r_0,\rho}^{Q,+}\label{KZInd1}\\
		&\simeq (\Lambda_{\mathcal{J}}^+)^{-1}\circ\Theta^+\circ \IM^*\Ind_{\mathbb{H}(Z_{Q^+}(t_c))}^{\mathbb{H}(Z_{G^+}(t_c))}E_{y,\bar{t},r_0,\rho}^{Q,+}\label{KZInd2}\\
		&\simeq (\Lambda_{\mathcal{J}}^+)^{-1}\circ\Theta^+\circ\IM^*E_{y,\bar{t},r_0,\rho}^+\label{KZInd3}
	\end{align}
	Here \eqref{KZInd1} follows from a canonical isomorphism of $\mathbb{H}(Z_{G^+}(t_c))$-modules: \begin{equation}
		\begin{aligned}
			\IM^*(\mathbb{H}(Z_{G^+}(t_c))\otimes_{\mathbb{H}(Z_{Q^+}(t_c))}V) &\xrightarrow{\simeq} \mathbb{H}(Z_{G^+}(t_c))\otimes_{\mathbb{H}(Z_{Q^+}(t_c))}\IM^*(V)\\
			h\otimes v &\mapsto \IM(h)\otimes v
		\end{aligned}
	\end{equation}  The equality \eqref{KZInd2} follows from the fact that $y$ lies in the open orbit.
	The equality \eqref{KZInd3} follows from  the induction Theorem \ref{Ind31} for graded Hecke algebra.
	
	If $M(\phi,1)\not \simeq \gamma^*M(\phi,1)$, we also have
	\begin{align}
		E^+(\phi,\rho)\simeq(\Lambda_{\mathcal{J}}^+)^{-1}\circ\Theta^+\circ\IM^*E_{y,\bar{t},r_0,1}^+
	\end{align}
	which is given by \eqref{StaModI1}.
	Recall that $M^+(\phi,\rho)=(\Lambda_{\mathcal{J}}^+)^{-1}\circ\Theta^+\circ\IM^*M_{y,\bar{t},r_0,\rho}^+$. 
	Hence using Proposition \ref{GeoMul}, we get the theorem when $\hat{\Gamma}\cdot t\subseteq G\cdot t$.
\end{proof}

\section{Whittaker normalization}
In the last section, for any irreducible representation $M(\phi,1)$ of $\mathcal{G}=\GL_n(\mathbb{Q}_p)$ that is not $\gamma$-invariant, we can always get the irreducible representation $M^+(\phi,1)=\mathrm{ind}_{\mathcal{G}}^{\mathcal{G}^+} M(\phi,1)$ and the standard representation $E^+(\phi,1)=\mathrm{ind}_{\mathcal{G}}^{\mathcal{G}^+} E(\phi,1)$ by induction. However, if $M(\phi,1)$ is $\gamma$-invariant, there are two extensions of $M(\phi,1)$. In the previous section, we used geometry to obtain a uniform classification $M^+(\phi,1)$ and $M^+(\phi,-1)$ ( where 1 denotes the trivial representation of the component group and -1 denotes the sign representation of the component group).

But if $M(\phi,1)$ is $\gamma$-invariant, we can also use the Whittaker normalization introduced by Arthur on the $p$-adic side to get a canonical extension  $(M(\phi,1))^+_W$ of $\mathcal{G}^+=\GL_n(\Qp)\rtimes\Gamma$.
In this section, we will prove that $(M(\phi,1))^+_W\simeq M^+(\phi,1)$.
\subsection{Definition}
Recall that $\Pi$ denotes the set of simple roots determined by the upper triangular Borel subgroup. The abelian group $\prod_{\alpha\in\Pi}\mathcal{U}_\alpha$ is a quotient of $\mathcal{U}$, where $\mathcal{U}_\alpha$ is the root subgroup for $\alpha$. Given characters $\varphi_\alpha:\mathcal{U}_\alpha\to\mathbb{C}^*$, their product defines a  character on $\prod_{\alpha\in\Pi}\mathcal{U}_\alpha$ and hence a character $\varphi$ on $\mathcal{U}$. We say that $\varphi$ is principal if all the $\varphi_\alpha$ are non-trivial. We say that $\varphi$ is unramified if all the character $\varphi_\alpha$ are trivial on $\mathscr{O}$ but non-trival on $\mathscr{P}^{-1}$.

Since $\gamma$ normalizes $\mathcal{U}$, we make such a choice of $\varphi$  that $\varphi$ is  principal, unramified and $\gamma$ invariant.

Suppose first that an irreducible representation $\pi\in \mathrm{Rep}(\mathcal{G})$ is tempered. Then we have a Whittaker functional $\mathcal{W}$ on $\pi$. By definition, $\mathcal{W}$ is a nonzero linear functional on the underlying space $V$ of $\pi$ such that \begin{equation*}
	\mathcal{W}(\pi(u)v)=\varphi(u)\mathcal{W}(v)
\end{equation*}

It is unique up to a scalar multiple.

If the representation $\pi$  is equivalent to $\gamma^*(\pi)$, we can find a nontrivial intertwining operator $I$ from $\pi$ to $\gamma^*(\pi)$, which is unique up to a nonzero scalar multiple. Hence $\mathcal{W}\circ I$ is nonzero linear functional on $V$, and it satisfies
\begin{equation}
\begin{aligned}
	\mathcal{W}\circ I(\pi(u)v)&=\mathcal{W}(\pi(\gamma(u))I(v))\\
	&=\varphi(\gamma(u))\mathcal{W}(I(v))\\
	&=\varphi(u)\mathcal{W}\circ I(v)
\end{aligned}
\end{equation}
Therefore  $\mathcal{W}\circ I$ is also a Whittaker functional on $\pi$. It therefore equals $c\mathcal{W}$, for some $c\in \mathbb{C}^*$. We set \begin{equation*}
	I_W:=c^{-1}I
\end{equation*}
Then $I_W$ is the unique intertwining operator from $\pi$ to $\gamma^*(\pi)$ such that \begin{equation*}
	\mathcal{W}=\mathcal{W}\circ I_W
\end{equation*}

Hence we can define a representation $(\pi^+_W,V)$ of $\mathcal{G}^+$  by \begin{equation}
\begin{aligned}
	\pi^+_W(g,1)\cdot v&=\pi(g)\cdot v\\
	\pi^+_W(g,\gamma)\cdot v&= \pi(g)\cdot I_W(v) 
\end{aligned}
\end{equation}

Suppose $\pi$ is general irreducible and $\pi$ is equivalent to $\gamma^*(\pi)$. From the Langlands classification,  we have $\pi\simeq M(\phi,1)$ for some Langlands parameter $\phi$, which is the unique quotient of the standard module $E(\phi,1)$. Using the argument in the last section, we have
\begin{align}
	E(\phi,1)=\St(z_1,n_1)\times \cdots \times \St(z_k,n_k)
\end{align}
where either $p^{z_i}=(p^{z_i})^{-1}$ or $p^{z_{k+1-i}}=(p^{z_{i}})^{-1}$ with $n_{k+1-i}=n_{i}$.

Pick those  $(z_i,n_i)$ satisfying $p^{z_i}=\pm 1$, then we have a tempered representation \begin{align}
	\pi^0:=\St(z_{i_1},n_{i_1})\times\cdots\times \St(z_{i_l},n_{i_l})
\end{align} of 
$\mathcal{Q}^0=\GL_{n_{i_1}+\cdots+n_{i_l}}(\Qp)$.

Recall 
$\mathcal{Q}=\GL_{n_1}(\mathbb{Q}_p)\times\cdots\times \mathcal{Q}^0\times\cdots\times \GL_{n_k}(\mathbb{Q}_p)$.

We have the representation
\begin{align}\label{InvRep}
	\St (z_1,n_1) \otimes \cdots \otimes\pi^0\otimes\cdots\otimes \St( z_k,n_k)
\end{align}
of $\mathcal{Q}$, which satisfies
\begin{equation}
\begin{aligned}
	\gamma^*(\St(z_i,n_i))&\simeq \St(z_{k+1-i},n_{k+1-i})\\
	\gamma^*(\pi^0)&\simeq {\pi}^0
\end{aligned}
\end{equation}

Hence\begin{equation}
\begin{aligned}
	E(\phi,1)&=\St(z_1,n_1)\times \cdots \times \St(z_k,n_k)\\
	&=i_{\mathcal{P}}(\St (z_1,n_1) \otimes \cdots \otimes\pi^0\otimes\cdots\otimes \St(z_k,n_k))
\end{aligned}
\end{equation}

From the tempered case, we have the Whittaker normalized intertwining operator $I_{W,\pi^0}$ from $\pi^0$ to $\gamma^*(\pi^0)$. Since
$\gamma^*(\St(z_i,n_i))\simeq \St(z_{k+1-i})$,
we have a morphism $I_i$ from the underlying space of $\St(z_i,n_i)$ to the space $\St(z_{k+1-i},n_{k+1-i})$ such that $I_i$ is an isomorphism from the representation $\gamma^*(\St(z_i,n_i))$ to $\St(z_{k+1-i},n_{k+1-i})$.

Let $\phi \in \St (z_1,n_1) \times \cdots \times\pi^0\times\cdots\times \St( z_k,n_k)$,  which is the induction of $\St (z_1,n_1) \otimes \cdots \otimes\pi^0\otimes\cdots\otimes \St( z_k,n_k)$. For every $g\in \GL_n(\mathbb{Q}_p)$, we define\begin{align}\label{ActIndP1}
	I_W(\phi)(g):=(I_1^{-1}\otimes \cdots\otimes I_{W,\pi_0}\otimes\cdots\otimes I_1)\iota(\phi(\gamma(g)))
\end{align}
where $\iota$ is the isomorphism from $\St (z_1,n_1) \otimes \cdots \otimes\pi^0\otimes\cdots\otimes \St( z_k,n_k)$ to $\St (z_k,n_k) \otimes \cdots \otimes\pi^0\otimes\cdots\otimes \St(z_1,n_1)$ that swaps the terms.

Hence $I_W$ intertwines $E(\phi,1)$ and $\gamma^*E(\phi,1)$, and using  $I_W$ on  $E(\phi,1)=\St (z_1,n_1) \times \cdots \times\pi^0\times\cdots\times \St( z_k,n_k)$, we get the representation $(E(\phi,1))^+_W$ of $\mathcal{G}^+$. 
\begin{remark}
	From \cite{Transfer}, this $I_W$ induces a trivial action on the Whittaker functional of $E(\phi,1)$. Hence this action is the Whittaker-normalized from the definition in the introduction.
\end{remark}

Since $\pi\simeq M(\phi,1)$ is the unique irreducible quotient of  $E(\phi,1)$, we can get an action $I_{W,\pi}$ on $\pi\simeq M(\phi,1)$ that intertwines $\pi$ and $\gamma^*(\pi)$. Hence, using this $I_{W,\pi}$, we get the representation $(M(\phi,1))^+_W$ of $\mathcal{G}^+$.

Note that $\gamma$ normalizes $\mathcal{Q}$ and the representation $ \St (z_1,n_1) \otimes \cdots \otimes\pi^0\otimes\cdots\otimes \St( z_k,n_k)$ is also $\gamma$ invariant. We define \begin{align}\label{LeviAct}
	I_{\mathcal{Q}}:=(I_1^{-1}\otimes \cdots\otimes I_{W,\pi^0}\otimes\cdots\otimes I_1)\circ \iota
\end{align}
Hence  $I_{\mathcal{Q}}$  intertwines  $\St (z_1,n_1) \otimes \cdots \otimes\pi^0\otimes\cdots\otimes \St( z_k,n_k)$ and $\gamma^*( \St (z_1,n_1) \otimes \cdots \otimes\pi^0\otimes\cdots\otimes \St( z_k,n_k))$.

Therefore we can get a representation $(\St (z_1,n_1) \otimes \cdots \otimes\pi^0\otimes\cdots\otimes \St( z_k,n_k))^+$ of $\mathcal{Q}^+=\mathcal{Q}\rtimes\Gamma$.

\begin{proposition}\label{WhiInd}
	Denote $\pi_{\mathcal{Q}}:=\St (z_1,n_1) \otimes \cdots \otimes\pi^0\otimes\cdots\otimes \St( z_k,n_k)$ and  $\pi_{\mathcal{Q}}^+:=(\St (z_1,n_1) \otimes \cdots \otimes\pi^0\otimes\cdots\otimes \St( z_k,n_k))^+$. Then we have $E(\phi,1)=i_{\mathcal{P}}\pi_{\mathcal{Q}}$.
	
	There is an isomorphism of representations of $\mathcal{G}^+$:\begin{align}
		(E(\phi,1))^+_W\simeq i_{\mathcal{P}^+}\pi_{\mathcal{Q}}^+
	\end{align}
	which means that the representation $i_{\mathcal{P}^+}\pi_{\mathcal{Q}}^+$ of $\mathcal{G}^+$ has the same underlying vector space as $i_{\mathcal{P}}\pi_{\mathcal{Q}}$, and the $\gamma$ action on $i_{\mathcal{P}}\pi_{\mathcal{Q}}$ is given by \eqref{ActIndP1}.
\end{proposition}
\begin{proof}
	Taking $\phi\in i_{\mathcal{P}^+}\pi_{\mathcal{Q}}^+$, we restrict $\phi$ as a function on $\mathcal{G}$, which means we have \begin{equation}
		\begin{aligned}
			\mathrm{Res}: i_{\mathcal{P}^+}\pi_{\mathcal{Q}}^+&\to i_{\mathcal{P}}\pi_{\mathcal{Q}}\\
			\phi&\mapsto\bar{\phi}
		\end{aligned}
	\end{equation} 
	Since \begin{align}
		\phi((g,\gamma))&=\phi((1,\gamma)(\gamma(g),1))\nonumber\\&=\pi_{\mathcal{Q}}^+((1,\gamma))\phi(\gamma(g),1)\nonumber\\&=I_{\mathcal{Q}}\circ \bar{\phi}(\gamma(g))\label{ActIndPhi}
	\end{align}
	the map $\mathrm{Res}$ is an isomorphism. Since $(1,\gamma)\cdot\phi(g,1)=\phi((g,1)\cdot(1,\gamma))=\phi((g,\gamma))$, the $(1,\gamma)$ action on $\bar{\phi}$ is given by \eqref{ActIndPhi}, which coincides with \eqref{ActIndP1}.
\end{proof}

Since $M^+(\phi,1)$ is the unique   irreducible quotient of $E^+(\phi,1)$, if we want to prove $(M(\phi,1))^+_W\simeq M^+(\phi,1)$,  it suffices to prove $ (E(\phi,1))^+_W\simeq E^+(\phi,1)$.

Since $E^+(\phi,1)= i_{\mathcal{P}^+} E^+(\phi_{\mathcal{Q}},1)=i_{\mathcal{P}^+} M^+(\phi_{\mathcal{Q}},1)$ and $(E(\phi,1))^+_W\simeq i_{\mathcal{P}^+}\pi_{\mathcal{Q}}^+$, we just need to prove $\pi_{\mathcal{Q}}^+\simeq M^+(\phi_{\mathcal{Q}},1)$.

\subsection{Tempered case}
Let $\mathcal{G}=\GL_n(\mathbb{Q}_p)$. Denote
\begin{equation}
\begin{aligned}
	\pi_n&:=\St(z,n)\quad p^z=1\\
	\pi'_n&=\St(z',n)\quad p^{z'}=-1
\end{aligned}
\end{equation}
Hence $z$ lies in  $(2\pi i /\mathrm{log}p)\mathbb{Z} $ and $z'$  lies in $\pi i/\mathrm{log}p+(2\pi i /\mathrm{log}p)\mathbb{Z}$.

Take an irreducible representation $\pi$ in $\mathrm{Rep}(\mathcal{G})^{[\mathcal{T},1_{\triv}]}$. If $\pi$ is a tempered representation, $\pi$ has the form \begin{equation*}
	\pi=\pi_{k_1}\times\cdots\times\pi_{k_i}\times\pi'_{l_1}\times\cdots\times\pi'_{l_j}
\end{equation*}

Hence we can view $\pi$ as a unique subrepresentation of the unramified principal series: 
\begin{equation*}
	\pi\hookrightarrow\sigma:= (\times_\lambda\rho|\cdot|_F^{\lambda})\times(\times_{\lambda'}\rho'|\cdot|_F^{\lambda'})
\end{equation*}
where $\rho=|\cdot|_F^z$ with $p^z=1$,  $\rho'=|\cdot|_F^{z'}$ with $p^{z'}=-1$,  and $\lambda$,$\lambda'$ are the half integers, taken in decreasing order.

Denote $k=k_1+\cdots+k_i$, $l=l_1+\cdots+l_{j}$, $\mathcal{M}=\GL_k(\mathbb{Q}_p)\times \GL_l(\mathbb{Q}_p)$.
Hence $\sigma=i_{\mathcal{T}}^{\mathcal{G}}\mathbb{C}_{(\lambda,\lambda')}=i_{\mathcal{M}}^{\mathcal{G}}i_{\mathcal{T}}^{\mathcal{M}}\mathbb{C}_{(\lambda,\lambda')}=i_{\mathcal{M}}^{\mathcal{G}}\sigma_{\mathcal{M}}$, where $\sigma_{\mathcal{M}}=i_{\mathcal{T}}^{\mathcal{M}}\mathbb{C}_{(\lambda,\lambda')}$

Denote $\pi_1=\pi_{k_1}\times\cdots\times\pi_{k_i}$, $\pi_{-1}=\pi'_{l_1}\times\cdots\times\pi'_{l_j}$, then $\pi=i_{\mathcal{M}}^{\mathcal{G}}(\pi_1\otimes\pi_{-1})$.
\subsubsection{Intertwining operator and Whittaker functional}  
We mainly use the ideas in \cite{haines2003iwahori}, but translate their right action into our left action situation.

Denote $V_{\mathcal{J}}=c\operatorname{-ind}_{\mathcal{J}}^{\mathcal{G}}(1)=C_c^\infty(\mathcal{J}\backslash \mathcal{G})$, $\mathcal{H}=C_c^\infty(\mathcal{J}\backslash\mathcal{G}/\mathcal{J})$.
Hence $V_{\mathcal{J}}$ is projective generator, and we have $\mathcal{H}'\stackrel{\simeq}{\longrightarrow}\mathrm{End}_G(V_{\mathcal{J}})$ $(\phi\mapsto t_\phi)$ given by left convolution:
\begin{equation*}
	t_\phi(f)(g)=\int_{\mathcal{G}}\phi(x)f(x^{-1}g)dx
\end{equation*}

We also can view $V_{\mathcal{J}}^{\mathcal{J}}$ as an $\mathcal{H}$-module via

\begin{equation*}
	\psi\circ f=\int_{\mathcal{G}}\psi(g)\pi(g)fdg
\end{equation*}

which means 

\begin{equation}\label{HAct1}
	(\psi\circ f)(x)=\int_{\mathcal{G}}\psi(g)f(xg)dg
\end{equation}
where $\psi \in \mathcal{H}$, $f\in V_{\mathcal{J}}^{\mathcal{J}}$.

Hence we have a left $\mathcal{H}$-module isomorphism
\begin{equation}\label{HeIso}
    \begin{aligned}
	\mathcal{H} &\stackrel{\simeq}{\longrightarrow} V_{\mathcal{J}}^{\mathcal{J}}\\
	\psi&\mapsto \psi\circ 1_{\mathcal{J}1\mathcal{J}} 
\end{aligned}    
\end{equation}

Using \eqref{HAct1}, we get
\begin{equation*}
	(\psi\circ 1_{\mathcal{J}1\mathcal{J}})(x)= \psi(x^{-1})
\end{equation*}

Hence for every $\psi'\in \mathcal{H}$ we can define an action on $V_{\mathcal{J}}^{\mathcal{J}}$
\begin{equation}\label{rightmul}
	\psi': \psi\circ 1_{\mathcal{J}1\mathcal{J}}\mapsto (\psi\star\psi')\circ 1_{\mathcal{J}1\mathcal{J}}
\end{equation}
which is a left $\mathcal{H}$-module isomorphism because of \eqref{HeIso}.

We have a commutative diagram
\begin{equation}\label{diagram}
	\begin{tikzcd}
		\mathrm{End}_{\mathcal{G}}(V_{\mathcal{J}})^{\mathrm{opp}}\times\mathrm{Hom}_{\mathcal{G}}(V_{\mathcal{J}},W)\arrow[r]\arrow[d]&\mathrm{Hom}_{\mathcal{G}}(V_{\mathcal{J}},W)\arrow[d]\\
		\mathcal{H}\times\mathrm{Hom}_{\mathcal{H}}(V_{\mathcal{J}}^{\mathcal{J}},W^{\mathcal{J}})\arrow[r]\arrow[d]&\mathrm{Hom}_{\mathcal{H}}(V_{\mathcal{J}}^{\mathcal{J}},W^{\mathcal{J}})\arrow[d]\\
		\mathcal{H}\times W^{\mathcal{J}}\arrow[r] & W^{{\mathcal{J}}}
	\end{tikzcd}
\end{equation}
where the second and the third rows are induced by the action on first row.

We have an anti-isomorphism from $\mathcal{H}$ to $\mathcal{H}'$ given by $\psi \mapsto \psi'$ with $\psi(x) = \psi'(x^{-1})$. Hence we obtain an isomorphism $\mathcal{H}\stackrel{\simeq}{\longrightarrow}\mathrm{End}_{\mathcal{G}}(V_{\mathcal{J}})^{\mathrm{opp}}$ $(\psi\mapsto t_\psi)$ via 
\begin{equation*}
	t_\psi(f)(g)=\int_{\mathcal{G}}\psi(x^{-1})f(x^{-1}g)dx
\end{equation*}
Then we can check that the multiplication in the middle of the diagram \eqref{diagram} is given by \eqref{rightmul}.

Denote the group ring $R=\mathbb{C}[\mathcal{T}/\mathcal{T}^0]$, which is isomorphic to $C_c^\infty(\mathcal{T}/\mathcal{T}^0)$. Thus the element $\varpi^\mu$ $(\mu\in X_*(\mathbb{T}))$ form a $\mathbb{C}$-basis for the vector space $R$. We view $R$ as a $\mathcal{T}$-module via the tautological character $\chi_{\mathrm{univ}} : \mathcal{T}/\mathcal{T}^0 \to R$ sending $\varpi^\mu$ to $\varpi^\mu$.

Let $\mathcal{B}=\mathcal{T}\mathcal{U}$ be the upper triangular Borel subgroup. We define $i_{\mathcal{B}}^{\mathcal{G}}(\chi_{\mathrm{univ}})$ as the space of locally constant $R$-valued functions $f$ on $\mathcal{G}$ satisfying \begin{equation*}
	f(tug)=\delta(t)^{1/2}t\cdot f(g)
\end{equation*}
where $\delta$ is the modulus character of $\mathcal{B}$, $t\in \mathcal{T}$, $u\in \mathcal{U}$, $g\in \mathcal{G}$, and the group $\mathcal{G}$ acts by right translations.

Define the $\mathcal{G}$-projection $\mathfrak{P}: C_c^\infty(\mathcal{G}) \to i_{\mathcal{B}}^{\mathcal{G}}(\chi_{\mathrm{univ}})$ by
\begin{equation*}
	\mathfrak{P}(f)(g)=\int_{\mathcal{B}}\chi_{\mathrm{univ}}^{-1}\delta^{1/2}(x)f(xg)dx
\end{equation*}
where $\chi_{\mathrm{univ}}$ and $\delta$ are trivial on $\mathcal{U}$, and $\mathcal{B}$ has the left Haar measure such that $\mathcal{B}\cap\mathcal{J}$ has  volume 1.

Recall that $\mathcal{G}=\sqcup \mathcal{B}W\mathcal{J}$, and $T_{(\mu,w)}$ is the characteristic function of the double coset $\mathcal{J}(\varpi^{-\mu},w)\mathcal{J}$, where $(\varpi^{-\mu},w)\in \mathcal{T}/\mathcal{T}^0\rtimes W$. Set $f_{(\mu,w)}=\mathfrak{P}(T_{(\mu,w)})$. In particular, $f_{(1,w)}$ is identically 0 off $\mathcal{B}w\mathcal{J}$ and for $x_1\in \mathcal{B}$, $x_2\in \mathcal{J}$ we have\begin{equation*}
	f_{(1,w)}(x_1wx_2)=\chi_{\mathrm{univ}}\delta^{1/2}(x_1)
\end{equation*}

From  \cite{ChirssPadic}, $\mathfrak{P}$ induces a $\mathcal{G}$-isomorphism from $V_{\mathcal{J}}$ to $i_{\mathcal{B}}^{\mathcal{G}}(\chi_{\mathrm{univ}})$. Consequently $\mathfrak{P}^{\mathcal{J}}$ induces an $\mathcal{H}$-isomorphism from $V^{\mathcal{J}}_{\mathcal{J}}$ to $i_{\mathcal{B}}^{\mathcal{G}}(\chi_{\mathrm{univ}})^{\mathcal{J}}$.

Using \eqref{HeIso}, we have a left $\mathcal{H}$-module isomorphism
\begin{align}
	\Psi:\mathcal{H} &\stackrel{\simeq}{\longrightarrow} i_{\mathcal{B}}^{\mathcal{G}}(\chi_{\mathrm{univ}})^{\mathcal{J}}\label{HIso2}\\
	\psi&\mapsto \psi\circ f_1 \nonumber
\end{align}
where $f_1=f_{(1,1)}$.

Recall $Y=\Hom(\mathbb{G}_m,\mathbb{T})$ and $Y\simeq \mathcal{T}/\mathcal{T}^0$ given by $y\mapsto y(\varpi^{-1})$. From Section 1 of our paper, we have an embedding $\mathbb{C}[Y]\to \mathcal{H}$. Then $\mathcal{H}$ becomes a right $\mathbb{C}[Y]$-module via right multiplication. In section 1, $\mathbb{C}[Y]$ has basis $\theta_y$ $(y\in Y)$. We define a $\mathbb{C}$-algebra isomorphism\begin{align}
	\mathbb{C}[Y]& \stackrel{\simeq}{\longrightarrow}  R=\mathbb{C}[\mathcal{T}/\mathcal{T}^0]\label{RIso}\\
	\theta_y &\mapsto y(\varpi^{-1})\nonumber
\end{align}
Thus $\mathcal{H}$ can also be viewed as a right $R=\mathbb{C}[\mathcal{T}/\mathcal{T}^0]$ module via \eqref{RIso}.

Note that $i_{\mathcal{B}}^{\mathcal{G}}(\chi_{\mathrm{univ}})$  naturally carries a right $R=\mathbb{C}[\mathcal{T}/\mathcal{T}^0]$-module structure. By  \cite[Lemma 5.10]{ChirssPadic}, the map $\Psi$ in \eqref{HIso2} is also a right $R$-module homomorphism. Combining  these facts, we have 
\begin{proposition}\label{HIsoPro1}
	The map $\Psi: \mathcal{H} \to i_{\mathcal{B}}^{\mathcal{G}}(\chi_{\mathrm{univ}})^{\mathcal{J}}$ defined by $\psi\mapsto \psi\circ f_1$ induces an isomorphism of $(\mathcal{H},R)$-bimodules.
\end{proposition}

Given an unramified character $\chi:\mathcal{T}/\mathcal{T}^0\to \mathbb{C}^*$, we consider the normalized parabolic induction $i_{\mathcal{B}}^{\mathcal{G}}(\chi)$. The following is obvious. \begin{proposition}\label{HIsoPro2}
	The map \begin{equation}
	    \begin{aligned}
		\Phi: i_{\mathcal{B}}^{\mathcal{G}}(\chi_{\mathrm{univ}})^{\mathcal{J}}\otimes_R\mathbb{C}_\chi \stackrel{\simeq}{\longrightarrow} i_{\mathcal{B}}^{\mathcal{G}}(\chi)^{\mathcal{J}}
        ,\quad \Phi(f\otimes z)(g)=z\chi(f(g))
	\end{aligned}
	\end{equation}
	is an isomorphism of left $\mathcal{H}$-modules.
\end{proposition}

Combine Proposition \ref{HIsoPro1} and Proposition \ref{HIsoPro2}, we have a left $\mathcal{H}$-module isomorphism
\begin{equation*}
	\begin{tikzcd}
		\mathcal{H}\otimes_R\mathbb{C}_\chi\arrow[d,"\Psi\otimes 1"]\arrow[dr, dashrightarrow]& \\
		i_{\mathcal{B}}^{\mathcal{G}}(\chi_{\mathrm{univ}})^{\mathcal{J}}\otimes_R\mathbb{C}_\chi\arrow[r,"\Phi"]&i_{\mathcal{B}}^{\mathcal{G}}(\chi)^{\mathcal{J}}
	\end{tikzcd}
\end{equation*}
\begin{remark}
	Using \eqref{commutediagram2}, we can also get the same isomorphism $\mathcal{H}\otimes_R\mathbb{C}_\chi \stackrel{\simeq}{\longrightarrow}i_{\mathcal{B}}^{\mathcal{G}}(\chi)^{\mathcal{J}}$.
\end{remark}

Recall the multiplication in $\mathcal{H}$:\begin{equation*}
	\theta_xT_{s_\alpha}-T_{s_\alpha}\theta_{s_\alpha(x)}=(p-1)\frac{\theta_x-\theta_{s_\alpha(x)}}{1-\theta_{-\alpha}}
\end{equation*}
where $\alpha\in \Pi$ is a simple root and $s_\alpha$ is the corresponding simple reflection.

We define an intertwining element \begin{equation}
	 \begin{aligned}
		\iota_{s_\alpha}&=T_{s_\alpha}(1-\theta_\alpha)+(p-1)\theta_\alpha\\
		&= (1-\theta_{-\alpha})T_{s_\alpha}+(1-p)\nonumber
	\end{aligned}
\end{equation}
where the last equality follows from the special case:\begin{equation*}
	T_{s_\alpha}\theta_\alpha=\theta_{-\alpha}T_{s_\alpha}+(p-1)(1+\theta_\alpha)
\end{equation*}

Using \eqref{HIso2} and the same method as in \eqref{rightmul}, for every $\psi'\in \mathcal{H}$ we can define an action on $i_{\mathcal{B}}^{\mathcal{G}}(\chi_{\mathrm{univ}})^{\mathcal{J}}$ by
\begin{equation}\label{rightmul2}
	\psi': \psi\circ f_1\mapsto (\psi\star\psi')\circ f_1
\end{equation}
which is a left $\mathcal{H}$-module isomorphism.

Taking $\psi'=\iota_{s_\alpha}$, we define the intertwining operator $I_{s_\alpha}$ on $i_{\mathcal{B}}^{\mathcal{G}}(\chi_{\mathrm{univ}})^{\mathcal{J}}$ via \eqref{rightmul2}. 

Using the equivalence of categories between $\mathrm{Rep}(\mathcal{G})^{(\mathcal{J},1_{\triv})}$ and $\mathcal{H}-\mathrm{Mod}$,  we can transfer $I_{s_\alpha}$ to an intertwining operator on $i_{\mathcal{B}}^{\mathcal{G}}(\chi_{\mathrm{univ}})$.

Let $\varphi$ be a principal character on $\mathcal{U}$. Since $R$ is a commutative $\mathbb{C}$-algebra, the inclusion  $\mathbb{C} \hookrightarrow R$ allows us to view $\varphi$ as a character with values in $R^\times$.

A Whittaker functional on $i_{\mathcal{B}}^{\mathcal{G}}(\chi_{\mathrm{univ}})$ is a right $R$-module map \begin{equation*}
	\mathcal{W}: i_{\mathcal{B}}^{\mathcal{G}}(\chi_{\mathrm{univ}})\to R
\end{equation*}
such that $\mathcal{W}(\pi(u)f)=\varphi(u)\mathcal{W}(f)$ for all $u\in\mathcal{U}$ and all $f\in i_{\mathcal{B}}^{\mathcal{G}}(\chi_{\mathrm{univ}})$.

\begin{proposition}[\cite{haines2003iwahori}]\label{WhiRank1}
	Let $\varphi$ be a principal character of $\mathcal{U}$. The $R$-module of all Whittaker functionals is free of rank 1.
\end{proposition}
Denote by $w_0$ the longest element of $W$. We realize the Whittaker functional $\mathcal{W}$ as the unique functional whose restriction to functions $f\in i_{\mathcal{B}}^{\mathcal{G}}(\chi_{\mathrm{univ}})$ supported on $\mathcal{B}w_0\mathcal{B}$ is given by  \begin{equation}\label{WhFun}
	\mathcal{W}(f)=\int_{\mathcal{U}}f(w_0u)\varphi^{-1}(u)du
\end{equation}
where $du$ denotes the left Haar measure on $\mathcal{U}$ that gives measure 1 to $\mathcal{N}\cap\mathcal{J}$.

From now on, we assume  that  $\varphi$ is principal and unramified.

\begin{proposition}\label{IntWhi}
	Let $\mathcal{W}$ be the Whittaker functional defined above \eqref{WhFun}. Then \begin{equation*}
		\mathcal{W}(I_{s_\alpha}\circ f)=\mathcal{W}(f)^{s_\alpha}(p-\theta_{-\alpha})
	\end{equation*}
	where $\mathcal{W}(f)^{s_\alpha}$ denotes the usual action of $s_\alpha$ on $R=\mathbb{C}[\mathcal{T}/\mathcal{T}^0]$, and $\theta_{-\alpha}$ is viewed as an element of $R$ via \eqref{RIso}.
\end{proposition}
\begin{proof}

	Denote $f_{w_0}=f_{(1,w_0)}$, $f_{s_\alpha w_0}=f_{(1,{s_\alpha w_0})}$, with $w_0$ the longest element in $W$. By \cite{CassUnP},
	$\mathcal{W}(f_{s_\alpha w_0}+f_{w_0})=1-p^{-1}\alpha(\varpi)$.
	From \cite[Lemma 5.10]{ChirssPadic}, \begin{align*}
		T_{w_0^{-1}}\circ f_1&=f_{w_0}\\
		T_{w_0^{-1}s_\alpha^{-1}}\circ f_1&=f_{s_\alpha w_0}
	\end{align*}
	Then  \begin{align*}
		I_{s_\alpha}\circ(f_{s_\alpha w_0}+f_{w_0})&=((T_{w_0^{-1}s_\alpha^{-1}}+T_{w_0^{-1}})\star \iota_{s_\alpha})\circ f_1\\
		&=((T_{w_0^{-1}s_\alpha^{-1}}+T_{w_0^{-1}})\cdot(T_{s_\alpha}(1-\theta_\alpha)+(p-1)\theta_\alpha))\circ f_1\\
		&=((T_{w_0^{-1}s_\alpha^{-1}}+T_{w_0^{-1}})\cdot(p-\theta_\alpha))\circ f_1\\
		&=(f_{s_\alpha w_0}+f_{w_0})\cdot (p-\alpha(\varpi^{-1}))
	\end{align*}
    Hence
	\begin{equation*}
		\mathcal{W}(I_{s_\alpha}\circ (f_{s_\alpha w_0}+f_{w_0}))=\mathcal{W}((f_{s_\alpha w_0}+f_{w_0}))^{s_\alpha}\cdot(p-\alpha(\varpi))
	\end{equation*}
	By Proposition \ref{WhiRank1}, we conclude \begin{equation*}
		\mathcal{W}(I_{s_\alpha}\circ f)=\mathcal{W}(f)^{s_\alpha}(p-\theta_{-\alpha})
	\end{equation*}
\end{proof}

In Section 1, we define an action $\gamma$ on $\mathcal{G}$ which is given by $\gamma(g)=J(^tg^{-1})J^{-1}$. This induces an action $\gamma$ on $\mathcal{H} = C_c^\infty(\mathcal{J}\backslash\mathcal{G}/\mathcal{J})$ via $\gamma(f)(g) = f(\gamma(g))$. 
 We can also define an action on $i_{\mathcal{B}}^{\mathcal{G}}(\chi_{\mathrm{univ}})^{\mathcal{J}}$ by
\begin{equation*}
	\gamma: \psi\circ f_1\mapsto \gamma(\psi)\circ f_1
\end{equation*}
and an action $\gamma$  on $R=\mathbb{C}[\mathcal{T}/\mathcal{T}^0]$ induced by the $\gamma$ action on $\mathcal{T}$.

\begin{proposition}\label{WThP22}
	Let $\varphi$ be a character of $\mathcal{U}$ invariant under $\gamma$. Then
	\begin{equation*}
		\mathcal{W}(\gamma(f))=\mathcal{W}(f)^\gamma
	\end{equation*}
\end{proposition}
\begin{proof}
We have
	\begin{align*}
		\gamma(\pi(u)f)(x)=(\pi(u)f)(\gamma(x))=f(\gamma(x)u)=(\pi(\gamma(u))\gamma(f))(x)
	\end{align*}      
	Hence \begin{align*}
		\mathcal{W} (\gamma(\pi(u)f))&=\mathcal{W}(\pi(\gamma(u))\gamma(f))=\varphi(\gamma(u))\mathcal{W}(\gamma(f))\\
		&=\varphi(u)\mathcal{W}(\gamma(f))
	\end{align*}
	By Proposition \ref{WhiRank1}, $\mathcal{W}(\gamma(f))=c\cdot\mathcal{W}(f)^\gamma$ for some $c\in R$. Take $f_1$, then $\mathcal{W}(\gamma(f_1))=\mathcal{W}(f_1)=\mathcal{W}(f_1)^\gamma=p^{-l(w_0)}$, where $w_0$ is the longest element in $W$. Hence $c=1$, and the proposition follows.
\end{proof}

Recall that the center of $\mathcal{H}$ is $R^W$. Let $\mathcal{F}$ denote the quotient field of the center. Define \begin{equation*}
	_{\mathcal{F}}\mathcal{H}:=\mathcal{F}\otimes_{R^W}\mathcal{H}
\end{equation*}

For a simple reflection $s_\alpha$, define a normalized intertwining element in $_{\mathcal{F}}\mathcal{H}$\cite{SolUn}
\begin{align*}
	\iota_{s_\alpha}^0&=\iota_{s_\alpha}(p-\theta_{-\alpha})^{-1}
\end{align*}
Let $w=s_1\cdots s_m$ be a reduced expression for $w\in W$, where each $s_i$ is a simple reflection. Set \begin{equation*}
	\iota_w=\iota_{s_1}\cdots\iota_{s_m}
\end{equation*} 
\begin{equation*}
	\iota^0_w=\iota^0_{s_1}\cdots\iota^0_{s_m}
\end{equation*}

\begin{lemma}\label{IntPa}
	\begin{equation*}
		\iota_w^0=\iota_w\cdot(\prod_{\beta\in R_w}(p-\theta_{-\beta}))^{-1}
	\end{equation*}
	where $R_w$ is the set of positive roots $\beta$ such that $w^{-1}\beta$ is negative.
\end{lemma}
\begin{proof}
	Using the multiplication in $\mathcal{H}$, we have \begin{equation}\label{ProCom51}
		\iota_{s_\alpha}\cdot\theta_y=\theta_{s_\alpha (y)}\cdot\iota_{s_\alpha}
	\end{equation}
    The case $l(w)=1$ is clear. Now, we prove it by induction on $l(w)$. If $l(ws_\alpha)>l(w)$, we have $R_{ws_\alpha}=\{\alpha\}\sqcup\{s_\alpha(\beta)|\beta\in R_w\}$. Denote $n(\beta)=p-\theta_{-\beta}$. We have $\prod_{\beta\in R_{ws_\alpha}}n(\beta)=\prod_{\beta\in R_{w}}n(s_\alpha(\beta))\cdot n(\alpha)$. By the induction hypothesis, we have $\iota^0_w=\iota_w\cdot(\prod_{\beta\in R_{w}}n(\beta))^{-1}$. Hence  combining \eqref{ProCom51}, we get \begin{align*}
		\iota^0_{ws\alpha}&=\iota^0_{w}\iota^0_{s\alpha}=\iota_w\cdot(\prod_{\beta\in R_{w}}n(\beta))^{-1}\cdot\iota_{s_\alpha}\cdot n(\alpha)^{-1}\\&=\iota_{w}\iota_{s_\alpha}(\prod_{\beta\in R_{w}}n(s_\alpha(\beta))\cdot n(\alpha))^{-1}\\
		&=\iota_{ws_\alpha}(\prod_{\beta\in R_{ws_\alpha}}n(\beta))^{-1}
	\end{align*}
\end{proof}

Recall from Proposition \ref{HIsoPro1} the isomorphism $\Psi: \mathcal{H} \to i_{\mathcal{B}}^{\mathcal{G}}(\chi_{\mathrm{univ}})^{\mathcal{J}}$, and from \eqref{rightmul2} the action on $i_{\mathcal{B}}^{\mathcal{G}}(\chi_{\mathrm{univ}})^{\mathcal{J}}$ by right multiplication on $\mathcal{H}$. Taking $\psi' = \iota_w$, we define the intertwining operator $I_w$ on $i_{\mathcal{B}}^{\mathcal{G}}(\chi_{\mathrm{univ}})^{\mathcal{J}}$ by \eqref{rightmul2}:

\begin{equation*}
	I_w: h\circ f_1\mapsto (h\cdot \iota_{w})\circ f_1
\end{equation*}
where $h\in \mathcal{H}$.

Define \begin{equation*}
	n_w:=\prod_{\beta\in R_w}(p-\theta_{-\beta})
\end{equation*}

Combining Proposition \ref{IntWhi} and Lemma \ref{IntPa}, we have 
\begin{proposition}\label{WThP21}
   Let $\mathcal{W}$ be the Whittaker functional defined above \eqref{WhFun}. Then 
    \begin{equation*}
	\mathcal{W}(I_{w}\circ f)=\mathcal{W}(f)^{w}n_w
\end{equation*}
\end{proposition}

\subsubsection{Induction}

Recall we have \begin{equation*}
	\pi\hookrightarrow\sigma= (\times_\lambda\rho|\cdot|_F^{\lambda})\times(\times_{\lambda'}\rho'|\cdot|_F^{\lambda'})
\end{equation*}
where $\rho=|\cdot|_F^z$ with $p^z=1$,  $\rho'=|\cdot|_F^{z'}$ with $p^{z'}=-1$,  and $\lambda$, $\lambda'$ are  half integers taken in decreasing order.

Denote $t=(\times_\lambda p^{\lambda})\times(\times_{\lambda'}-p^{\lambda'})$, hence $t_c=(1,\cdots,1,-1,\cdots,-1)$.

Recall that $G=\GL_n(\mathbb{C})$ is the complex dual group of $\mathcal{G}=\GL_n(\Qp)$, $T$ is a maximal torus in $G$, and $W=N_G(T)/T$ is the Weyl group. Then $M= Z_G(t_c)=\GL_k(\mathbb{C})\times \GL_l(\mathbb{C})$. Let $X$ be the character lattice of $T$. 

Recall  $G^+:=G\rtimes \hat{\Gamma}$, where $\hat{\Gamma} = \langle\hat{\gamma}\rangle$.
 Denote $M^+=Z_{G^+}(t_c)$ and let $w_M\in W$  be a permutation that swaps the first block of size $l$ with the last block of size $k$:
\begin{align*}
	w_M=(s_ks_{k+1}\cdots s_{k+l-1})(s_{k-1}s_k\cdots s_{k+l-2})\cdots(s_1s_2\cdots s_l)\in W
\end{align*}

Then $M^+$ is generated by $M$ and $(\overline{w}_{M},\hat{\gamma})$, where $\overline{w}_{M}$ is a representative of $w_M\in W$ in $G$.

Recall that  $\hat{\Gamma}$ induces an action on $W$ and $T$. We define an action $\hat{\gamma}_M$ on $W$ and $T$ by\begin{align}
    \hat{\gamma}_M(w)&=w_M\cdot \hat{\gamma}(w)\cdot w_M^{-1}\quad w\in W\\
\hat{\gamma}_M(s)&=\overline{w}_{M}\cdot \hat{\gamma}(s)\cdot \overline{w}_{M}^{-1}\quad s\in T
\end{align}
Since $\mathbb{C}[X]\otimes\mathbb{C}[v,v^{-1}]\simeq \mathcal{O}(T\times\mathbb{C}^*)$ (the ring of regular functions on $T\times \mathbb{C}^*$) and 
$\mathcal{H}(M,v)$ is generated by $\mathbb{C}[W]$ and  $\mathbb{C}[X]\otimes\mathbb{C}[v,v^{-1}]$,
this $\hat{\gamma}_M$ induces an action on $\mathcal{H}(M,v)$ and we obtain $\mathcal{H}(M,v)\rtimes \langle\hat{\gamma}_M\rangle$. Define \begin{equation*}
    \mathcal{H}(M^+,v):= \mathcal{H}(M,v)\rtimes \langle\hat{\gamma}_M\rangle
\end{equation*}

The unramified principal series satisfies
\begin{align*}
	\sigma^{\mathcal{J}}&=\mathcal{H}(G,p^{1/2})\otimes_{\mathbb{C}[X]}\mathbb{C}_t\\
	&=\mathcal{H}(G,v)\otimes_{\mathcal{H}(T,v)}\mathbb{C}_{(t,p^{1/2})}\\
	&=\mathcal{H}(G,v)\otimes_{\mathcal{H}(M,v)}\mathcal{H}(M,v)\otimes_{\mathcal{H}(T,v)}\mathbb{C}_{(t,p^{1/2})}\\
	&=\mathcal{H}(G,v)\otimes_{\mathcal{H}(M,v)}\sigma_{\mathcal{M}}^{\mathcal{J}_{\mathcal{M}}}
\end{align*}
where $\mathbb{C}_t$ is the one-dimensional module of $\mathbb{C}[X]\simeq \mathcal{O}(T)$ given by evaluation at $t$, and $\mathbb{C}_{(t,p^{1/2})}$ is the one-dimensional module of $\mathcal{H}(T,v)\simeq\mathbb{C}[X]\otimes\mathbb{C}[v,v^{-1}]\simeq \mathcal{O}(T\times\mathbb{C}^*)$ given by evaluation at $(t,p^{1/2})$.

For convenience, we denote $V=\sigma^{\mathcal{J}}$ and \begin{align*}
V_M=\sigma_{\mathcal{M}}^{\mathcal{J}_{\mathcal{M}}}=\mathcal{H}(M,v)\otimes_{\mathcal{H}(T,v)}\mathbb{C}_{(t,p^{1/2})}
\end{align*} 
then $V=\mathcal{H}(G,v)\otimes_{\mathcal{H}(M,v)}V_M$.

Recall $t=(\times_\lambda p^{\lambda})\times(\times_{\lambda'}-p^{\lambda'})\in T$ satisfies $p^{\lambda_1}\geq\cdots\geq p^{\lambda_k}$ and $p^{\lambda'_1}\geq \cdots \geq p^{\lambda_l}$.

Since $\pi\simeq \gamma^* \pi$, we have 
\begin{equation}
    \begin{aligned}
        \lambda_{k+1-i}&=-\lambda_i\\
        \lambda'_{l+1-i}&=-\lambda'_i
    \end{aligned}
\end{equation}

 Hence $\hat{\gamma}_M({t})={t}$.

We  define a $\hat{\gamma}_M$ action on $V_M$   by \begin{equation*}
	\hat{\gamma}_M(h\otimes a)=\hat{\gamma}_M(h)\otimes a
\end{equation*}

where $h\in \mathcal{H}(M,v)$ and $a\in \mathbb{C}_{({t},p^{1/2})}$.

The action is well defined because \begin{align*}
   \hat{\gamma}_M (h_1h_2\otimes a)&=\hat{\gamma}_M(h_1h_2)\otimes a=\hat{\gamma}_M(h_1)\hat{\gamma}_M(h_2)\otimes a\\
&=\hat{\gamma}_M(h_1)\otimes\hat{\gamma}_M(h_2)\cdot a=\hat{\gamma}_M(h_1)\otimes(\hat{\gamma}_M(h_2))({t},p^{1/2})a\\
&=\hat{\gamma}_M(h_1)\otimes h_2(\hat{\gamma}_M({t}),p^{1/2})a=\hat{\gamma}_M(h_1)\otimes h_2({t},p^{1/2})a\\
&=\hat{\gamma}_M(h_1\otimes h_2\cdot a)
\end{align*}
where $h_1\in  \mathcal{H}(M,v)$, $h_2\in \mathcal{H}(T,v)$, and $a\in \mathbb{C}_{({t},p^{1/2})}$.

Denote by ${\hat{\gamma}_M}^*(V_M)$ the module of $\mathcal{H}(M,v)$ with the same underlying vector space but with the $\mathcal{H}(M,v)$ action twisted by ${\hat{\gamma}_M}$.

 Then $\hat{\gamma}_M$ is an isomorphism from $V_M$ to ${\hat{\gamma}_M^*}V_M$:\begin{align*}
	\hat{\gamma}_M(h'\cdot (h\otimes a ))&=\hat{\gamma}_M(h'h\otimes a)=\hat{\gamma}_M(h'h)\otimes a\\
	&=\hat{\gamma}_M(h')\hat{\gamma}_M(h)\otimes a\\
	&=\hat{\gamma}_M(h')\cdot \hat{\gamma}_M(h\otimes a)
\end{align*}
Thus we can view $V_M$ as an $\mathcal{H}(M^+,v)$ module by:
\begin{align*}
    (h,1)\cdot v&= h\cdot v\\
    (h,\hat{\gamma}_M)\cdot v&= h\cdot \hat{\gamma}_M(v)
\end{align*}

Recall that $\mathscr{A}^{W^+}$ is the center of  $\mathcal{H}(G^+,v)\simeq \mathcal{H}(G,v)\rtimes\langle\hat{\gamma}\rangle$, and $\mathscr{I}_{({t},p^{1/2})}^+$ is the maximal ideal of $\mathscr{A}^{W^+}$ associated with the orbit $(W^+\cdot {t},p^{1/2})$.

Denote by $\widehat{\mathscr{A}^{W^+}_{({t},p^{1/2})}}$ the $\mathscr{I}_{({t},p^{1/2})}^+$-adic completion of $\mathscr{A}^{W^+}$. Now we consider the structure of the completion\begin{align}
	\widehat{\mathcal{H}}(G^+,v):=\widehat{\mathscr{A}^{W^+}_{({t},p^{1/2})}}\otimes_{\mathscr{A}^{W^+}}\mathcal{H}(G^+,v)
\end{align}

Let $W^+_M=N_{M^+}(T)/T$. The center of $\mathcal{H}(M^+,v)$ is $\mathscr{A}^{W^+_M}$.  Denote by $\widehat{\mathscr{A}^{W^+_M}_{({t},p^{1/2})}}$ the  completion of $\mathscr{A}^{W^+_M}$ at the maximal ideal associated with the orbit $(W^+_M\cdot {t},p^{1/2})$.

Put $\varpi=W_M^+ \cdot{t}$, the equivalence class of ${t}$.
We associate idempotents $e_{w\varpi}\in \widehat{\mathcal{H}}(G^+,v)$  with the equivalence classes $w\varpi\in W^+\cdot{t}$. We have the decomposition\begin{align*}
	\widehat{\mathcal{H}}(G^+,v) &=\bigoplus_{u,v\in W/W_M} \iota_u^0 e_\varpi \widehat{\mathcal{H}}(M^+,v)\iota_{v^{-1}}^0\\
	&=\bigoplus_{u,v\in W/W_M} e_{u\varpi} \widehat{\mathcal{H}}(G^+,v)e_{v\varpi}
\end{align*}

 \begin{proposition}[\cite{SolUn}]\label{ComCatEq1}
	(1) There is a natural algebraic  isomorphism 
	\begin{equation*}
		\widehat{\mathcal{H}}(M^+,v)\simeq e_{\varpi} \widehat{\mathcal{H}}(G^+,v)e_{\varpi}
	\end{equation*}
	
	(2) There is a natural algebraic  isomorphism 
	\begin{equation*}
		\widehat{\mathcal{H}}(G^+,v)\simeq  (\widehat{\mathcal{H}}(M^+,v))_N
	\end{equation*}
	the algebra of $N\times N$ matrices with entries in $\widehat{\mathcal{H}}(M^+,v)$, where $N=[W:W_M]$.
\end{proposition}

The embedding of algebras
\begin{align*}
	\widehat{\mathcal{H}}(M^+,v)\simeq {e_{\varpi}\widehat{\mathcal{H}}(M^+,v)}\hookrightarrow \widehat{\mathcal{H}}(G^+,v)
\end{align*}
sends \begin{align*}
	\hat{\gamma}_M\mapsto e_\varpi(\iota_{w_M}^0,\hat{\gamma}) 
\end{align*}
where $\iota_{w_M}^0$ is the intertwining element for $w_M$.

Note that if we delete the "$+$" in the structure of completion, we obtain the same decomposition and proposition \cite{Lusztig1989AffineHA}.

Hence we have the isomorphism\begin{align*}
	\widehat{\mathcal{H}}(G,v)\otimes_{e_\varpi\widehat{\mathcal{H}}(M,v)}V_M
	&\simeq  \widehat{\mathcal{H}}(G^+,v)\otimes_{e_\varpi\widehat{\mathcal{H}}(M^+,v)}V_M \\
	g\otimes v &\mapsto g\otimes v
\end{align*}

where ${e_{\varpi}\widehat{\mathcal{H}}(M,v)}$ (resp. ${e_{\varpi}\widehat{\mathcal{H}}(M^+,v)}$ ) acts on $V_M$ via  ${e_{\varpi}\widehat{\mathcal{H}}(M,v)}\simeq  \widehat{\mathcal{H}}(M,v)$ (resp. ${e_{\varpi}\widehat{\mathcal{H}}(M^+,v)}\simeq  \widehat{\mathcal{H}}(M^+,v)$).

Recall we defined \begin{equation*}
	n_w:=\prod_{\beta\in R_w}(p-\theta_{-\beta})
\end{equation*}
Denote\begin{equation*}
	n_w(t):=\prod_{\beta\in R_w}(p-\theta_{-\beta}(t))
\end{equation*}

Hence $ n_{w_M}(t)\neq 0$ and  $ n_{w_M^{-1}}(t)\neq 0$, which implies
 $\iota_{w_M}^0\in  \widehat{\mathcal{H}}(G,v)$ and $\iota_{w_M^{-1}}^0\in  \widehat{\mathcal{H}}(G,v)$.

Using this isomorphism, we obtain a $\hat{\gamma}$ action on $\widehat{\mathcal{H}}(G,v)\otimes_{e_\varpi\widehat{\mathcal{H}}(M,v)}V_M$:\begin{align*}
	\hat{\gamma}\cdot (g\otimes v) &= ((1,\hat{\gamma})\cdot (g,1)) \otimes v= \hat{\gamma}(g)\cdot (1,\hat{\gamma}) \otimes v\\
	&=\hat{\gamma}(g)\iota_{w_M^{-1}}^0 (\iota_{w_M}^0,\hat{\gamma}) \otimes v= \hat{\gamma}(g)\iota_{w_M^{-1}}^0\otimes e_\varpi (\iota_{w_M}^0,\hat{\gamma}) \cdot v \\
	&=\hat{\gamma}(g)\iota_{w_M^{-1}}^0\otimes \hat{\gamma}_M (v)
\end{align*}

We have:
\begin{align}
	V&=\mathcal{H}(G,v)\otimes_{\mathcal{H}(T,v)}\mathbb{C}_{(t,p^{1/2})}\nonumber\\
	&=\mathcal{H}(G,v)\otimes_{\mathcal{H}(M,v)}V_M\nonumber\\
	&\simeq \widehat{\mathcal{H}}(G,v)\otimes_{e_\varpi\widehat{\mathcal{H}}(M,v)}V_M\nonumber\\
	&\simeq  \widehat{\mathcal{H}}(G^+,v)\otimes_{e_\varpi\widehat{\mathcal{H}}(M^+,v)}V_M \label{GaAct1}
\end{align}

Thus we have a $(1,\hat{\gamma})$ action on $V$ via \eqref{GaAct1}

\begin{equation*}
	\begin{tikzcd}
		\mathcal{H}(G,v)\otimes_{\mathcal{H}(T,v)}\mathbb{C}_{(t,p^{1/2})} \ar[r,] \ar[d,"{(1,\hat{\gamma})}"]& \widehat{\mathcal{H}}(G,v)\otimes_{e_\varpi\widehat{\mathcal{H}}(M,v)}V_M\ar[d,"{(1,\hat{\gamma})}"] \ar[r] & \widehat{\mathcal{H}}(G^+,v)\otimes_{e_\varpi\widehat{\mathcal{H}}(M^+,v)}V_M \ar[d,"{(1,\hat{\gamma})}"]\\
		\mathcal{H}(G,v)\otimes_{\mathcal{H}(T,v)}\mathbb{C}_{(t,p^{1/2})} \ar[r]& \widehat{\mathcal{H}}(G,v)\otimes_{e_\varpi\widehat{\mathcal{H}}(M,v)}V_M \ar[r] & \widehat{\mathcal{H}}(G^+,v)\otimes_{e_\varpi\widehat{\mathcal{H}}(M^+,v)}V_M 
	\end{tikzcd}
\end{equation*}
which is given explicitly by
\begin{equation*}
	\begin{tikzcd}
		g\otimes a\ar[r,mapsto] \ar[d,"{(1,\hat{\gamma})}"]& g\otimes a  \ar[r,mapsto] \ar[d,"{(1,\hat{\gamma})}"]& g\otimes a \ar[d,"{(1,\hat{\gamma})}"]\\
		\hat{\gamma}(g)\iota_{w_M^{-1}}\otimes (n_{w_M^{-1}}(t))^{-1} a \ar[r]& \hat{\gamma}(g)\iota_{w_M^{-1}}^0\otimes a \ar[r] &  \hat{\gamma}(g)\iota_{w_M^{-1}}^0\otimes a
	\end{tikzcd}
\end{equation*}
where $a\in \mathbb{C}_{(t,p^{1/2})}$, $g\in \mathcal{H}(G,v)$, because $\hat{\gamma}_M(a)=a$ and $\iota_{w_M^{-1}}^0=\iota_{w_M^{-1}}^0\cdot (n_{w_M^{-1}})^{-1}$.

Hence we get a $\hat{\gamma}=(1,\hat{\gamma})$ action on $V$ given by \begin{equation*}
	\hat{\gamma}(g\otimes a)=\hat{\gamma}(g)\iota_{w_M^{-1}}\otimes (n_{w_M^{-1}}(t))^{-1} a
\end{equation*}
Denote by ${\hat{\gamma}}^*V$ the $\mathcal{H}(G,v)$ module with the same underlying vector space but with the action twisted by ${\hat{\gamma}}$. Then $\hat{\gamma}$ is an isomorphism from $V$ to ${\hat{\gamma}}^*V$.

Recall 
$\sigma= (\times_\lambda\rho|\cdot|_F^{\lambda})\times(\times_{\lambda'}\rho'|\cdot|_F^{\lambda'})$ and $V=\sigma^{\mathcal{J}}$. We have a ${\gamma}$ action on $\sigma$, which is an isomorphism of $\mathcal{G}$ representation from $\sigma$ to ${{\gamma}}^*\sigma$ by the equivalence of categories \eqref{TwiE}.

Since \begin{equation*}
	i_{\mathcal{B}}^{\mathcal{G}}(\chi_{\mathrm{univ}})\otimes_R\mathbb{C}_{(\lambda,\lambda')} \simeq i_{\mathcal{B}}^{\mathcal{G}}(\mathbb{C}_{(\lambda,\lambda')})=\sigma
\end{equation*} the formula \begin{equation}\label{WThP1}
	\mathcal{W}(f\otimes a)=\mathcal{W}(f)\otimes a
\end{equation}  defines a Whittaker functional on  $\sigma$.
\begin{theorem}
	For the Whittaker functional $\mathcal{W}$ on $\sigma$, the ${\gamma}$ action on $\sigma$ satisfies
	\begin{equation*}
		\mathcal{W}({\gamma}(x))=\mathcal{W}(x)
	\end{equation*}
\end{theorem}
\begin{proof}
	We have the commutative diagram
	\begin{equation*}
		\begin{tikzcd}
			\mathcal{H}(G,v)\otimes_{\mathcal{H}(T,v)}\mathbb{C}_{(t,p^{1/2})}\arrow[d,"\simeq"]\arrow[dr,"\simeq"]& \\
			i_{\mathcal{B}}^{\mathcal{G}}(\chi_{\mathrm{univ}})^{\mathcal{J}}\otimes_R\mathbb{C}_{(\lambda,\lambda')} \arrow[r,"\simeq"]&i_{\mathcal{B}}^{\mathcal{G}}(\mathbb{C}_{(\lambda,\lambda')})^{\mathcal{J}}
		\end{tikzcd}
	\end{equation*}
	Taking $g\in \mathcal{H}(G,v)\simeq i_{\mathcal{B}}^{\mathcal{G}}(\chi_{\mathrm{univ}})^{\mathcal{J}}$, we have\begin{align}
		\mathcal{W}(\hat{\gamma}(g\otimes a)) &= \mathcal{W}(\hat{\gamma}(g)\iota_{w_M^{-1}}\otimes (n_{w_M^{-1}}(t))^{-1} a))\nonumber\\   
		&=\mathcal{W}(\hat{\gamma}(g)\iota_{w_M^{-1}})\otimes (n_{w_M^{-1}}(t))^{-1} a)\label{WThE1}\\
		&= (\mathcal{W}(g)^{w_M^{-1}})^{\hat{\gamma}}n_{w_M^{-1}}\otimes (n_{w_M^{-1}}(t))^{-1} a\label{WThE2}\\
		&= \mathcal{W}(g)^{\hat{\gamma}_M}\otimes a \label{WThE3} \\
		&=  \mathcal{W}(g)\otimes a \label{WThE4}
	\end{align}
	Here \eqref{WThE1} comes from \eqref{WThP1}, \eqref{WThE2} comes from \eqref{WThP21} and Proposition \ref{WThP22}, 
	\eqref{WThE3} comes from the definition of $\hat{\gamma}_M$, and
	\eqref{WThE4} comes from $\hat{\gamma}_M(t)=t$.
	
\end{proof}

Recall that we have a tempered representation $\pi$ as the unique subrepresentation of $\sigma$, then ${\gamma}$ restricted to $\pi$ is “Whittaker-normalized". 

Recall $\pi=i_{\mathcal{M}}^{\mathcal{G}}(\pi_1\otimes\pi_{-1})$. Denote $\bar{V}_M=(\pi_{1}\otimes\pi'_{-1})^{\mathcal{J}_{\mathcal{M}}}$. Then  $\bar{V}_M$ is a submodule of $V_M$ and $\pi^{\mathcal{J}}=\mathcal{H}(G,v)\otimes_{\mathcal{H}(M,v)}\bar{V}_M$ is a submodule of $\sigma^{\mathcal{J}}=V=\mathcal{H}(G,v)\otimes_{\mathcal{H}(M,v)}{V}_M$. Recall that we have an action $\hat{\gamma}_M$ on $V_M$ which is an isomorphism from $V_M$ to ${\hat{\gamma}_M^*}V_M$. Since  $\bar{V}_M$ is the unique submodule of $V_M$, the $\hat{\gamma}_M$ action preserves $\bar{V}_M$ and we get the  $\hat{\gamma}_M$ action on $\bar{V}_M$ by restriction.

Because the $\hat{\gamma}$ action on $V=\mathcal{H}(G,v)\otimes_{\mathcal{H}(M,v)}{V}\simeq \widehat{\mathcal{H}}(G,v)\otimes_{e_\varpi\widehat{\mathcal{H}}(M,v)}V_M$ is given by $\hat{\gamma}\cdot (g\otimes v)=\hat{\gamma}(g)\iota_{w_M^{-1}}^0\otimes \hat{\gamma}_M (v)$, the restriction of $\hat{\gamma}$ on $\mathcal{H}(G,v)\otimes_{\mathcal{H}(M,v)}\bar{V}_M\simeq \widehat{\mathcal{H}}(G,v)\otimes_{e_\varpi\widehat{\mathcal{H}}(M,v)}\bar{V}_M$ is also given by $\hat{\gamma}\cdot (g\otimes v)=\hat{\gamma}(g)\iota_{w_M^{-1}}^0\otimes \hat{\gamma}_M (v)$ where $v\in \bar{V}_M$ and $\hat{\gamma}_M$ is the restriction.

We summarize what has been proved in the above subsections and get the following theorem:\begin{theorem}
	For a tempered representation $\pi$ which can  be viewed as a unique subrepresentation of the unramified principal series:  $\pi\hookrightarrow\sigma$, denote $\sigma^{\mathcal{J}}\simeq \mathcal{H}(G,v)\otimes_{\mathcal{H}(M,v)} V_M=\mathcal{H}(G,v)\otimes_{\mathcal{H}(T,v)}\mathbb{C}_{(t,p^{1/2})}$ and  $\pi^{\mathcal{J}}\simeq \mathcal{H}(G,v)\otimes_{\mathcal{H}(M,v)}\bar{V}_M$.  If we give an action $\hat{\gamma}_M$ on $V_M =\mathcal{H}(M,v)\otimes_{\mathcal{H}(T,v)}\mathbb{C}_{({t},p^{1/2})}$ induced by the $\hat{\gamma}_M$ action on $\mathcal{H}(M,v)$ and the trivial action on  $\mathbb{C}_{(t,p^{1/2})}$, we have an action $\hat{\gamma}_M$ on $\bar{V}_M$ by restriction. 

    The action on $\pi^{\mathcal{J}}=\mathcal{H}(G,v)\otimes_{\mathcal{H}(M,v)}\bar{V}_M\simeq \widehat{\mathcal{H}}(G,v)\otimes_{e_\varpi\widehat{\mathcal{H}}(M,v)}\bar{V}_M$ given by $\hat{\gamma}\cdot (g\otimes v)=\hat{\gamma}(g)\iota_{w_M^{-1}}^0\otimes \hat{\gamma}_M (v)$ induces a Whittaker-normalized action from $\pi$ to $\gamma^*\pi$.
\end{theorem}

The $\hat{\gamma}_M$ action on $\bar{V}_M$ is an isomorphism from $\bar{V}_M$ to $\hat{\gamma}_M^*\bar{V}_M$. Thus we can view $\bar{V}_M$ as an $\mathcal{H}(M^+,v)$ module. 

From Lusztig's first reduction theorem for the twisted version (Proposition \ref{ComCatEq1}), we get an $\widehat{\mathcal{H}}(G^+,v)$ module by
$\widehat{\mathcal{H}}(G^+,v)\otimes_{e_\varpi\widehat{\mathcal{H}}(M^+,v)}\bar{V}_M$.

Using the same method, we also have \begin{align*}
	\pi^{\mathcal{J}}&\simeq \mathcal{H}(G,v)\otimes_{\mathcal{H}(M,v)}\bar{V}_M\\ 
	&\simeq \widehat{\mathcal{H}}(G,v)\otimes_{e_\varpi\widehat{\mathcal{H}}(M,v)}\bar{V}_M\\
	&\simeq  \widehat{\mathcal{H}}(G^+,v)\otimes_{e_\varpi\widehat{\mathcal{H}}(M^+,v)}\bar{V}_M 
\end{align*}
We have \begin{align*}
	\widehat{\mathcal{H}}(G^+,v)\otimes_{e_\varpi\widehat{\mathcal{H}}(M^+,v)}\bar{V}_M\hookrightarrow\widehat{\mathcal{H}}(G^+,v)\otimes_{e_\varpi\widehat{\mathcal{H}}(M^+,v)}V_M 
\end{align*}
so the $\hat{\gamma}$ action induced by $\widehat{\mathcal{H}}(G^+,v)\otimes_{e_\varpi\widehat{\mathcal{H}}(M^+,v)}\bar{V}_M$ is the same as the $\hat{\gamma}$ action restricted from $\sigma$.
 Then we have another description of the theorem

\begin{theorem}\label{MainWhTem}
	For a tempered representation $\pi$ which can  be viewed as a unique subrepresentation of the unramified principal series:  $\pi\hookrightarrow\sigma$, denote $\sigma^{\mathcal{J}}\simeq \mathcal{H}(G,v)\otimes_{\mathcal{H}(M,v)} V_M=\mathcal{H}(G,v)\otimes_{\mathcal{H}(T,v)}\mathbb{C}_{(t,p^{1/2})}$ and  $\pi^{\mathcal{J}}\simeq \mathcal{H}(G,v)\otimes_{\mathcal{H}(M,v)}\bar{V}_M$. If we give an action $\hat{\gamma}_M$ on $V_M =\mathcal{H}(M,v)\otimes_{\mathcal{H}(T,v)}\mathbb{C}_{({t},p^{1/2})}$ induced by the $\hat{\gamma}_M$ action on $\mathcal{H}(M,v)$ and the trivial action on  $\mathbb{C}_{(t,p^{1/2})}$, we have an action $\hat{\gamma}_M$ on $\bar{V}_M$ by restriction.  Thus we can view $\bar{V}_M$ as an $\mathcal{H}(M^+,v)$ module.
	
	Then, by Lusztig's first reduction theorem for the twisted version (Proposition \ref{ComCatEq1}), we obtain a $\mathcal{G}^+$ representation $\pi$ with a Whittaker-normalized $\gamma$-action.
\end{theorem}

\subsection{Geometry and the main theorem}
We use the notation of Section 5.1.

Recall $\mathcal{Q}=\GL_{n_1}(\mathbb{Q}_p)\times\cdots\times \mathcal{Q}^0\times\cdots\times \GL_{n_k}(\mathbb{Q}_p)$.

Since $n_{k+1-i}=n_{i}$, for convenience, we write \begin{align*}
	\mathcal{Q}=\GL_{n_1}(\mathbb{Q}_p)\times\cdots\times\GL_{n_{\bar{k}}}(\mathbb{Q}_p) \times\mathcal{Q}^0\times\GL_{n_{\bar{k}}}(\mathbb{Q}_p) \times\cdots\times \GL_{n_1}(\mathbb{Q}_p)
\end{align*}
and its complex dual group:
\begin{align*}
	{{Q}}:=\GL_{n_1}(\mathbb{C})\times\cdots\times\GL_{n_{\bar{k}}}(\mathbb{C}) \times{Q^0}\times\GL_{n_{\bar{k}}}(\mathbb{C}) \times\cdots\times \GL_{n_1}(\mathbb{C})
\end{align*} 
Define $\mathcal{Q}^+:=\mathcal{Q}\rtimes{\Gamma}$,
$Q^+:=Q\rtimes\hat{\Gamma}$.

We write the representation $\pi_{\mathcal{Q}}:=\St (z_1,n_1) \otimes \cdots \otimes\St (z_{\bar{k}},n_{\bar{k}})\otimes\pi^0\otimes\St (-z_{\bar{k}},n_{\bar{k}})\otimes\cdots\otimes \St( -z_1,n_1)$  of $\mathcal{Q}$, where $\pi^0$ is an irreducible tempered representation  of $\mathcal{Q}^0$.

Using \begin{align*}
	\Lambda_{\mathcal{J}_{\mathcal{Q}}}: \mathrm{Rep}(\mathcal{Q})^{(\mathcal{J}_{\mathcal{Q}},1_{\triv})}  \xrightarrow[]{\simeq} \mathcal{H}(\mathcal{J}_{\mathcal{Q}}\backslash\mathcal{Q}/\mathcal{J}_{\mathcal{Q}})-\Mod
\end{align*}
\begin{align*}
	\mathcal{H}(\mathcal{J}_{\mathcal{Q}}\backslash\mathcal{Q}/\mathcal{J}_{\mathcal{Q}})\simeq \mathcal{H}(Q,p^{1/2})
\end{align*}

and
\begin{equation*}
	\Theta_Q:	\mathbb{H}(Z_{Q}(t_c))-\mathrm{Mod}_{(\mathrm{log}t_h, \mathrm{log}p^{1/2})}\xrightarrow[]{\simeq}\mathcal{H}(Q,v)-\mathrm{Mod}_{(t,p^{1/2})} 
\end{equation*} 
we have $\pi_{\mathcal{Q}}=(\Lambda_{\mathcal{J_Q}})^{-1}\circ\Theta_Q \circ\IM^*M_{y,\bar{t},r_0}$.

In fact, from \cite{aubert2017affine}, the functor $\Theta_Q$ is the composition of
\begin{equation*}
	\Exp:	\mathbb{H}(Z_{Q}(t_c))-\mathrm{Mod}_{(\mathrm{log}t_h, \mathrm{log}p^{1/2})}\xrightarrow[]{\simeq}\mathcal{H}(Z_{Q}(t_c),v)-\mathrm{Mod}_{(t,p^{1/2})} 
\end{equation*} 
and \begin{equation*}
	\Ind:\mathcal{H}(Z_{Q}(t_c),v)-\mathrm{Mod}_{(t,p^{1/2})} \xrightarrow[]{\simeq}\mathcal{H}(Q,v)-\mathrm{Mod}_{(t,p^{1/2})} 
\end{equation*} 
Hence $\Lambda_{\mathcal{J_Q}}(\pi_{\mathcal{Q}})=\pi_{\mathcal{Q}}^{\mathcal{J_Q}}=\Ind_{\mathcal{H}(Z_{Q}(t_c),v)}^{\mathcal{H}(Q,v)}\Exp \IM^*M_{y,\bar{t},r_0}$. 

For convenience, We write \begin{align*}
	\bar{{Q}}:=Z_Q(t)=\GL_{n_1}(\mathbb{C})\times\cdots\times\GL_{n_{\bar{k}}}(\mathbb{C})\times\bar{Q}^0\times\GL_{n_{\bar{k}}}(\mathbb{C})\times\cdots\times\GL_{n_1}(\mathbb{C})
\end{align*} 
where $\bar{Q}^0=\GL_{m_1}(\mathbb{C})\times\GL_{m_2}(\mathbb{C})$

Since each Steinberg representation $\St(z_i,n_i)$ of $\GL_{n_i}(\mathbb{Q}_p)$ corresponds to a one dimensional module $\chi_i$ of $\mathcal{H}(\GL_{n_i}(\mathbb{C}),p^{1/2})$ given by: 
\begin{align*}
	\chi_i(T_{s_\alpha})&=-1  &\alpha\in  \Pi\\
	\chi_i(\theta_y)&=y(s)  &s=(p^{z_i}p^{{\frac{1-n_i}{2}}},p^{z_i}p^{{\frac{3-n_i}{2}}},\cdots,p^{z_i}p^{{\frac{n_i-1}{2}}})
\end{align*}

we can write \begin{align*}
	\pi_{\mathcal{Q}}^{\mathcal{J_Q}}&=\Ind_{\mathcal{H}(\bar{Q},v)}^{\mathcal{H}(Q,v)}\Exp \IM^*M_{y,\bar{t},r_0}\\
	&={\mathcal{H}({Q},v)}\otimes_{\mathcal{H}(\bar{Q},v)}({\chi}_1 \otimes \cdots \otimes {\chi}_{\bar{k}} \otimes 
	\pi_1 \otimes \pi_2 \otimes 
	{\chi}_{\bar{k}}^\vee \otimes \cdots \otimes {\chi}^\vee_1)
\end{align*}

Recall also that  the irreducible representation $M^+(\phi_{\mathcal{Q}},1)$ of $\mathcal{Q}$ is given by  \begin{align*}
	\Lambda^+_{\mathcal{J}_{\mathcal{Q}}}: \mathrm{Rep}(\mathcal{Q}^+)^{(\mathcal{J}_{\mathcal{Q}},1_{\triv})}  \xrightarrow[]{\simeq} \mathcal{H}(\mathcal{J}_{\mathcal{Q}}\backslash\mathcal{Q}^+/\mathcal{J}_{\mathcal{Q}})-\Mod
\end{align*}

\begin{align*}
	\mathcal{H}(\mathcal{J}_{\mathcal{Q}}\backslash\mathcal{Q}^+/\mathcal{J}_{\mathcal{Q}})\simeq \mathcal{H}(Q^+,p^{1/2})\simeq \mathcal{H}(Q,p^{1/2})\rtimes \Gamma
\end{align*}

and
\begin{equation*}
	\Theta^+_Q:	\mathbb{H}(Z_{Q^+}(t_c))-\mathrm{Mod}_{(\mathrm{log}t_h, \mathrm{log}p^{1/2})}\simeq \mathcal{H}(Q^+,v)-\mathrm{Mod}_{(t,p^{1/2})}
\end{equation*} 
which means \begin{align*}
	M^+(\phi_{\mathcal{Q}},1)=(\Lambda_{\mathcal{J_Q}}^+)^{-1}\circ\Theta^+_Q\circ\IM^*M_{y,\bar{t},r_0,1}^{+} 
\end{align*}

Note that $\mathcal{H}(Q^+,v)=\mathcal{H}(Q,v)\rtimes\hat{\Gamma}$ and 
$\mathbb{H}(Z_{Q^+}(t_c))=\mathbb{H}(Z_Q(t_c))\rtimes\mathfrak{R}$. Let ${w}_0\in W $ be a permutation that swaps the first block of size ${m_1}$ with the last block of size ${m_2}$ and let $\bar{w}_0$ be a representative of $w_0$ in $Q$.

The action $\tau\in \mathfrak{R}$ on the Weyl group $W$ of $Z_Q(t_c)$ and on the torus $T$ of $Z_Q(t_c)$ is given by 
\begin{equation}\label{RActM1}
	\begin{aligned}
		\tau(w)&={w}_0\hat{\gamma}(w){w}^{-1}_0\quad w\in W\\
		\tau(s)&=\bar{w}_0\hat{\gamma}(s)\bar{w}^{-1}_0 \quad s\in T
	\end{aligned}
\end{equation}

Hence we have the $\mathfrak{R}$ action on $\mathbb{H}(Z_Q(t_c))$ induced by \eqref{RActM1}.

From \cite{aubert2017affine}, the functor $\Theta_Q^+$ is the composition of \begin{equation*}
	\Exp:	\mathbb{H}(Z_{Q^+}(t_c))-\mathrm{Mod}_{(\mathrm{log}t_h, \mathrm{log}p^{1/2})}\xrightarrow[]{\simeq}\mathcal{H}(Z_{Q^+}(t_c),v)-\mathrm{Mod}_{(t,p^{1/2})} 
\end{equation*} 
and \begin{equation*}
	\Ind:\mathcal{H}(Z_{Q^+}(t_c),v)-\mathrm{Mod}_{(t,p^{1/2})} \xrightarrow[]{\simeq}\mathcal{H}(Q^+,v)-\mathrm{Mod}_{(t,p^{1/2})} 
\end{equation*}

We also have $\mathcal{H}(Z_{Q^+}(t_c),v)=\mathcal{H}(Z_Q(t_c),v)\rtimes\mathfrak{R}$, where the $\mathfrak{R}$ action is induced by \eqref{RActM1}.

In the connected case, $\mathcal{H}(Z_{Q}(t_c),v)$ is a subalgebra of $\mathcal{H}({Q},v)$, so the $\operatorname{Ind}$ functor is given by tensor product. However,  $\mathcal{H}(Z_{Q^+}(t_c),v)=\mathcal{H}(Z_Q(t_c),v)\rtimes\mathfrak{R}$ is not a subalgebra of $\mathcal{H}(Q^+,v)=\mathcal{H}(Q,v)\rtimes\hat{\Gamma}$. Fortunately, we still have an algebra embedding when we consider their completion (Proposition \ref{ComCatEq1}):
\begin{align*}
	\widehat{\mathcal{H}}(Z_{Q^+}(t_c),v)\simeq {e_{\varpi}\widehat{\mathcal{H}}(Z_{Q^+}(t_c),v)}\hookrightarrow \widehat{\mathcal{H}}(Q^+,v)
\end{align*}

where \begin{align*}
	\tau\mapsto e_{\varpi}(\iota_{{w_0}}^0,\hat{\gamma})
\end{align*}
and $\iota_{{w}_0}^0$ is the intertwining element of ${w}_0$.

Thus for every $(\pi,V)\in \mathcal{H}(Z_{Q^+}(t_c),v)-\mathrm{Mod}_{(t,p^{1/2})}$, we can write the induction as:\begin{align*}
	\Ind(\pi)=\widehat{\mathcal{H}}(Q^+,v)\otimes_{e_{\varpi}\widehat{\mathcal{H}}(Z_{Q^+}(t_c),v)}V
\end{align*}

where ${e_{\varpi}\widehat{\mathcal{H}}(Z_{Q^+}(t_c),v)}$ acts on $V$ via  ${e_{\varpi}\widehat{\mathcal{H}}(Z_{Q^+}(t_c),v)}\simeq  \widehat{\mathcal{H}}(Z_{Q^+}(t_c),v)$.

Consequently, we can write \begin{align*}
	\Lambda_{\mathcal{J_Q}}^+ (M^+(\phi_{\mathcal{Q}},1))&=M^+(\phi_{\mathcal{Q}},1)^{\mathcal{J_Q}}\\
	&=\Ind(\Exp\IM^*M_{y,\bar{t},r_0,1}^{+})
\end{align*}

 From Proposition \ref{GeoActM}, we get $M_{y,\bar{t},r_0,1}^{+}\simeq M_{y,\bar{t},r_0}\otimes\mathbb{C}_\triv$ with  $J_\tau$. Now we consider the geometric intertwining $J_\tau$ and translate it to the $p$-adic side.

Recall $t=\lambda_\phi(\mathrm{Fr})$, $\bar{Q}=Z_Q(t_c)=\GL_{n_1}(\mathbb{C})\times\cdots\times\GL_{n_{\bar{k}}}(\mathbb{C})\times\GL_{m_1}(\mathbb{C})\times\GL_{m_2}(\mathbb{C})\times\GL_{n_{\bar{k}}}(\mathbb{C})\times\cdots\times\GL_{n_1}$ and  $t'=\phi(\mathrm{Fr})$.

Denote $\bar{t}=-\mathrm{log}|t'|+d_\nu\begin{pmatrix}
	r_0 &0\\
	0& -r_0
\end{pmatrix}=(\bar{t}_1,\cdots,\bar{t}_{\bar{k}},\bar{t}'_1,\bar{t}'_2,{\bar{t}}_{\bar{k}}^\vee,\cdots,{\bar{t}}_{1}^\vee)$,
$y=d_\phi\begin{pmatrix}
	0 &1\\
	0& 0
\end{pmatrix}=(y_1,\cdots,y_{\bar{k}},y'_1,y'_2,y_{\bar{k}},\cdots,y_1)$, where
$y_i=d_{\nu_{n_i}}\begin{pmatrix}
	0 &1\\
	0& 0
\end{pmatrix}$, each $y_i$ is regular unipotent element of $\GL_{n_i}(\mathbb{C})$, $\GL_{m_1}(\mathbb{C})$ and $\GL_{m_2}(\mathbb{C})$.

Define \begin{align*}
	y^*=(y_1,\cdots,y_{\bar{k}},0,0,y_{\bar{k}},\cdots,y_1)
\end{align*}

\begin{lemma}
	The irreducible representation $M_{y,\bar{t},r_0,1}^+$ is the unique submodule of $E_{y^*,\bar{t},r_0,1}^+$.
\end{lemma}
\begin{proof}
	By \eqref{GeoMul}, we have $m(M_{y,\bar{t},r_0,1}^+,E_{y^*,\bar{t},r_0,1}^+)=\sum_k\mathrm{dim}H^k(i_{y^*}^!IC_{\phi^+}))^1$. Since the closure of the orbit of $y$ is a nonsingular variety, we get $m(M_{y,\bar{t},r_0,1}^+,E_{y^*,\bar{t},r_0,1}^+)=1$. Because the irreducible representation $M_{y,\bar{t},r_0,1}$ is the unique submodule of $E_{y^*,\bar{t},r_0,1}$, the lemma follows.
\end{proof}

Recall $\bar{t}=(\bar{t}_1,\cdots,\bar{t}_{\bar{k}},\bar{t}'_1,\bar{t}'_2,{\bar{t}}_{\bar{k}}^\vee,\cdots,{\bar{t}}_{1}^\vee)$. Since $\bar{t}'_1,\bar{t}'_2$ are diagonal matrices,
let us reorder their diagonal entries so that
 \begin{align*}
	\bar{t}’^{*}_1&=(\bar{t}’^{*}_1(1),\cdots,\bar{t}’^{*}_1({m_1}))\\
	\bar{t}’^{*}_2&=(\bar{t}’^{*}_2(1),\cdots \bar{t}’^{*}_2({m_2}))
\end{align*}
satisfy $\bar{t}’^{*}_1(i)\leq \bar{t}’^{*}_1(j)\quad(i\leq j)$ and $\bar{t}’^{*}_2(i)\leq\bar{t}’^{*}_2(j)\quad(i\leq j)$ 

Set \begin{align*}
	\bar{t}^*=(\bar{t}_1,\cdots,\bar{t}_{\bar{k}},\bar{t}'^{*}_1,\bar{t}'^{*}_2,{\bar{t}}_{\bar{k}}^\vee,\cdots,{\bar{t}}_{1}^\vee)
\end{align*}

Then $\tau(\bar{t}^*)=\bar{t}^*$ ($\tau\in\mathfrak{R}$) and  there is an   $\mathbb{H}(Z_{Q^+}(t_c))=\mathbb{H}(Z_Q(t_c))\rtimes\mathfrak{R}$-module isomorphism $E_{y^*,\bar{t},r_0,\rho}^+\simeq E_{y^*,\bar{t}^*,r_0,\rho}^+$.

Recall $\bar{{Q}}=Z_Q(t_c)=\GL_{n_1}(\mathbb{C})\times\cdots\times\GL_{n_{\bar{k}}}(\mathbb{C})\times\GL_{m_1}(\mathbb{C})\times\GL_{m_2}(\mathbb{C})\times\GL_{n_{\bar{k}}}(\mathbb{C})\times\cdots\times\GL_{n_1}$, then $Z_{Q^+}(t_c)$ is generated by $\bar{{Q}}=Z_Q(t_c)$ and $(\bar{w}_0,\hat{\gamma})$.
Let \begin{align*}
	\bar{{K}}=\GL_{n_1}(\mathbb{C})\times\cdots\times\GL_{n_{\bar{k}}}(\mathbb{C})\times_{m_1}\GL_{1}(\mathbb{C})\times_{m_2}\GL_{1}(\mathbb{C})\times\GL_{n_{\bar{k}}}(\mathbb{C})\times\cdots\times\GL_{n_1}
\end{align*}
and $\bar{{K}}^+$ be the group generated by $\bar{{K}}$ and $(\bar{w}_0,\hat{\gamma})$. We have $\mathbb{H}(\bar{{K}}^+)=\mathbb{H}(\bar{{K}})\rtimes\mathfrak{R}$.

From Theorem \ref{Ind31}, we have  $E_{y^*,\bar{t}^*,r_0,\rho}^+\simeq \mathbb{H}(\bar{Q})\rtimes\mathfrak{R}\otimes_{\mathbb{H}(\bar{K})\rtimes\mathfrak{R}}E_{y^*,\bar{t}^*,r_0,\rho}^{\bar{K}^+,+}$.

Note that $E_{y^*,\bar{t}^*,r_0}^{\bar{K}}=\mathbb{C}_a\otimes_{H_{M_{y^*}^0}}H^{M_y^0}_*(\mathcal{P}_{y^*}^{\bar{K}},{\dot{\mathcal{L}}})\simeq H_*^{\{1\}}(\mathcal{P}_{y^*}^{\bar{K},a},{\dot{\mathcal{L}}})$
where ${\dot{\mathcal{L}}}$ is the constant sheaf, and $\mathcal{P}_{y^*}^{\bar{K},a}$ is a point. 

Hence
\begin{align*}
	E_{y^*,\bar{t}^*,r_0}^{\bar{K}} = 
	\bar{\chi}_1 \otimes \cdots \otimes \bar{\chi}_{\bar{k}} \otimes 
	\bar{\chi}'_1 \otimes \bar{\chi}'_2 \otimes 
	\bar{\chi}_{\bar{k}}^\vee \otimes \cdots \otimes \bar{\chi}^\vee_1
\end{align*}

From \cite[Lemma 3.14]{Aubert2016GradedHA},  the $\tau$ action on 
$E_{y,\bar{t},r_0}^{\bar{K}}$ is given by \begin{align}\label{KAct1}
	I_\tau(a_1,\cdots ,a_{\bar{k}},b_1,b_2,a_{\bar{k}}^\vee,\cdots,a^\vee_1)=(a^\vee_1,\cdots,a_{\bar{k}}^\vee,b_1,b_2,a_{\bar{k}},\cdots,a_1)
\end{align}
and we  have the $\mathcal{H}(\bar{K}^+,v)=\mathcal{H}(\bar{K},v)\rtimes\mathfrak{R}$-module:\begin{align*}
	\Exp\IM^* E_{y^*,\bar{t}^*,r_0}^{\bar{K}}={\chi}_1 \otimes \cdots \otimes {\chi}_{\bar{k}} \otimes 
	{\chi}'_1 \otimes \chi'_2 \otimes 
	{\chi}_{\bar{k}}^\vee \otimes \cdots \otimes {\chi}^\vee_1
\end{align*}

\begin{proposition}\label{ProThWh1}
	We have $\bar{{Q}}=\GL_{n_1}(\mathbb{C})\times\cdots\times\GL_{n_{\bar{k}}}(\mathbb{C})\times\bar{Q}^0\times\GL_{n_{\bar{k}}}(\mathbb{C})\times\cdots\times\GL_{n_1}$, $\bar{{K}}=\GL_{n_1}(\mathbb{C})\times\cdots\times\GL_{n_{\bar{k}}}(\mathbb{C})\times_{m_1}\GL_{1}(\mathbb{C})\times_{m_2}\GL_{1}(\mathbb{C})\times\GL_{n_{\bar{k}}}(\mathbb{C})\times\cdots\times\GL_{n_1}$.

	We have a $\mathcal{H}(\bar{Q},v)$-module: $\Exp \IM^*M_{y,\bar{t},r_0}={\chi}_1 \otimes \cdots \otimes {\chi}_{\bar{k}} \otimes 
	\pi_1 \otimes \pi_2 \otimes 
	{\chi}_{\bar{k}}^\vee \otimes \cdots \otimes {\chi}^\vee_1$

	Denote $\bar{T}=\times_{m_1}\GL_{1}(\mathbb{C})\times_{m_2}\GL_{1}(\mathbb{C})$. Then $\mathcal{H}(\bar{Q}^0,v)$ and  $\mathcal{H}(\bar{T},v)$  are stable under the action $\tau$.
	
	Define an $I_\tau^0$ action on $\mathcal{H}(\bar{Q}^0,v)\otimes_{\mathcal{H}(\bar{T},v)}({{\chi}'_1} \otimes {\chi'_2})$ by $I_\tau^0(h,a)=\tau(h)\otimes a$. The irreducible module  $\pi_1\otimes\pi_2$  is the unique submodule of $\mathcal{H}(\bar{Q}^0,v)\otimes_{\mathcal{H}(\bar{T},v)}({{\chi}'_1} \otimes {\chi'_2})$. Then we give an action $I^0_\tau$  on $\pi_1\otimes\pi_2$ by restriction from $\mathcal{H}(\bar{Q}^0,v)\otimes_{\mathcal{H}(\bar{T},v)}({{\chi}'_1} \otimes {\chi'_2})$.
	
	We give an $I_\tau$ action on  $\Exp \IM^*M_{y,\bar{t},r_0}={\chi}_1 \otimes \cdots \otimes {\chi}_{\bar{k}} \otimes 
	\pi_1 \otimes \pi_2 \otimes 
	{\chi}_{\bar{k}}^\vee \otimes \cdots \otimes {\chi}^\vee_1$ by:
	\begin{align*}
		I_\tau(a_1,\cdots ,a_{\bar{k}},b_1,b_2,a_{\bar{k}}^\vee,\cdots,a^\vee_1)=(a^\vee_1,\cdots,a_{\bar{k}}^\vee,I^0_\tau(b_1,b_2),a_{\bar{k}},\cdots,a_1)
	\end{align*}
	
	Then using this $I_\tau$ action we have  $\Exp \IM^*M^+_{y,\bar{t},r_0,1}\simeq \Exp \IM^*M_{y,\bar{t},r_0}\otimes \mathbb{C}_\triv$ as $\mathcal{H}(\bar{Q},v)\rtimes\mathfrak{R}$-modules.
\end{proposition}

\begin{proof}
	We have \begin{align*}
		\Exp \IM^*E^+_{y^*,\bar{t}^*,r_0,1}&\simeq\Exp \IM^*E^+_{y^*,\bar{t}^*,r_0,1}\\
		&\simeq\Exp \IM^*(\mathbb{H}(\bar{Q})\rtimes\mathfrak{R}\otimes_{\mathbb{H}(\bar{K})\rtimes\mathfrak{R}}E_{y^*,\bar{t}^*,r_0,\rho}^{\bar{K}^+,+})\\
		&\simeq \mathcal{H}(\bar{Q},v)\rtimes\mathfrak{R}\otimes_{\mathcal{H}(\bar{K},v)\rtimes\mathfrak{R}}(\Exp\IM^*E_{y^*,\bar{t}^*,r_0,\rho}^{\bar{K}^+,+})
	\end{align*}
	Use the same argument, we also have \begin{align}
		\Exp \IM^*E_{y^*,\bar{t}^*,r_0}&\simeq\mathcal{H}(\bar{Q},v)\otimes_{\mathcal{H}(\bar{K},v)}(\Exp\IM^*E_{y^*,\bar{t}^*,r_0}^{\bar{K}})\\
		&\simeq {\chi}_1 \otimes \cdots \otimes {\chi}_{\bar{k}} \otimes 
		\mathcal{H}(\bar{Q}^0,v)\otimes_{\mathcal{H}(\bar{T},v)}({{\chi}'_1} \otimes {\chi'_2}) \otimes 
		{\chi}_{\bar{k}}^\vee \otimes \cdots \otimes {\chi}^\vee_1\label{SSInd}
	\end{align} 
	
	Since $\tau(\bar{t}^*)=\bar{t}^*$,  the completion of $\mathcal{H}(\bar{K},v)\rtimes\mathfrak{R}$ with the corresponding central character is  $\widehat{\mathcal{H}}(\bar{K},v)\rtimes\mathfrak{R}$ and the completion of $\mathbb{H}(\bar{K})\rtimes\mathfrak{R}$ with the corresponding central character is  $\widehat{\mathbb{H}}(\bar{K})\rtimes\mathfrak{R}$.
	
	Using  $\widehat{\mathcal{H}}(\bar{K},v)\rtimes\mathfrak{R}\simeq\widehat{\mathbb{H}}(\bar{K})\rtimes\mathfrak{R}$ and \eqref{KAct1}, we have \begin{align*}
		\Exp\IM^*E_{y^*,\bar{t}^*,r_0,1}^{\bar{K}^+,+}\simeq \Exp\IM^*E_{y^*,\bar{t}^*,r_0,1}^{\bar{K}}\otimes\mathbb{C}_\triv
	\end{align*}
	
	where  $I_\tau^{\bar{K}}$ is given by \begin{align*}
		I_\tau^{\bar{K}}(a_1,\cdots ,a_{\bar{k}},b_1,b_2,a_{\bar{k}}^\vee,\cdots,a^\vee_1)=(a^\vee_1,\cdots,a_{\bar{k}}^\vee,b_1,b_2,a_{\bar{k}},\cdots,a_1)
	\end{align*}
	on  $\Exp\IM^*E_{y^*,\bar{t}^*,r_0,1}^{\bar{K}}= 
	{\chi}_1 \otimes \cdots \otimes {\chi}_{\bar{k}} \otimes 
	{\chi}'_1 \otimes \chi'_2 \otimes 
	{\chi}_{\bar{k}}^\vee \otimes \cdots \otimes {\chi}^\vee_1$
	
	Using  the isomorphism of $\mathcal{H}(\bar{Q},v)$-modules \begin{align*}
		\mathcal{H}(\bar{Q},v)\rtimes\mathfrak{R}\otimes_{\mathcal{H}(\bar{K},v)\rtimes\mathfrak{R}}(\Exp\IM^*E_{y^*,\bar{t}^*,r_0,1}^{\bar{K}}\otimes\mathbb{C}_\triv)\simeq \mathcal{H}(\bar{Q},v)\otimes_{\mathcal{H}(\bar{K},v)}(\Exp\IM^*E_{y^*,\bar{t}^*,r_0}^{\bar{K}})
	\end{align*}
	we can get the $I_\tau$ action on $\mathcal{H}(\bar{Q},v)\otimes_{\mathcal{H}(\bar{K},v)}(\Exp\IM^*E_{y^*,\bar{t}^*,r_0}^{\bar{K}})$ given by $I_\tau(h\otimes v):=\tau(h)\otimes I_\tau^{\bar{K}}(v)$.
	
	Using the expression \eqref{SSInd}, we get the $I_\tau$ action on ${\chi}_1 \otimes \cdots \otimes {\chi}_{\bar{k}} \otimes 
	\mathcal{H}(\bar{Q}^0,v)\otimes_{\mathcal{H}(\bar{T},v)}({{\chi}'_1} \otimes {\chi'_2}) \otimes 
	{\chi}_{\bar{k}}^\vee \otimes \cdots \otimes {\chi}^\vee_1$ as:\begin{align*}
		I_\tau(a_1,\cdots ,a_{\bar{k}},h\otimes (b_1,b_2),a_{\bar{k}}^\vee,\cdots,a^\vee_1)=(a^\vee_1,\cdots,a_{\bar{k}}^\vee,\tau(h)\otimes (b_1,b_2),a_{\bar{k}},\cdots,a_1)
	\end{align*}
	The irreducible module $\Exp \IM^*M^+_{y,\bar{t},r_0,1}$  is the unique submodule of $\Exp \IM^*E^+_{y^*,\bar{t},r_0,1}\simeq \Exp \IM^*E^+_{y^*,\bar{t}^*,r_0,1}$. Hence if we give the $I_\tau$ action on  $\Exp \IM^*M_{y,\bar{t},r_0}={\chi}_1 \otimes \cdots \otimes {\chi}_{\bar{k}} \otimes 
	\pi_1 \otimes \pi_2 \otimes 
	{\chi}_{\bar{k}}^\vee \otimes \cdots \otimes {\chi}^\vee_1$ by restricted from the former, we obtain the proposition.
\end{proof}

\begin{theorem}\label{MainTh}
	There is an isomorphism of $\mathcal{G}^+$ representations \begin{equation*}
		(M(\phi,1))^+_W\simeq M^+(\phi,1)
	\end{equation*}    
\end{theorem}
\begin{proof}
We have
	\begin{align}\label{ExpThe1}
		&{\mathcal{H}}(Q,v)\otimes_{{\mathcal{H}}(Z_{Q}(t),v)}\Exp \IM^*M_{y,\bar{t},r_0}\\
		&={\chi}_1 \otimes \cdots \otimes {\chi}_{\bar{k}} \otimes 
		{\mathcal{H}}(Q^0,v)\otimes_{{\mathcal{H}}(\bar{Q}^0,v)}(\pi_1 \otimes \pi_2)\otimes 
		{\chi}_{\bar{k}}^\vee \otimes \cdots \otimes {\chi}^\vee_1
	\end{align}
	Using 
	\begin{align*}
		\widehat{\mathcal{H}}(Q^+,v)\otimes_{e_{\varpi}\widehat{\mathcal{H}}(Z_{Q^+}(t),v)}\Exp \IM^*M_{y,\bar{t},r_0,1}^+&\simeq  \widehat{\mathcal{H}}(Q,v)\otimes_{e_{\varpi}\widehat{\mathcal{H}}(Z_{Q}(t),v)}\Exp \IM^*M_{y,\bar{t},r_0,1}^+\\
		&\simeq{\mathcal{H}}(Q,v)\otimes_{{\mathcal{H}}(Z_{Q}(t),v)}\Exp \IM^*M_{y,\bar{t},r_0,1}^+
	\end{align*}
	we have \begin{align*}
		\hat{\gamma}\circ (g\otimes v)&= \hat{\gamma}(g)\iota_{w^{-1}}^0\otimes I_\tau(v)
	\end{align*}
	where the $I_\tau$ action  comes from Proposition \ref{ProThWh1}.
	Since $\iota_{w^{-1}}^0\in_{\mathcal{F}}{\mathcal{H}}(\bar{Q}^0,v)$, using the expression \eqref{ExpThe1}, we get the $I_{\hat{\gamma}}$ action on $ {\mathcal{H}}(Q,v)\otimes_{{\mathcal{H}}(Z_{Q}(t),v)}\Exp \IM^*M_{y,\bar{t},r_0}$:\begin{align*}
		I_{\hat{\gamma}}(a_1,\cdots ,a_{\bar{k}},h\otimes (b_1,b_2),a_{\bar{k}}^\vee,\cdots,a^\vee_1)=(a^\vee_1,\cdots,a_{\bar{k}}^\vee,\hat{\gamma}(h)\iota_{w^{-1}}^0\otimes I_\tau^0(b_1,b_2),a_{\bar{k}},\cdots,a_1)
	\end{align*}
	
	By Theorem \ref{MainWhTem}, the middle action $h\otimes (b_1,b_2)\mapsto\hat{\gamma}(h)\iota_{w^{-1}}^0\otimes I_\tau^0(b_1,b_2)$ corresponds to a Whittaker normalized  action $I_{W,\pi^0}$ on $\pi^0$. Hence $I_{\hat{\gamma}}$ corresponds to the action \eqref{LeviAct} on $\pi_{\mathcal{Q}}$. Finally we get \begin{align*}
		\pi_{\mathcal{Q}}^+\simeq \widehat{\mathcal{H}}(Q^+,v)\otimes_{e_{\varpi}\widehat{\mathcal{H}}(Z_{Q^+}(t),v)}\Exp \IM^*M_{y,\bar{t},r_0,1}^+=M^+(\phi_{\mathcal{Q}},1) 
	\end{align*}
	
	From Proposition \ref{WhiInd}, we have $ (E(\phi,1))^+_W\simeq i_{\mathcal{P}^+}\pi_{\mathcal{Q}}^+\simeq E^+(\phi,1)= i_{\mathcal{P}^+} M^+(\phi_{\mathcal{Q}},1)$. Using the definition, we conclude $(M(\phi,1))^+_W\simeq M^+(\phi,1)$.
\end{proof}

\bibliographystyle{alpha}
	\bibliography{ref1}
\email{Email address: chaiy20@mails.tsinghua.edu.cn}

\end{document}